\documentclass[11pt,reqno]{amsart}

\usepackage{amsmath}
\usepackage{amssymb}
\usepackage{mathtools}
\usepackage{xcolor}
\usepackage{amsthm}
\usepackage[letterpaper,textwidth=6.25in,textheight=8.75in]{geometry}
\usepackage[hidelinks]{hyperref}

\numberwithin{equation}{section}
\allowdisplaybreaks

\newcommand{\R}{\mathbb{R}}
\newcommand{\C}{\mathbb{C}}
\newcommand{\Z}{\mathbb{Z}}
\newcommand{\E}{\mathbb{E}}
\newcommand{\Prob}{\mathbb{P}}
\newcommand{\one}{\mathbf{1}}
\newcommand{\ind}{\mathbf{1}}
\newcommand{\norm}[1]{\lVert #1\rVert}
\newcommand{\inner}[2]{\langle #1,#2\rangle}

\makeatletter
\def\author@andify{%
  \nxandlist
    {\unskip,\penalty-1\space\ignorespaces}%
    {\unskip{} \@@and\space}%
    {\unskip\penalty-2\space\@@and\space}%
}
\makeatother

\DeclareMathOperator{\Tr}{Tr}
\DeclareMathOperator{\diag}{diag}
\DeclareMathOperator{\supp}{supp}
\DeclareMathOperator{\range}{range}
\DeclareMathOperator{\arsinh}{arsinh}

\theoremstyle{plain}
\newtheorem{thm}{Theorem}[section]
\newtheorem{prop}[thm]{Proposition}
\newtheorem{lem}[thm]{Lemma}
\newtheorem{cor}[thm]{Corollary}

\theoremstyle{definition}
\newtheorem{defn}[thm]{Definition}
\newtheorem{setup}[thm]{Setup}
\newtheorem{notation}[thm]{Notation}

\theoremstyle{remark}

\title[Sharp theta asymptotics]{Sharp Asymptotics for the Lov\'asz Number of Random Circulant Graphs}

\author{Xin Chen, Hongwei Hu, Yanqing Hu, Wei Wang and Luzhen Yang}

\thanks{Emails: Xin Chen, \texttt{chenx8@sustech.edu.cn};
Hongwei Hu, \texttt{12531170@mail.sustech.edu.cn};
Yanqing Hu, \texttt{huyq@sustech.edu.cn};
Wei Wang, \texttt{12632272@mail.sustech.edu.cn};
Luzhen Yang, \texttt{12631208@mail.sustech.edu.cn}.}

\thanks{All authors are with the Department of Statistics and Data Science,
Southern University of Science and Technology, Shenzhen, China.}

\date{}

\newcommand{\1}{\mathbf{1}}
\newcommand{\ezero}{\mathbf{e}_0}
\newcommand{\abs}[1]{\lvert #1\rvert}

\newcommand{\symdiff}{\mathbin{\triangle}}

\newcommand{\aut}{\operatorname{aut}}
\newcommand{\Var}{\operatorname{Var}}
\newcommand{\Cov}{\operatorname{Cov}}
\newcommand{\rank}{\operatorname{rank}}
\newcommand{\im}{\operatorname{im}}

\newcommand{\cG}{\mathcal{G}}
\newcommand{\cT}{\mathcal{T}}
\newcommand{\He}{\operatorname{He}}
\newcommand{\coef}{\operatorname{coef}}
\newcommand{\Id}{\mathrm{I}}
\newcommand{\tree}{\mathrm{tree}}

\begin{document}

\begin{abstract}
We study the Lov\'asz number of dense random circulant graphs generated by independent Bernoulli bits indexed by cyclic distances. Although only $O(n)$ independent random inputs determine $O(n^2)$ globally dependent edge indicators, we prove sharp expectation asymptotics along all odd orders, including composite orders, with leading constant one. This resolves a conjecture posed for random circulant graphs in 2025. The proof combines an exact Fourier reduction and complementary duality with a positive interior smoothing that reveals a nonlinear random fixed-point equation. To analyze this equation under the strong dependence of the distance-bit model, we construct canonical tree fields, establish growing-degree Gaussian approximations through polynomial comparison and Wick/Hermite expansions, and transfer the resulting approximation back to the original sign model. A deterministic residual estimate together with an explicit exceptional-event integration yields the sharp expectation asymptotics.
\end{abstract}

\maketitle

\section{Introduction}\label{sec:introduction}

The Lov\'asz theta number is a fundamental semidefinite relaxation of graph parameters such as the independence and chromatic numbers. Its asymptotic behavior on random graphs connects probabilistic combinatorics with harmonic analysis and semidefinite optimization. A particularly revealing setting is provided by random circulant graphs, where strong algebraic structure creates a dependence pattern that is very different from the independent-edge models in which random graph asymptotics are often studied.

\medskip
\noindent\textbf{Random circulant graphs and strongly dependent randomness.}
A circulant graph is invariant under cyclic translations. In the random model considered here, randomness is assigned independently to cyclic distances rather than to individual edges: all pairs of vertices at the same cyclic distance share the same edge indicator. Thus, for odd order \(n\), the graph is generated by only
\[
 \frac{n-1}{2}
\]
independent Bernoulli variables, although there are \(n(n-1)/2\) possible edges. In other words, only \(O(n)\) independent random inputs generate \(O(n^2)\) globally dependent edges. This creates a dependence structure fundamentally different from
that of the Erd\H{o}s--R\'enyi model and is the main probabilistic feature of the problem. More precisely, for odd $n$, let $V_n:=\mathbb Z_n$ be the vertex set, and let $E_n$ be the set of unordered pairs $\{i,j\}$ of distinct vertices satisfying
\[
 \{i,j\}\in E_n
 \quad\Longleftrightarrow\quad
 \xi_{d_n(i,j)}=1,
 \qquad
 d_n(i,j):=\min\{|i-j|,n-|i-j|\},
\]
where $\xi_1,\ldots,\xi_{(n-1)/2}$ are independent $\operatorname{Bernoulli}(1/2)$ random variables. We then write $G_n:=(V_n,E_n)$ for this dense distance-bit random circulant graph throughout the paper.

\medskip
\noindent\textbf{The sharp open problem.}
The asymptotic behavior of the Lov\'asz number in this distance-bit model was recently investigated by Bandeira, B\l{}asiok, Dmitriev, Faure, Kireeva, and Kunisky \cite{BBDDFKK25}. Exploiting the Fourier diagonalization of circulant matrices, they reduced the theta semidefinite program to a random Fourier optimization problem and proved
\[
 \sqrt n
 \leq
 \E\theta(G_n)
 \leq
 C\sqrt{n\log\log n}.
\]
They conjectured that the logarithmic loss is not intrinsic and that the sharp asymptotic is
\[
 \E\theta(G_n)=(1+o(1))\sqrt n.
\]
This conjecture was subsequently recorded as Conjecture~18 in the 2025 open-problem collection of Randomstrasse101 \cite{BDKNR26}. Thus the central unresolved question was not merely to improve an upper bound, but to determine whether the \(\sqrt n\) law is the intrinsic scale of the distance-bit model or an artifact of the available Fourier estimates. The present paper resolves this conjecture. In particular, we remove the \(\sqrt{\log\log n}\) loss and identify the sharp leading constant.

\medskip
\noindent\textbf{Why the dependent model is genuinely different.}
The same leading-order scale has very recently been established in the classical independent-edge model. Potechin and Xu \cite{PX26} proved that
\[
 \vartheta(G(n,1/2))=(1+o_n(1))\sqrt n
\]
with high probability. This result makes the distance-bit problem especially natural: the two models have the same predicted leading order, but their sources of randomness are radically different. In \(G(n,1/2)\), the edges are independent; in the distance-bit circulant model, a single random bit controls all edges at a given distance. Our theorem establishes the same sharp leading-order asymptotic for a strongly dependent, algebraically structured random graph model. More broadly, it shows that the leading $\sqrt n$ scale persists under a substantial change in the underlying dependence structure, while the proof requires a different probabilistic analysis.

\medskip
\noindent\textbf{Main result.}
We prove the conjectured sharp asymptotic on the original distance-bit probability space, without restricting the order to primes. In particular, all odd composite orders are included:
\begin{thm}[Sharp asymptotics]\label{thm:main}
For the random graph defined in Definition~\ref{defn:model},
\begin{equation}\label{eq:main-limit}
 \lim_{\substack{n\to\infty\\ n\ {\rm odd}}}
 \frac{\E\theta(G_n)}{\sqrt n}=1.
\end{equation}
Equivalently,
\[
 \E\theta(G_n)=(1+o(1))\sqrt n
\]
along all odd integers.
\end{thm}

The constant one is forced by an exact deterministic complementary identity. If \(G_D\) is a circulant graph with difference set \(D\), and \(G_N\) is its complement, then
\[
 \theta(G_D)\,\theta(G_N)=n.
\]
Since the fair distance-bit model is invariant in law under complementation, this gives the lower bound \(\E\theta(G_n)\geq\sqrt n\). The upper bound must therefore match a complementary lower bound to first order. This formulation is useful because it makes clear what has to be proved: one must construct a sharp upper certificate despite the strong dependence among the edge indicators.

The result has two complementary aspects. First, it settles the sharp asymptotic for the random circulant model and removes the logarithmic loss in the previously known upper bound. Second, the proof develops a probabilistic approximation scheme for the strongly dependent Fourier fields generated by the distance-bit structure. The latter is specific to the structured dependence encountered here: the dependence is organized through canonical tree fields, which can be compared at growing degree to Gaussian polynomial fields and ultimately linked to a nonlinear Gaussian fixed-point problem. We do not claim a general universality theorem for other dependent random graph models; rather, the point is that this mechanism provides a way to recover sharp asymptotics in a model where the usual independent-edge arguments are unavailable.

\medskip
\noindent\textbf{Proof strategy and main ideas.}
The proof combines three main ideas. The first gives an exact Fourier reduction of the theta problem; the second uses positive smoothing to expose a nonlinear fixed-point structure; and the third develops a tree-based Gaussian approximation that captures the dependence induced by the distance-bit representation.

\medskip
\noindent\textbf{Idea I: exact Fourier duality.}
Translation averaging reduces the theta semidefinite program to a finite-dimensional Fourier optimization with prescribed Fourier zeros. This reduction is exact and deterministic. The complementary Fourier formulations expose the product identity above and provide the sharp normalization against which the probabilistic upper bound is measured.

\medskip
\noindent\textbf{Idea II: smoothing reveals a nonlinear fixed point.}
The previous upper bound is based on uniform control of random Fourier quantities. To remove its logarithmic loss, we need a more precise description of the random Fourier field. We introduce a strictly positive interior smoothing of the complementary Fourier certificates. Prescribing positive coordinate products yields a unique smoothed pair and exposes a nonlinear random fixed-point equation for the Fourier variables. A deterministic residual estimate converts an approximate solution of this equation into an explicit logarithmic certificate for the original theta value. Thus the problem becomes the approximation of a nonlinear random Fourier field, rather than merely the control of its largest coordinate.

\medskip
\noindent\textbf{Idea III: tree structure and Gaussianization under strong dependence.}
The remaining problem is probabilistic. Because only \(O(n)\) independent signs generate the field, its coordinates are strongly dependent and a direct independent-coordinate Gaussian approximation is unavailable. The dependence can nevertheless be organized through canonical tree fields arising from the Fourier kernel. At fixed degree, these tree fields have a Gaussian limit; the growing-degree argument then uses polynomial comparison, Wick grafting, and Hermite approximation to make this Gaussianization uniform over the range needed by the nonlinear fixed point. In this sense, the tree structure is not an auxiliary surrogate imposed on the model: it captures the leading combinatorial organization of the dependence created by shared distance bits.

The probabilistic argument can be summarized by the chain
\[
 \text{distance-bit Fourier field}
 \;\longrightarrow\;
 \text{canonical tree field}
 \;\longrightarrow\;
 \text{Gaussian nonlinear fixed point}.
\]
The growing-degree comparison is essential: fixed-degree convergence alone would not control the nonlinear fixed-point construction.

\medskip
\noindent\textbf{Transfer back to the original probability space.}
The Gaussian construction is used only as an intermediate description of the dependent Fourier field. The final approximants are evaluated on the original distance-bit signs, and the deterministic residual estimate turns their error into a theta certificate. Exceptional events are then integrated directly on the original graph probability space. Consequently, the proof does not require independence between the graph, its approximant, and the exceptional event. The resulting convergence is for the expectation itself, not merely for a Gaussian surrogate or a fixed-degree statistic.

\medskip
\noindent\textbf{Scope and organization.}
The theorem concerns the expectation and does not assert a limiting distribution or a finite-size rate for \(\E\theta(G_n)\). The limit is taken through all odd integers, including composite orders, and the smoothing parameter is fixed before \(n\to\infty\) and sent to zero afterward. Section~\ref{sec:notation} introduces the distance-bit model and Fourier normalization. Sections~\ref{sec:fourier}--\ref{sec:interior} develop the exact Fourier, complementary, and smoothing framework; the subsequent sections establish the Fourier-kernel, tree-field, Gaussian, and transfer estimates; and Section~\ref{sec:expectation} completes the exceptional-event integration and the expectation bound.

\section{Model and Fourier setup}\label{sec:notation}

We now specify the distance-bit model and the theta convention, and then fix the deterministic Fourier normalization used throughout the proof.

\begin{defn}[Distance-bit model and theta normalization]\label{defn:model}
Let \(n\geq3\) be odd, and identify
\[
 V_n=\{0,1,\ldots,n-1\}
\]
with \(\Z_n\). For distinct \(i,j\in V_n\), set
\[
 d_n(i,j)=\min\{|i-j|,n-|i-j|\}.
\]
Let \(\xi_1,\ldots,\xi_{(n-1)/2}\) be independent Bernoulli random variables with parameter \(1/2\). Define \(G_n=(V_n,E_n)\) by
\[
 \{i,j\}\in E_n
 \quad\Longleftrightarrow\quad
 \xi_{d_n(i,j)}=1.
\]
For a graph \(G=(V,E)\), its edge-zero Lov\'asz number is
\begin{equation}\label{eq:theta-sdp}
 \theta(G)=
 \max\left\{
 \sum_{i,j\in V}X_{ij}:\
 \begin{array}{l}
 X\in\R^{V\times V},\quad X=X^{\mathsf T},\quad X\succeq0,\\
 \Tr X=1,\quad X_{ij}=0\ \text{for every }\{i,j\}\in E
 \end{array}
 \right\}.
\end{equation}
\end{defn}

We use the distance-bit model of Definition~\ref{defn:model}, as in \cite{BBDDFKK25}, and fix the deterministic Fourier normalization needed throughout the proof.

\begin{notation}[Vectors and limits]\label{notation:conventions}
All limits in \(n\) are through odd integers. All logarithms are natural. We write \(\norm{\cdot}_2\) for the Euclidean norm and \(\norm{\cdot}_1\) for the sum of absolute coordinate values. Inner products of real vectors are Euclidean. Vector inequalities are coordinatewise. The vectors \(\one\) and \(\ezero\) denote the all-ones vector and the vector equal to one at residue zero and zero elsewhere. A vector has spatial mean zero when its coordinate sum is zero. The function \(\arsinh\) is the inverse hyperbolic sine.
\end{notation}

\begin{setup}[Circulant Fourier normalization]\label{setup:fourier}
Let \(n\geq3\) be odd. Let
\[
 D=-D\subseteq\Z_n\setminus\{0\},
 \qquad
 N=(\Z_n\setminus\{0\})\setminus D.
\]
The graph \(G_D\) has an edge \(\{i,j\}\) precisely when \(j-i\in D\); its complement is \(G_N\). Define the unnormalized Fourier matrix by
\[
 F_{jk}=\exp\left(-\frac{2\pi i jk}{n}\right),
 \qquad j,k\in\Z_n.
\]
The notation \(F^*\) means conjugate transpose. A vector \(v\) is even if \(v_d=v_{-d}\) for every residue \(d\).

Define
\begin{equation}\label{eq:projections}
 J=\frac{\one\one^{\mathsf T}}n,
 \qquad
 H_D=\frac1nF\diag(\ind_D)F^*,
 \qquad
 H_N=\frac1nF\diag(\ind_N)F^*,
 \qquad
 K_0=H_N-H_D.
\end{equation}
For the random graph, take
\[
 D=\{d\neq0:\xi_{\min(d,n-d)}=1\},
 \qquad
 \epsilon_j=1-2\xi_j,
 \qquad
 a_0=0,\quad a_j=a_{n-j}=\epsilon_j.
\]
Thus the \(\epsilon_j\) are independent fair signs, and
\begin{equation}\label{transfer:eq:fourier-operators}
 K_0=\frac1nF\diag(a)F^*.
\end{equation}
The multiplier is negative on edge differences and positive on nonedge differences. Every canonical field and every occurrence of \(K_0\) uses these same signs.
\end{setup}

\section{Fourier optimization}\label{sec:fourier}

In this section, we prove the exact Fourier certificates and the tail estimate used in Section~\ref{sec:expectation}.

\begin{prop}[Fourier formulations of the theta program]\label{prop:fourier}
Let \(n,D,F\) be as in Setup~\ref{setup:fourier}, and put
\[
 b_d=1-2\ind_D(d),
 \qquad
 g=Fb.
\]
The following maxima are attained and equal to \(\theta(G_D)\):
\begin{align}
 \max\biggl\{&
 \sum_d x_d:
 x\in\R^n\ \text{even},\
 x_0=1,\
 x_d=0\ (d\in D),\
 Fx\geq0
 \biggr\},                                                   \label{eq:primal-x}\\
 \max\biggl\{&
 ny_0:
 y\in\R^n\ \text{even},\
 y\geq0,\
 \sum_k y_k=1,\
 (Fy)_d=0\ (d\in D)
 \biggr\},                                                   \label{eq:primal-y}\\
 \max\biggl\{&
 \sum_k y_kg_k:
 y\ \text{satisfies the constraints in \eqref{eq:primal-y}}
 \biggr\}.                                                   \label{eq:primal-g}
\end{align}
The exact changes of variables are
\[
 y=\frac{Fx}{n},
 \qquad
 x=Fy.
\]
Every feasible \(x\) defines a feasible matrix
\[
 X_{ij}=\frac{x_{j-i}}n
\]
in \eqref{eq:theta-sdp}, with eigenvalues \(y_k\). Moreover,
\[
 1\leq\theta(G_D)\leq n.
\]
These assertions include empty and complete graphs.
\end{prop}

\begin{proof}
Set \(\omega=\exp(-2\pi i/n)\). Summing a geometric series gives
\[
 \sum_{j=0}^{n-1}\omega^{jr}
 =
 \begin{cases}
 n,&r=0\pmod n,\\
 0,&r\neq0\pmod n.
 \end{cases}
\]
Consequently,
\begin{equation}\label{eq:Fourier-identities}
 F^*F=nI,
 \qquad
 (F^2v)_d=nv_{-d}.
\end{equation}
Pairing the terms at \(d\) and \(-d\) shows that \(F\) sends real even vectors to real even vectors. On this subspace, \(F^2=nI\).

Let \(X\) be feasible in \eqref{eq:theta-sdp}. Each translate
\[
 X^{(s)}_{ij}=X_{i+s,j+s}
\]
is a permutation conjugate of \(X\). It preserves positive semidefiniteness, trace, edge zeros, and the objective. Hence
\[
 \overline X=\frac1n\sum_{s\in\Z_n}X^{(s)}
\]
is feasible with the same objective and is invariant under simultaneous translation of its indices. With \(x_d=n\overline X_{0d}\), this invariance gives
\[
 \overline X_{ij}=\frac{x_{j-i}}n.
\]
Symmetry gives \(x_d=x_{-d}\), the trace condition gives \(x_0=1\), and the edge conditions give \(x_D=0\).

The Fourier vector with coordinates \(\omega^{ki}\) is an eigenvector of \(\overline X\), with eigenvalue
\[
 \frac1n\sum_d x_d\omega^{kd}=\frac{(Fx)_k}{n}.
\]
The geometric-series identity makes these vectors a basis. Thus positive semidefiniteness is equivalent to \(Fx\geq0\). Each difference occurs in exactly \(n\) ordered pairs, so
\[
 \sum_{i,j}\overline X_{ij}=\sum_d x_d.
\]
Conversely, the matrix defined by any feasible \(x\) in \eqref{eq:primal-x} has all the required properties. This proves \eqref{eq:primal-x}.

For \(y=Fx/n\), \eqref{eq:Fourier-identities} gives
\[
 Fy=x,
 \qquad
 \sum_k y_k=x_0=1,
 \qquad
 \sum_d x_d=ny_0.
\]
These identities give the equivalence with \eqref{eq:primal-y}. Its feasible set is closed in the probability simplex. It contains \(y=\one/n\), whose transform is \(\ezero\). Compactness therefore gives attainment, which transfers to the other formulations and to the semidefinite program. The uniform vector has objective \(1\), and \(y_0\leq1\) gives the upper bound \(n\).

Finally, \(b\) and \(g\) are real and even. Since \(F^{\mathsf T}=F\),
\[
 \sum_k y_kg_k
 =\sum_d(Fy)_db_d
 =\sum_d x_d-2\sum_{d\in D}x_d
 =\sum_d x_d
 =ny_0.
\]
This proves \eqref{eq:primal-g} pointwise on the feasible set. None of the constructions requires an edge or a nonedge to exist.
\end{proof}

\begin{prop}[Exact Fourier duals]\label{prop:dual}
Under Setup~\ref{setup:fourier}, each of the following minima is attained and equals \(\theta(G_D)\):
\begin{align}
 \min\bigl\{&
 \lambda:
 \lambda\in\R,\
 a\in\R^n\ \text{even},\
 \supp(a)\subseteq D,\
 \lambda\one+Fa\geq n\ezero
 \bigr\},                                                    \label{eq:dual-a}\\
 \min\biggl\{&
 1+\sum_k z_k:
 z\in\R^n\ \text{even},\
 z\geq0,\
 (Fz)_d=-1\ (d\in N)
 \biggr\},                                                   \label{eq:dual-z}\\
 \min\bigl\{&
 1+nt_0:
 t\in\R^n\ \text{even},\
 Ft\geq0,\
 t_d=-1/n\ (d\in N)
 \bigr\},                                                    \label{eq:dual-t}\\
 \min\bigl\{&
 \max_k(Fm)_k:
 m\in\R^n\ \text{even},\
 m_0=1,\
 m_d=1\ (d\in N)
 \bigr\}.                                                    \label{eq:dual-m}
\end{align}
The substitutions between \eqref{eq:dual-z} and \eqref{eq:dual-t} are
\[
 t=\frac{Fz}{n},
 \qquad
 z=Ft.
\]
For \(b=1-2\ind_D\) and \(g=Fb\), \eqref{eq:dual-m} also equals
\[
 \min_{\substack{h\in\R^n\ {\rm even}\\ \supp(h)\subseteq D}}
 \max_k(g+Fh)_k.
\]
Thus \(g+Fh\leq L\one\) certifies \(\theta(G_D)\leq L\) for every such \(h\). The constraints in \eqref{eq:dual-z} and \eqref{eq:dual-t} are on nonedges, whereas the primal Fourier zeros are on edges. All formulas include empty and complete graphs.
\end{prop}

\begin{proof}
We first prove dual attainment for the real linear program supplied by Proposition~\ref{prop:fourier}. Let
\[
 D_+=D\cap\{1,\ldots,(n-1)/2\},
 \qquad
 r_d(k)=\cos(2\pi dk/n).
\]
The maximum of \(ny_0\) over
\[
 y\geq0,\qquad \sum_k y_k=1,\qquad r_dy=0\quad(d\in D_+)
\]
is unchanged if evenness is imposed. To see this, replace \(y_k\) by \((y_k+y_{-k})/2\). This preserves each constraint and the objective. For the resulting even vector, the sine terms cancel and \((Fy)_d=r_dy\).

Let \(C\) have first row \(\one^{\mathsf T}\) and remaining rows \(r_d\), and put
\[
 v=(1,0,\ldots,0)^{\mathsf T},
 \qquad c=n\ezero.
\]
The feasible set \(Cy=v,\ y\geq0\) is nonempty and compact. Choose a maximizer \(y^*\), and consider
\[
 \mathcal K=
 \left\{
 C^{\mathsf T}u-s:
 s\geq0,\quad s_k=0\ \text{whenever }y_k^*>0
 \right\},
\]
where \(u\) is unrestricted.

The cone \(\mathcal K\) is generated by the positive and negative rows of \(C\), together with \(-e_k\) at the zero coordinates of \(y^*\). We check its closedness. Any conic combination of finitely many generators can be reduced to a combination of a linearly independent subset. If the generators with positive coefficients are dependent, choose a nonzero relation with at least one positive coefficient. Subtract the largest multiple of that relation that keeps all coefficients nonnegative. At least one coefficient becomes zero. Repeating this operation gives the asserted reduction. A convergent sequence of cone points therefore has a subsequence using the same independent subset. The inverse linear map on its span is continuous, so the coefficients converge to nonnegative limits. The limit belongs to the cone.

Suppose that \(c\notin\mathcal K\). Closedness gives a closest point \(p\in\mathcal K\), by restricting the distance minimization to a sufficiently large closed ball. Set \(w=c-p\). Variation along the segment from \(p\) to \(k\in\mathcal K\) gives
\[
 w^{\mathsf T}(k-p)\leq0.
\]
Taking \(k=0\) and \(k=2p\) gives \(w^{\mathsf T}p=0\). Hence
\[
 w^{\mathsf T}k\leq0\quad(k\in\mathcal K),
 \qquad
 w^{\mathsf T}c=\norm w_2^2>0.
\]
The cone contains \(\range(C^{\mathsf T})\) with both signs, so \(Cw=0\). It contains \(-e_k\) whenever \(y_k^*=0\), so \(w_k\geq0\) at those coordinates. Therefore \(y^*+\epsilon w\) is feasible for every sufficiently small \(\epsilon>0\). Its objective is strictly larger, a contradiction.

It follows that
\[
 c=C^{\mathsf T}u-s,
 \qquad
 s\geq0,\qquad s^{\mathsf T}y^*=0.
\]
Write \(\lambda=u_0\), and define
\[
 a_d=a_{-d}=u_d/2\quad(d\in D_+),
 \qquad
 a_k=0\quad(k\notin D).
\]
Then
\[
 C^{\mathsf T}u=\lambda\one+Fa\geq n\ezero,
 \qquad
 \lambda=u^{\mathsf T}v
 =(C^{\mathsf T}u)^{\mathsf T}y^*
 =c^{\mathsf T}y^*.
\]
Conversely, a feasible pair \((\lambda,a)\) in \eqref{eq:dual-a} satisfies, for every primal feasible \(y\),
\[
 ny_0
 \leq\lambda\sum_k y_k+\sum_k y_k(Fa)_k
 =\lambda+\sum_d a_d(Fy)_d
 =\lambda.
\]
This proves \eqref{eq:dual-a}, including attainment.

For a feasible pair in \eqref{eq:dual-a}, set
\[
 z=\frac{\lambda\one+Fa-n\ezero}{n}.
\]
The vector \(z\) is real, even, and nonnegative. The identities
\[
 F\one=n\ezero,\qquad F\ezero=\one,\qquad F^2a=na
\]
give
\[
 Fz=\lambda\ezero+a-\one,
 \qquad
 (Fz)_d=-1\quad(d\in N),
 \qquad
 \sum_k z_k=\lambda-1.
\]
Conversely, for \(z\) feasible in \eqref{eq:dual-z}, define
\[
 \lambda=1+\sum_k z_k,
 \qquad
 a=Fz+\one-\lambda\ezero.
\]
Then \(a\) vanishes at zero and on \(N\), and
\[
 \lambda\one+Fa-n\ezero=nz\geq0.
\]
The objectives agree in both directions. This proves \eqref{eq:dual-z} and its attainment.

The inverse substitutions \(t=Fz/n\) and \(z=Ft\) give
\[
 Ft=z\geq0,\qquad
 t_d=-1/n\quad(d\in N),\qquad
 \sum_k z_k=nt_0.
\]
They prove \eqref{eq:dual-t}.

Finally, put \(m=\one-a\) in \eqref{eq:dual-a}. Its prescribed coordinates are \(m_0=m_N=1\), and
\[
 Fm=n\ezero-Fa.
\]
Thus the dual inequality is equivalent to \(\lambda\geq\max_k(Fm)_k\). Every feasible \(m\) in \eqref{eq:dual-m} arises in this way. Since \(b\) also equals one at zero and on \(N\), these \(m\) are precisely \(b+h\) with \(h\) real, even, and supported on \(D\). This proves the remaining formulas and the certificate.
\end{proof}

\begin{prop}[Complementary product identity]\label{prop:complement}
Under Setup~\ref{setup:fourier},
\begin{equation}\label{eq:theta-product}
 \theta(G_D)\theta(G_N)=n,
\end{equation}
including empty and complete graphs. In the fair distance-bit model,
\begin{align}
 \E\theta(G_n)&\geq\sqrt n,                                  \label{eq:expectation-lower}\\
 \E\log\frac{\theta(G_n)}{\sqrt n}&=0,                        \label{eq:log-symmetry}\\
 \E\theta(G_n)-\sqrt n
 &=
 \frac12\E\left[
 \frac{(\theta(G_n)-\sqrt n)^2}{\theta(G_n)}
 \right].                                                   \label{eq:gap-identity}
\end{align}
\end{prop}

\begin{proof}
Choose attained maximizers \(x\) and \(v\) in \eqref{eq:primal-x} for \(G_D\) and \(G_N\). They satisfy
\[
 x_0=v_0=1,\qquad x_D=0,\qquad v_N=0,\qquad Fx,Fv\geq0.
\]
Put \(y=Fx/n\) and \(q=Fv/n\). These are nonnegative probability vectors. The partition of the nonzero residues into \(D\) and \(N\) gives
\[
 \sum_d x_dv_d=1.
\]
Using \eqref{eq:Fourier-identities} and \(F^{\mathsf T}=F\), we obtain
\[
 1=n\sum_k y_kq_k\geq ny_0q_0.
\]
Since \(\theta(G_D)=ny_0\) and \(\theta(G_N)=nq_0\), their product is at most \(n\).

For the reverse inequality, choose an attained optimizer \(z\) in \eqref{eq:dual-z} for \(G_D\), and set
\[
 L=1+\sum_k z_k=\theta(G_D),
 \qquad
 v'=\frac{Fz+\one}{L}.
\]
Then
\[
 v'_0=1,\qquad v'_N=0,\qquad
 Fv'=\frac{nz+n\ezero}{L}\geq0.
\]
The vector \(v'\) is real and even, so it is feasible in \eqref{eq:primal-x} for \(G_N\). Its objective is
\[
 \sum_d v'_d=\frac{nz_0+n}{L}\geq\frac nL.
\]
This proves the reverse inequality and hence \eqref{eq:theta-product}. The argument also covers either empty difference set.

Complementing every fair bit preserves the joint law. Therefore, for \(T=\theta(G_n)\), the variables \(T\) and \(n/T\) have the same distribution. The bounds \(1\leq T\leq n\) ensure integrability. The identity
\[
 T+\frac nT-2\sqrt n
 =
 \left(\sqrt T-\sqrt{\frac nT}\right)^2
 \geq0
\]
gives \eqref{eq:expectation-lower} after taking expectations. Equality in distribution also gives
\[
 \E\log T=\E\log(n/T)=\log n-\E\log T,
\]
which proves \eqref{eq:log-symmetry}. Finally,
\[
 \frac{(T-\sqrt n)^2}{T}=T+\frac nT-2\sqrt n,
\]
and taking expectations proves \eqref{eq:gap-identity}.
\end{proof}

\begin{prop}[Fourier tail and exceptional-event estimate]\label{prop:tail}
For every odd \(n\geq3\) and every \(s\geq0\),
\begin{equation}\label{eq:theta-tail-simple}
 \Prob\{\theta(G_n)>1+s\}
 \leq
 \min\left\{1,n\exp\left(-\frac{s^2}{4(n-1)}\right)\right\}.
\end{equation}
More precisely,
\begin{equation}\label{eq:theta-tail-precise}
 \Prob\{\theta(G_n)>1+s\}
 \leq
 \exp\left(-\frac{s^2}{4(n-1)}\right)
 +(n-1)\exp\left(-\frac{s^2}{2(n-2)}\right).
\end{equation}
Let \(A\) be any event on the distance-bit probability space, and put \(\delta=\Prob(A)\). For \(0<\delta\leq1\), set
\[
 U=2\sqrt{(n-1)\log(n/\delta)}.
\]
Then
\begin{equation}\label{eq:event-tail}
 \E[\theta(G_n)\ind_A]
 \leq
 \min\left\{
 n\delta,\
 \delta\left(1+U+\frac{2(n-1)}U\right)
 \right\}.
\end{equation}
The expectation is zero if \(\delta=0\).

For any sequence of events \(A_n\),
\[
 \Prob(A_n)\sqrt{\log n}\longrightarrow0
 \quad\Longrightarrow\quad
 \E[\theta(G_n)\ind_{A_n}]=o(\sqrt n).
\]
Consequently, if \(\epsilon_n\geq0\), \(\epsilon_n\to0\), and
\[
 \theta(G_n)\leq(1+\epsilon_n)\sqrt n
 \quad\text{on }A_n^c,
 \qquad
 \Prob(A_n)=o((\log n)^{-1/2}),
\]
then \(\E\theta(G_n)\leq(1+o(1))\sqrt n\). Independently, the deterministic bound \(\theta(G_n)\leq n\) makes \(\Prob(A_n)=o(n^{-1/2})\) sufficient.
\end{prop}

\begin{proof}
For the random edge set, put \(b_0=1\) and \(b_d=b_{-d}=\epsilon_d=1-2\xi_d\) for \(1\leq d\leq(n-1)/2\). Proposition~\ref{prop:fourier} gives
\[
 \theta(G_n)\leq\max_k g_k,\qquad \theta(G_n)\leq n,
 \qquad g=Fb,
\]
because every feasible \(y\) in \eqref{eq:primal-g} is a probability vector. Pairing reflected coordinates gives
\begin{equation}\label{eq:gk-sign-sum}
 g_k=1+2\sum_{d=1}^{(n-1)/2}\epsilon_d\cos(2\pi kd/n).
\end{equation}
Coefficient comparison, using \((2j)!\geq2^jj!\), proves
\[
 \cosh u\leq e^{u^2/2}\qquad(u\in\R).
\]
Define
\[
 S_k=4\sum_{d=1}^{(n-1)/2}\cos^2(2\pi kd/n).
\]
Independence in \eqref{eq:gk-sign-sum} gives
\[
 \E e^{t(g_k-1)}\leq e^{t^2S_k/2}\qquad(t\geq0).
\]
Applying Markov's inequality and taking \(t=s/S_k\), we obtain
\begin{equation}\label{eq:gk-tail}
 \Prob(g_k-1>s)\leq\exp\left(-\frac{s^2}{2S_k}\right).
\end{equation}

At zero, \(S_0=2(n-1)\). For \(k\neq0\), oddness implies \(2k\neq0\) modulo \(n\). Thus the geometric-series identity gives
\[
 \sum_{d=0}^{n-1}\cos(4\pi kd/n)=0,
 \qquad
 \sum_{d=0}^{n-1}\cos^2(2\pi kd/n)=\frac n2.
\]
Removing the term at zero and pairing reflected indices yields
\[
 \sum_{d=1}^{(n-1)/2}\cos^2(2\pi kd/n)=\frac{n-2}{4}.
\]
Hence \(S_k=n-2\) for nonzero \(k\). A union bound in \eqref{eq:gk-tail} proves \eqref{eq:theta-tail-precise}. Bounding every \(S_k\) by \(2(n-1)\), and probabilities by one, proves \eqref{eq:theta-tail-simple}.

Let \(Z=(\theta(G_n)-1)_+\). Since \(\theta(G_n)\geq1\), we have \(\theta(G_n)=1+Z\). For \(U>0\),
\[
 Z\ind_A\leq U\ind_A+(Z-U)_+.
\]
Integrating the lengths of the intervals \(\{s:U<s<Z\}\) on the finite probability space gives
\[
 \E[(Z-U)_+]=\int_U^\infty\Prob(Z>s)\,ds.
\]
Therefore
\[
 \E[\theta(G_n)\ind_A]
 \leq\delta(1+U)
 +n\int_U^\infty e^{-s^2/(4(n-1))}\,ds.
\]
The Gaussian tail satisfies
\[
 \int_U^\infty e^{-s^2/(4(n-1))}\,ds
 \leq\frac1U\int_U^\infty s e^{-s^2/(4(n-1))}\,ds
 =\frac{2(n-1)}Ue^{-U^2/(4(n-1))}.
\]
For the stated choice of \(U\), the exponential equals \(\delta/n\). This gives the second bound in \eqref{eq:event-tail}. The first follows from \(\theta(G_n)\leq n\). A null event contributes zero.

Dividing \eqref{eq:event-tail} by \(\sqrt n\) gives the convenient bound
\begin{equation}\label{eq:normalized-event-bound}
 \E\left[\frac{\theta(G_n)}{\sqrt n}\ind_A\right]
 \leq
 \frac{\delta}{\sqrt n}
 +2\delta\sqrt{\log(n/\delta)}
 +\frac{\delta}{\sqrt{\log(n/\delta)}}
\end{equation}
for \(\delta>0\). The inequality
\[
 \sqrt{\log(n/\delta)}
 \leq\sqrt{\log n}+\sqrt{\log(1/\delta)}
\]
shows that the middle term tends to zero when \(\delta\sqrt{\log n}\to0\). This condition implies \(\delta\to0\), and
\[
 \delta\sqrt{\log(1/\delta)}\longrightarrow0.
\]
For the latter limit, put \(v=\log(1/\delta)\) and use \(e^v\geq v^2/2\) to obtain
\[
 e^{-v}\sqrt v\leq2v^{-3/2}\longrightarrow0.
\]
The last term in \eqref{eq:normalized-event-bound} is at most \(\delta/\sqrt{\log n}\). This proves the exceptional-event assertion, with zero-probability events treated separately.

Integrating the assumed theta bound on \(A_n^c\) and adding this exceptional contribution proves the stated expectation consequence. The alternative condition follows directly from
\[
 \E[\theta(G_n)\ind_{A_n}]\leq n\Prob(A_n)=o(\sqrt n).
\]
\end{proof}

\begin{cor}[Elementary expectation bounds]\label{cor:elementary-expectation}
In the fair distance-bit model,
\[
 \sqrt n\leq\E\theta(G_n)
 \leq\bigl(\sqrt2+o(1)\bigr)\sqrt{n\log n}.
\]
\end{cor}

\begin{proof}
The lower bound is \eqref{eq:expectation-lower}. For the upper bound, let
\[
 t_n=\sqrt{2(n-2)\log(n-1)}.
\]
Integrating \eqref{eq:theta-tail-precise} above \(t_n\), and using the Gaussian tail integration from Proposition~\ref{prop:tail}, gives
\begin{align*}
 \E\theta(G_n)
 &\leq
 1+t_n
 +\frac{2(n-1)}{t_n}e^{-t_n^2/(4(n-1))}
 +\frac{(n-1)(n-2)}{t_n}e^{-t_n^2/(2(n-2))}\\
 &=
 1+t_n
 +\frac{2(n-1)}{t_n}e^{-t_n^2/(4(n-1))}
 +\frac{n-2}{t_n}\\
 &\leq1+t_n+\frac{3n-4}{t_n}.
\end{align*}
Here \(t_n/\sqrt{n\log n}\to\sqrt2\), and the other terms are \(o(\sqrt{n\log n})\).
\end{proof}

\section{Interior smoothing and residual certificates}\label{sec:interior}

The goal of this section is to construct positive complementary certificates and prove the deterministic estimate that connects their roots to approximate fixed points.

\begin{defn}[Smoothing weights]\label{defn:smoothing}
Under Setup~\ref{setup:fourier}, let \(\tau>0\), and define
\begin{equation}\label{eq:smoothing-weights}
 w_{\tau,0}=1,\qquad
 w_{\tau,k}=\frac{\tau}{n}\quad(k\neq0),\qquad
 W_\tau=\sum_k w_{\tau,k}=1+\tau\left(1-\frac1n\right).
\end{equation}
For a real vector \(v\), set
\begin{equation}\label{eq:Phi}
 \Phi_\tau(v)_k=\sqrt{v_k^2+4w_{\tau,k}}.
\end{equation}
The corresponding scalar bulk map is
\[
 \phi_\tau(t)=\sqrt{t^2+4\tau}.
\]
Thus \(\Phi_\tau(v)_0=\sqrt{v_0^2+4}\), and
\[
 \Phi_\tau(v)_k=n^{-1/2}\phi_\tau(\sqrt n\,v_k)
 \qquad(k\neq0).
\]
\end{defn}

\begin{prop}[Interior complementary smoothing]\label{prop:smoothing}\label{prop:interior}
Let \(n,D,N\) satisfy Setup~\ref{setup:fourier}, and let \(\tau>0\). There is a unique pair of strictly positive real even vectors \(p_\tau,q_\tau\) satisfying
\begin{equation}\label{eq:interior-conditions}
 H_Dp_\tau=0,\qquad
 H_Nq_\tau=0,\qquad
 \sum_kp_{\tau,k}=\sum_kq_{\tau,k},\qquad
 p_{\tau,k}q_{\tau,k}=w_{\tau,k}.
\end{equation}
Their common sum is
\begin{equation}\label{eq:interior-sum}
 \sum_kp_{\tau,k}=\sum_kq_{\tau,k}=\sqrt{nW_\tau}.
\end{equation}
Equivalently, \(f_{\mathrm{int},\tau}=p_\tau-q_\tau\) is the unique real even mean-zero solution of
\begin{equation}\label{eq:interior-fixed-point}
 f_{\mathrm{int},\tau}=K_0\Phi_\tau(f_{\mathrm{int},\tau}),
\end{equation}
with
\[
 p_\tau=\frac{\Phi_\tau(f_{\mathrm{int},\tau})+f_{\mathrm{int},\tau}}2,
 \qquad
 q_\tau=\frac{\Phi_\tau(f_{\mathrm{int},\tau})-f_{\mathrm{int},\tau}}2.
\]
The vector \(p_\tau\) uniquely minimizes
\begin{equation}\label{eq:barrier}
 B_\tau(p)=\frac{(\sum_kp_k)^2}{2n}
 -\sum_kw_{\tau,k}\log p_k
\end{equation}
over strictly positive real even \(p\) satisfying \(H_Dp=0\).

For \(X=\theta(G_D)/\sqrt n\),
\begin{align}
 \frac{p_{\tau,0}}{\sqrt{W_\tau}}
 &\leq X\leq\sqrt{W_\tau}\,p_{\tau,0},                       \label{eq:interior-theta-bound}\\
 \left|\log X-\arsinh\left(\frac{f_{\mathrm{int},\tau,0}}2\right)\right|
 &\leq\frac12\log W_\tau.                                   \label{eq:interior-log-bound}
\end{align}
In the fair distance-bit model, complementation swaps \(p_\tau\) and \(q_\tau\), and
\begin{equation}\label{eq:interior-expectation}
 \E p_{\tau,0}
 =\frac12\E\sqrt{f_{\mathrm{int},\tau,0}^2+4},
 \qquad
 1\leq\E X
 \leq\sqrt{W_\tau}\left(1+\frac12\E|f_{\mathrm{int},\tau,0}|\right).
\end{equation}
Consequently,
\begin{equation}\label{eq:root-sufficiency}
 \lim_{\tau\downarrow0}
 \limsup_{\substack{n\to\infty\\n\ {\rm odd}}}
 \E|f_{\mathrm{int},\tau,0}|=0
\end{equation}
is sufficient for Theorem~\ref{thm:main}. The proposition does not assert \eqref{eq:root-sufficiency}.
\end{prop}

\begin{proof}
We work in the real Euclidean subspace of even vectors. The identity \(F^*F=nI\) shows that the operators in \eqref{eq:projections} are mutually orthogonal projections, because their Fourier multipliers are the indicators of the disjoint sets \(\{0\},D,N\). Thus
\begin{equation}\label{eq:projection-decomposition}
 J+H_D+H_N=I,\qquad
 K_0^2=I-J,\qquad K_0J=0.
\end{equation}
Write \(L_D=\ker H_D\), \(w=w_\tau\), and \(W=W_\tau\).

\smallskip
\noindent\textbf{Step 1. The barrier has a unique interior minimizer.}
The positive part of \(L_D\) contains \(\one\). On the positive orthant, the Hessian of \eqref{eq:barrier} is
\[
 J+\diag\left(\frac{w_k}{p_k^2}\right).
\]
Every \(w_k\) is positive, so this Hessian is positive definite. Hence the restriction of \(B_\tau\) is strictly convex.

For attainment, put \(S=\sum_kp_k\). Since \(0<p_k\leq S\),
\[
 B_\tau(p)\geq\frac{S^2}{2n}-W\log S.
\]
The right-hand side tends to infinity as \(S\to0\) or \(S\to\infty\). Thus \(S\) is bounded above and away from zero in any fixed nonempty sublevel set. If all coordinates are at most \(M\) and \(p_i\to0\), then
\[
 -w_i\log p_i\to\infty,
 \qquad
 -\sum_{k\neq i}w_k\log p_k\geq-(W-w_i)\log M.
\]
The quadratic term is nonnegative, so a coordinate cannot approach zero in that sublevel set. Its closure in \(L_D\) is therefore compact and lies inside the domain. A minimizing sequence has a convergent subsequence there. Strict convexity makes the minimizer \(p\) unique.

Define \(q_k=w_k/p_k\). Differentiating in each direction \(h\in L_D\) gives
\[
 \inner{Jp-q}{h}=0.
\]
Therefore \(q=Jp+r\) for some \(r\in\range H_D\). The vector \(q\) is positive, real, and even, and the projection decomposition gives
\[
 H_Nq=0,\qquad Jq=Jp.
\]
Both sums equal \(S\). Since \(p\in\range J+\range H_N\) and \(q\in\range J+\range H_D\),
\[
 W=\sum_kp_kq_k=\inner p q=\inner{Jp}{Jq}=\frac{S^2}{n}.
\]
Positivity gives \(S=\sqrt{nW}\), proving \eqref{eq:interior-sum}.

Conversely, any positive even pair satisfying \eqref{eq:interior-conditions} has \(q=w/p\) and \(Jp-q\in\range H_D\). It satisfies the stationarity equation for the strictly convex barrier. Its first component is therefore the unique minimizer, and its second is \(w/p\). This proves uniqueness of the pair. Exchanging \(D\) and \(N\) reverses the pair, so complementation swaps \(p\) and \(q\).

\smallskip
\noindent\textbf{Step 2. The pair is equivalent to the fixed point.}
Set \(f=p-q\) and \(u=p+q\). Equal sums give \(Jf=0\), while the support conditions and \eqref{eq:projection-decomposition} give
\[
 K_0u=p-q=f.
\]
Coordinatewise positivity and \(p_kq_k=w_k\) imply
\[
 u_k=\sqrt{f_k^2+4w_k}.
\]
This proves \eqref{eq:interior-fixed-point}.

Conversely, let \(f\) be real, even, and mean zero, with \(f=K_0u\) and \(u=\Phi_\tau(f)\). Put
\[
 p=\frac{u+f}{2},\qquad q=\frac{u-f}{2}.
\]
The inequality \(u_k>|f_k|\) gives positivity, and direct multiplication gives \(p_kq_k=w_k\). Equal sums follow from \(Jf=0\). Projecting the equation onto the two nonconstant ranges gives
\[
 H_Df=-H_Du,\qquad H_Nf=H_Nu.
\]
Hence \(H_Dp=0\) and \(H_Nq=0\). Uniqueness of the pair proves uniqueness of the fixed point.

\smallskip
\noindent\textbf{Step 3. The root bounds theta.}
For a real even vector \(z\), reflected pairing gives \(F^*z=Fz\). Multiplying the formula for \(H_Dz\) by \(F^*\) shows
\begin{equation}\label{eq:kernel-feasibility}
 H_Dz=0
 \quad\Longleftrightarrow\quad
 (Fz)_d=0\quad(d\in D).
\end{equation}
The corresponding equivalence holds for \(N\). Therefore \(p/S\) and \(q/S\), where \(S=\sqrt{nW}\), are feasible in \eqref{eq:primal-y} for \(G_D\) and \(G_N\). Proposition~\ref{prop:fourier} gives
\[
 \theta(G_D)\geq\frac{np_0}{S},
 \qquad
 \theta(G_N)\geq\frac{nq_0}{S}.
\]
The complementary identity \eqref{eq:theta-product} then gives
\[
 \theta(G_D)\leq\frac S{q_0}.
\]
Since \(p_0q_0=w_0=1\), these inequalities prove \eqref{eq:interior-theta-bound}.

The positive solution of \(p_0-p_0^{-1}=f_0\) is
\[
 p_0=\frac{\sqrt{f_0^2+4}+f_0}{2}
 =\exp\left(\arsinh\left(\frac{f_0}{2}\right)\right).
\]
Taking logarithms proves \eqref{eq:interior-log-bound}.

\smallskip
\noindent\textbf{Step 4. Complementation gives the expectation estimate.}
In the fair model, \(p_0\) and \(q_0=p_0^{-1}\) have the same distribution. Their expectations exist because the sign space is finite. Hence
\[
 \E p_0=\frac12\E(p_0+q_0)
 =\frac12\E\sqrt{f_0^2+4}
 \leq1+\frac12\E|f_0|.
\]
The last inequality uses \(\sqrt{t^2+4}\leq2+|t|\). Averaging \eqref{eq:interior-theta-bound} and using \eqref{eq:expectation-lower} proves \eqref{eq:interior-expectation}. Since \(W_\tau\leq1+\tau\),
\[
 \limsup_n\E X
 \leq\sqrt{1+\tau}
 \left(1+\frac12\limsup_n\E|f_{\mathrm{int},\tau,0}|\right).
\]
Under \eqref{eq:root-sufficiency}, letting \(\tau\downarrow0\) gives an upper limit of at most one. The lower bound is one for every \(n\). This proves the stated sufficient condition without asserting its probabilistic hypothesis.
\end{proof}

\begin{prop}[Deterministic residual certificate]\label{prop:residual}
Fix \(n,D,N,\tau\) as in Proposition~\ref{prop:smoothing}, and write
\[
 f^*=f_{\mathrm{int},\tau},\qquad
 p^*=p_\tau,\qquad
 w=w_\tau,\qquad W=W_\tau.
\]
For a real even mean-zero vector \(v\), define
\[
 u=\Phi_\tau(v),\qquad
 r_v=v-K_0u,\qquad
 a=\frac{u+v}{2},\qquad
 R_k=\frac{a_k}{p_k^*}.
\]
Every \(R_k\) is positive. The discrepancy
\[
 D_{\mathrm{rel}}(v)=\sum_kw_k(R_k+R_k^{-1}-2)
\]
satisfies
\begin{align}
 D_{\mathrm{rel}}(v)
 &=\sqrt{\frac Wn}\sum_ku_k-2W-\frac12\inner{r_v}{f^*},
                                                               \label{eq:relative-identity}\\
 0\leq D_{\mathrm{rel}}(v)
 &\leq
 \left(\norm v_2+\frac12\norm{f^*}_2\right)\norm{r_v}_2.
                                                               \label{eq:relative-bound}
\end{align}
Moreover,
\begin{equation}\label{eq:residual-certificate}
 \left|
 \log\frac{\theta(G_D)}{\sqrt n}
 -\arsinh\left(\frac{v_0}{2}\right)
 \right|
 \leq\frac12\log W+\sqrt{D_{\mathrm{rel}}(v)}.
\end{equation}
\end{prop}

\begin{proof}
Let \(q^*=q_\tau\) and \(u^*=p^*+q^*\). Proposition~\ref{prop:smoothing} gives
\[
 f^*=p^*-q^*=K_0u^*,\qquad
 \sum_ku_k^*=2\sqrt{nW},\qquad
 p_k^*q_k^*=w_k.
\]
Set \(b=(u-v)/2\). Since \(u_k=\sqrt{v_k^2+4w_k}>|v_k|\), both \(a\) and \(b\) are positive and \(a_kb_k=w_k\). Thus
\[
 a_kq_k^*=w_kR_k,\qquad
 b_kp_k^*=\frac{w_k}{R_k}.
\]
Expanding the two inner products gives
\begin{equation}\label{eq:relative-expand}
 \sum_kw_k(R_k+R_k^{-1})
 =\inner a{q^*}+\inner b{p^*}
 =\frac12\bigl(\inner u{u^*}-\inner v{f^*}\bigr).
\end{equation}
The operator \(K_0\) is symmetric, and \(K_0^2=I-J\). Consequently,
\begin{align*}
 \inner u{u^*}
 &=\inner{K_0u}{K_0u^*}+\inner{Ju}{Ju^*}\\
 &=\inner{v-r_v}{f^*}
   +\frac{(\sum_ku_k)(2\sqrt{nW})}{n}.
\end{align*}
Substitution into \eqref{eq:relative-expand} proves \eqref{eq:relative-identity}. Each summand of the discrepancy is nonnegative because
\[
 R+R^{-1}-2=\left(\sqrt R-\frac1{\sqrt R}\right)^2.
\]

To estimate the discrepancy, put \(s=(\sum_ku_k)/\sqrt n>0\). We have
\[
 \norm u_2^2=\norm v_2^2+4W,\qquad
 \norm{K_0u}_2^2=\norm u_2^2-s^2.
\]
Expanding \(r_v=v-K_0u\) therefore gives
\begin{equation}\label{eq:s-square}
 s^2-4W=2\inner v{r_v}-\norm{r_v}_2^2.
\end{equation}
It follows that
\begin{equation}\label{eq:s-linear}
 \sqrt W(s-2\sqrt W)
 =
 \frac{\sqrt W}{s+2\sqrt W}
 \left(2\inner v{r_v}-\norm{r_v}_2^2\right).
\end{equation}
The prefactor lies in \((0,1/2]\). If the parenthesis is nonpositive, the left-hand side is at most zero. If it is positive, the left-hand side is at most
\[
 \inner v{r_v}\leq\norm v_2\norm{r_v}_2.
\]
In either case,
\[
 \sqrt W(s-2\sqrt W)\leq\norm v_2\norm{r_v}_2.
\]
Combining this bound with \eqref{eq:relative-identity}, and applying Cauchy--Schwarz to its remaining inner product, proves \eqref{eq:relative-bound}.

At the root, \(w_0=1\), so
\[
 a_0=\frac{\sqrt{v_0^2+4}+v_0}{2}
 =\exp\left(\arsinh\left(\frac{v_0}{2}\right)\right).
\]
Set \(z=\log R_0\). The power series of \(\cosh z\) gives
\[
 e^z+e^{-z}-2\geq z^2.
\]
The root summand of \(D_{\mathrm{rel}}(v)\) is the left-hand side. Hence
\[
 \left|\log p_0^*-\arsinh\left(\frac{v_0}{2}\right)\right|
 \leq\sqrt{D_{\mathrm{rel}}(v)}.
\]
Together with \eqref{eq:interior-log-bound}, this proves \eqref{eq:residual-certificate}.
\end{proof}

\section{Paired Fourier sampling}
\label{sec:tech-fourier}
\label{fourier:sec:paired-fourier}

In this section, we prove the simultaneous kernel bounds in
Proposition~\ref{fourier:prop:paired-kernel}.
The argument combines a direct concentration estimate, a multiscale
approximation for independent Fourier rows, and a transfer to paired
distance bits.

\begin{notation}\label{fourier:notation:sampling}
We use the unnormalized Fourier matrix \(F\), the sets \(D,N\),
the multiplier \(a\), and the operators \(K_0,J\) from
Setup~\ref{setup:fourier}.
Let \(n\geq3\) be odd, and identify the row and column indices with
\([n]=\{0,\ldots,n-1\}\), with arithmetic modulo \(n\).
For \(T\subseteq[n]\), let \(F_T\) consist of the rows indexed by \(T\).
For \(x\in\C^n\), write
\[
 \norm{x}_1=\sum_{j=0}^{n-1}|x_j|,
 \qquad
 \norm{x}_2=\left(\sum_{j=0}^{n-1}|x_j|^2\right)^{1/2},
\]
and let \(x_T\) agree with \(x\) on \(T\) and vanish elsewhere.

Write
\[
 \sigma_j=\epsilon_j=1-2\xi_j
 \quad\left(1\leq j\leq\frac{n-1}{2}\right).
\]
These are the independent fair signs of Setup~\ref{setup:fourier};
equivalently, \(\xi_j=(1-\sigma_j)/2\) are the original independent
Bernoulli variables with parameter \(1/2\).
For \(r\neq0\), membership in \(D\) is equivalent to
\(\sigma_{\min(r,n-r)}=-1\), or to \(\xi_{\min(r,n-r)}=1\).
In this section, write
\[
 D^{\mathrm c}=N=([n]\setminus\{0\})\setminus D.
\]
The zero row belongs to neither set, and the multiplier retains
\(a_0=0\).
All logarithms are natural unless the base is displayed.
Averages over a multiset count multiplicities.
\end{notation}

\begin{lem}[Mixed relative and additive concentration]
\label{fourier:lem:mixed-concentration}
Let \(q\geq1\) be an integer and \(a>0\).
Let \(X_1,\ldots,X_q\) be independent real random variables with
\(0\leq X_i\leq a\) almost surely and common mean \(\E X_i=\mu\).
Set \(\overline X=q^{-1}\sum_{i=1}^qX_i\).
For every \(0<\epsilon\leq1/2\) and \(\alpha>0\),
\begin{equation}\label{fourier:eq:mixed-concentration}
 \Prob\{|\overline X-\mu|>\epsilon\mu+\alpha\}
 \leq2\exp\left(-\frac{q\epsilon\alpha}{2a}\right).
\end{equation}
Each one-sided tail is at most
\begin{equation}\label{fourier:eq:mixed-one-sided}
 \exp\left(
 -\frac{q\mu\epsilon^2}{4a}
 -\frac{q\epsilon\alpha}{2a}
 \right).
\end{equation}
The estimates include \(\mu=0\) and impose no restriction on
\(\alpha/\mu\).
\end{lem}

\begin{proof}
Normalize the variables by
\[
 Y_i=X_i/a,\qquad \nu=\mu/a,\qquad A=\alpha/a.
\]
Then \(0\leq Y_i\leq1\) and \(\E Y_i=\nu\).
For \(0\leq t\leq1\), convexity gives, for \(0\leq y\leq1\),
\[
 e^{ty}\leq1+y(e^t-1),
 \qquad
 e^{-ty}\leq1+y(e^{-t}-1).
\]
Taking expectations and using \(1+u\leq e^u\) yields
\[
 \begin{aligned}
 \E e^{t(Y_i-\nu)}
 &\leq \exp\bigl(\nu(e^t-1-t)\bigr),\\
 \E e^{-t(Y_i-\nu)}
 &\leq \exp\bigl(\nu(e^{-t}-1+t)\bigr).
 \end{aligned}
\]
The factor \(1+\nu(e^{-t}-1)\) in the second estimate is positive,
since \(0\leq\nu\leq1\) and \(e^{-t}>0\).

The terms of order at least two in either exponential series have
absolute sum bounded by
\[
 t^2\sum_{j\geq2}\frac1{j!}
 \leq t^2\sum_{j\geq2}2^{-(j-1)}
 =t^2.
\]
Here \(j!\geq2^{j-1}\) for \(j\geq2\).
Consequently,
\(e^t-1-t\leq t^2\) and \(e^{-t}-1+t\leq t^2\).
Independence therefore gives
\begin{equation}\label{fourier:eq:centered-mgf}
 \E\exp\left(\pm t\sum_{i=1}^q(Y_i-\nu)\right)
 \leq \exp(q\nu t^2).
\end{equation}

Take \(t=\epsilon/2\), so \(0<t\leq1/4\).
Markov's inequality applied to \eqref{fourier:eq:centered-mgf} gives
\[
 \begin{aligned}
 \Prob\left\{\sum_{i=1}^q(Y_i-\nu)>q(\epsilon\nu+A)\right\}
 &\leq \exp\bigl(-tq(\epsilon\nu+A)+q\nu t^2\bigr)\\
 &=\exp\left(-\frac{q\nu\epsilon^2}{4}
             -\frac{q\epsilon A}{2}\right).
 \end{aligned}
\]
Applying Markov's inequality to the negative exponential gives this
same bound for the event
\(\sum_i(Y_i-\nu)<-q(\epsilon\nu+A)\).
Substituting \(\nu=\mu/a\) and \(A=\alpha/a\) proves
\eqref{fourier:eq:mixed-one-sided}.
Adding the two bounds and discarding the nonpositive term involving
\(\mu\) proves \eqref{fourier:eq:mixed-concentration}.
No step divides by \(\mu\).
\end{proof}

\begin{lem}[Paired-distance transfer]
\label{fourier:lem:paired-transfer}
Let \(n\geq3\) be odd and \(q=\lceil n/4\rceil\).
Fix an integer \(1\leq k\leq n\), \(0<\epsilon<1\), and
\(0\leq\delta\leq1\).
Let \(R_1,\ldots,R_q\) be independent uniform elements of \([n]\),
and let \(A_q\) have row \(\ell\) equal to \(F_{R_\ell,:}/\sqrt q\).
Assume that, with probability at least \(1-\delta\),
\begin{equation}\label{fourier:eq:rip-assumption}
 (1-\epsilon)\norm{v}_2^2
 \leq\norm{A_qv}_2^2
 \leq(1+\epsilon)\norm{v}_2^2
\end{equation}
simultaneously for every complex vector \(v\) supported on at most
\(k\) coordinates.

Define
\[
 \lambda=\frac{\log2}{2},
 \qquad
 c_*=\lambda-\frac14-\frac14\log(4\lambda),
 \qquad
 D_\epsilon=1+\sqrt{\frac{1+\epsilon}{1-\epsilon}}.
\]
Then \(c_*>0\).
For the paired set \(D\) of Notation~\ref{fourier:notation:sampling},
with probability at least
\[
 1-\sqrt2\bigl(\delta+e^{-c_*n}\bigr),
\]
every \(x\in\C^n\) with \(F_Dx=0\) satisfies
\begin{equation}\label{fourier:eq:paired-transfer-conclusion}
 \norm{x}_2\leq D_\epsilon\frac{\norm{x}_1}{\sqrt k}.
\end{equation}
The probability lower bound may be negative, in which case it is
vacuous.
\end{lem}

\begin{proof}
First fix a matrix satisfying \eqref{fourier:eq:rip-assumption}.
For \(x\in\ker A_q\), arrange the coordinates in decreasing order of
absolute value and partition them into consecutive sets
\(T_0,T_1,\ldots\) of size \(k\), except possibly the last.
Whenever \(j\geq1\), the preceding block \(T_{j-1}\) has \(k\)
coordinates, each at least as large in absolute value as any
coordinate in \(T_j\).
Thus
\[
 \norm{x_{T_j}}_2
 \leq\frac{\norm{x_{T_{j-1}}}_1}{\sqrt k},
 \qquad
 \sum_{j\geq1}\norm{x_{T_j}}_2
 \leq\frac{\norm{x}_1}{\sqrt k}.
\]
If \(T_0\) is the only block, the lower inequality in
\eqref{fourier:eq:rip-assumption} gives \(x=0\).
Otherwise,
\[
 A_qx_{T_0}=-\sum_{j\geq1}A_qx_{T_j}.
\]
Applying the lower bound to \(x_{T_0}\), the upper bound to the other
blocks, and the triangle inequality gives
\[
 \sqrt{1-\epsilon}\,\norm{x_{T_0}}_2
 \leq\sqrt{1+\epsilon}\sum_{j\geq1}\norm{x_{T_j}}_2.
\]
A second triangle inequality proves the deterministic kernel bound
\begin{equation}\label{fourier:eq:rip-kernel}
 \begin{aligned}
 \norm{x}_2
 &\leq\norm{x_{T_0}}_2+\sum_{j\geq1}\norm{x_{T_j}}_2\\
 &\leq
 \left(1+\sqrt{\frac{1+\epsilon}{1-\epsilon}}\right)
 \frac{\norm{x}_1}{\sqrt k}.
 \end{aligned}
\end{equation}

To transfer this estimate, take an infinite sequence
\(R_1,R_2,\ldots\) of independent uniform residues.
Independently, let \(L\) be Poisson with mean \(\lambda n\), and put
\[
 N_r=\#\{\ell\leq L:R_\ell=r\}.
\]
For nonnegative integers \(a_0,\ldots,a_{n-1}\), write
\(a=\sum_ra_r\).
The joint mass function is
\begin{equation}\label{fourier:eq:poisson-splitting}
 \begin{aligned}
 \Prob\{N_r=a_r\text{ for every }r\}
 &=e^{-\lambda n}\frac{(\lambda n)^a}{a!}
   \frac{a!}{\prod_ra_r!}n^{-a}\\
 &=\prod_{r=0}^{n-1}e^{-\lambda}\frac{\lambda^{a_r}}{a_r!}.
 \end{aligned}
\end{equation}
The factorization in \eqref{fourier:eq:poisson-splitting} shows that
the counts \(N_r\) are independent Poisson variables of mean
\(\lambda\).

Let \(H=\{N_0=0\}\), so \(\Prob(H)=e^{-\lambda}=1/\sqrt2\).
Conditional on \(H\), the nonzero counts remain independent with
their original laws.
The reflected pairs are disjoint because \(n\) is odd, and
\[
 \Prob\{N_r+N_{n-r}=0\mid H\}=e^{-2\lambda}=\frac12,
 \qquad 1\leq r\leq(n-1)/2.
\]
It follows that
\[
 D'=\{r\neq0:N_r+N_{n-r}>0\},
\]
under the conditional law given \(H\), has the distribution of \(D\).

Let \(B\) be the event that the matrix formed from the first \(q\)
indices fails \eqref{fourier:eq:rip-assumption}.
Its unconditional probability is at most \(\delta\).
On \(H\cap\{L\geq q\}\cap B^{\mathrm c}\), these first \(q\) indices
are nonzero and belong to \(D'\).
Even when some indices repeat, this implies
\[
 \ker F_{D'}\subseteq\ker A_q.
\]
Hence \eqref{fourier:eq:rip-kernel} holds simultaneously on
\(\ker F_{D'}\).
This inclusion uses no independence between \(B\), \(H\), and \(L\).

Since \(L\) is integer valued,
\(L<\lceil n/4\rceil\) implies \(L\leq n/4\).
Set \(t=\log(4\lambda)>0\).
The Poisson moment formula and Markov's inequality give
\begin{equation}\label{fourier:eq:poisson-lower-tail}
 \begin{aligned}
 \Prob\{L<q\}
 &\leq e^{tn/4}\E e^{-tL}\\
 &=\exp\left(\frac{tn}{4}+\lambda n(e^{-t}-1)\right)
 =e^{-c_*n}.
 \end{aligned}
\end{equation}
For positivity, put \(u=4\lambda=2\log2\).
Integrating \(1/s>1/2\) over \(1\leq s<2\) gives \(u>1\).
The function \(u-1-\log u\) vanishes at \(1\) and has derivative
\(1-1/u>0\) for \(u>1\).
Therefore
\[
 4c_*=u-1-\log u>0.
\]

By \eqref{fourier:eq:poisson-lower-tail},
\[
 \begin{aligned}
 \Prob\bigl(B\cup\{L<q\}\mid H\bigr)
 &\leq\frac{\Prob(B)+\Prob(L<q)}{\Prob(H)}\\
 &\leq\sqrt2\bigl(\delta+e^{-c_*n}\bigr).
 \end{aligned}
\]
The conditional law of \(D'\) is the required paired-distance law,
so \eqref{fourier:eq:rip-kernel} proves
\eqref{fourier:eq:paired-transfer-conclusion}.
\end{proof}

\begin{prop}[Paired Fourier kernel estimate]
\label{fourier:prop:paired-kernel}
There are absolute constants \(C,c>0\) and an integer \(n_0\) such
that, for every odd \(n\geq n_0\), with probability at least
\[
 1-\exp\bigl(-c(\log n)^2\bigr),
\]
the following assertions hold simultaneously:
\begin{enumerate}
\item Every \(x\in\C^n\) satisfying \(F_Dx=0\) obeys
\[
 \norm{x}_2
 \leq C\frac{(\log n)^{3/2}}{\sqrt n}\norm{x}_1.
\]
\item Every \(x\in\C^n\) satisfying \(F_{D^{\mathrm c}}x=0\) obeys
\[
 \norm{x}_2
 \leq C\frac{(\log n)^{3/2}}{\sqrt n}\norm{x}_1.
\]
\end{enumerate}
The constants are uniform over all odd integers, including composite
orders.
\end{prop}

\begin{proof}
We establish an independent-row restricted isometry estimate at the
sample size \(q=\lceil n/4\rceil\), then apply
Lemma~\ref{fourier:lem:paired-transfer}.
The finite approximation construction follows
\cite[Section~3, Claims~3.2--3.3, Lemmas~3.4--3.6, and
Theorems~3.1 and~3.7]{HR15v2}.
We supply the concentration estimates and the complete argument for
the parameters used here.
All constants denoted \(C_i,c_i,c_0>0\) below are absolute.

Write \(B=F\).
For distinct columns, their inner product is the geometric sum of a
nontrivial \(n\)-th root of unity; for a column paired with itself,
the inner product is \(n\).
Thus, for every \(n\),
\begin{equation}\label{fourier:eq:fourier-orthogonality}
 B^*B=n\Id,
 \qquad
 \frac1n\sum_{u=0}^{n-1}|(Bx)_u|^2=\norm{x}_2^2
 \quad(x\in\C^n).
\end{equation}

\medskip
\noindent\textbf{Step 1.}
We derive the signed and complex concentration estimates used in
the approximation.

Let \(Y_1,\ldots,Y_s\) be independent identically distributed real
variables in \([-a,a]\), where \(a>0\).
Write
\(\mu=\E Y_\ell\), \(\mu_{\mathrm{abs}}=\E|Y_\ell|\), and
\(\overline Y=s^{-1}\sum_\ell Y_\ell\).
Apply Lemma~\ref{fourier:lem:mixed-concentration} to the positive
parts and to the negative parts, with relative parameter
\(0<\rho\leq1/2\) and additive parameter \(b/2\), where \(b>0\).
The union bound gives
\begin{equation}\label{fourier:eq:signed-concentration}
 \Prob\{|\overline Y-\mu|>\rho\mu_{\mathrm{abs}}+b\}
 \leq4\exp\left(-\frac{s\rho b}{4a}\right).
\end{equation}
The positive and negative parts of a single sample need not be
independent.

For \(0<d\leq2a\), take \(\rho=d/(4a)\) and \(b=3d/4\).
Since \(\mu_{\mathrm{abs}}\leq a\), the threshold in
\eqref{fourier:eq:signed-concentration} is at most \(d\), and
\begin{equation}\label{fourier:eq:bounded-real-concentration}
 \Prob\{|\overline Y-\mu|>d\}
 \leq4\exp\left(-\frac{3sd^2}{64a^2}\right).
\end{equation}
For \(d>2a\), the event is empty.
If instead the \(Y_\ell\) are complex with \(|Y_\ell|\leq a\),
a complex deviation exceeding \(d\) requires a real or imaginary
deviation exceeding \(d/\sqrt2\).
Applying \eqref{fourier:eq:bounded-real-concentration} to these two
parts gives
\begin{equation}\label{fourier:eq:bounded-complex-concentration}
 \Prob\{|\overline Y-\E Y_\ell|>d\}
 \leq8\exp\left(-c_0\frac{sd^2}{a^2}\right)
 \qquad(d>0).
\end{equation}
The reverse triangle inequality gives the same bound for
\(\bigl||\overline Y|-|\E Y_\ell|\bigr|>d\).

\medskip
\noindent\textbf{Step 2.}
We construct finite approximants at each scale and count the
associated nonnegative vectors.

Fix \(0<\rho\leq1/100\) and \(0<\eta<1/4\).
Set
\begin{equation}\label{fourier:eq:net-parameters}
 t=\left\lceil\log_2\frac1\eta\right\rceil,
 \qquad
 r=\left\lceil\log_2\frac1{\rho^2}\right\rceil,
 \qquad
 \gamma=\frac{\eta}{2t}.
\end{equation}
For the moment, \(Q\) is any multiset of \(q\) row indices.
For \(1\leq\ell\leq t+r\), put
\[
 s_\ell=\left\lceil C_1\,2^\ell\log\frac1\gamma\right\rceil.
\]
Here \(C_1\) is a sufficiently large absolute constant.
Let \(\cG_\ell\) contain all vectors
\[
 g=\frac{\sqrt2}{s_\ell}
       \sum_{a=1}^{s_\ell}\mathrm{i}^{v_a}B_{:,j_a},
 \qquad
 j_a\in[n],\quad v_a\in\{0,1,2,3\}.
\]
Their coordinates have modulus at most \(\sqrt2\).
Counting ordered choices of the pairs \((j_a,v_a)\) gives
\(|\cG_\ell|\leq(4n)^{s_\ell}\).
For \(n\geq3\) and \(\gamma<1/2\), this implies
\begin{equation}\label{fourier:eq:net-cardinality}
 \log|\cG_\ell|
 \leq C_2\,2^\ell\log n\log\frac1\gamma.
\end{equation}

Fix \(x\in\C^n\) with \(\norm{x}_1=1\).
We can choose nonnegative weights \(p_{j,v}\) such that
\[
 \sum_{v=0}^3p_{j,v}=|x_j|,
 \qquad
 \sqrt2\sum_{v=0}^3p_{j,v}\mathrm{i}^v=x_j.
\]
For this construction, assign the positive and negative real and
imaginary parts of \(x_j/\sqrt2\) to their corresponding phases.
The assigned mass is at most \(|x_j|\), because
\[
 |\operatorname{Re}x_j|+|\operatorname{Im}x_j|
 \leq\sqrt2|x_j|.
\]
Split the remaining mass equally between two opposite phases.
Their contributions cancel, giving the required weights.

Since \(\sum_{j,v}p_{j,v}=1\), the random vector
\(\sqrt2\,\mathrm{i}^vB_{:,j}\), sampled with these probabilities,
has expectation \(Bx\).
Average \(s_\ell\) independent copies to obtain \(g\in\cG_\ell\).
For each fixed row \(u\), apply
\eqref{fourier:eq:bounded-complex-concentration} with
\(a=\sqrt2\) and \(d=2^{-\ell/2}\).
Once \(C_1\) is large enough, the probability of failure of
\begin{equation}\label{fourier:eq:coordinate-approximation}
 \bigl||(Bx)_u|-|g_u|\bigr|\leq2^{-\ell/2}
\end{equation}
is at most \(\gamma/4\).

The expected failed fraction is at most \(\gamma/4\) both over all
rows and over \(Q\), counting repetitions.
For each fraction, Markov's inequality bounds by \(1/4\) the
probability that it exceeds \(\gamma\).
The union of these two events has probability at most \(1/2\).
Hence some \(g\in\cG_\ell\) satisfies
\eqref{fourier:eq:coordinate-approximation} outside a
\(\gamma\)-fraction of each row multiset.
This auxiliary sampling proves existence for every fixed \(x,Q,\ell\);
it imposes no independence condition on \(x\) and \(Q\).

For each tuple
\[
 (g_{1+r},\ldots,g_{t+r})
 \in\cG_{1+r}\times\cdots\times\cG_{t+r},
\]
assign \(u\) to \(\mathcal B_j\) if \(j\) is the first index in
\(\{1,\ldots,t\}\) for which
\[
 |g_{j+r,u}|\geq2\cdot2^{-j/2}.
\]
Rows with no such index remain unassigned.
The sets \(\mathcal B_j\) are disjoint.
Define
\[
 h_{j,u}
 =\min\left\{
 |g_{j+r,u}|^2\ind_{\{u\in\mathcal B_j\}},
 9\cdot2^{-j}
 \right\},
\]
and let \(\mathcal H_j\) be the set of all resulting vectors \(h_j\).
Only the first \(j\) members of the tuple determine \(h_j\).
Using \eqref{fourier:eq:net-cardinality}, the geometric sum over
these members, and \(2^r\leq2\rho^{-2}\), we obtain
\begin{equation}\label{fourier:eq:h-cardinality}
 \log|\mathcal H_j|
 \leq C_3\rho^{-2}2^j\log n\log\frac1\gamma.
\end{equation}
The constant \(C_3\) is independent of \(t,\rho,\eta,n\).

\medskip
\noindent\textbf{Step 3.}
We control all the finite approximants on one event for independent
row sampling.

Now take \(Q=(J_1,\ldots,J_q)\), where the \(J_a\) are independent
uniform row indices.
For fixed \(h\in\mathcal H_j\), the variables \(h_{J_a}\) are
independent, belong to \([0,9\cdot2^{-j}]\), and have common mean
\[
 \mu_h=\frac1n\sum_{u=0}^{n-1}h_u.
\]
Let \(\widetilde\eta=\eta/(2t)\).
Lemma~\ref{fourier:lem:mixed-concentration}, with relative error
\(\rho\) and additive error \(\widetilde\eta\), gives
\begin{equation}\label{fourier:eq:h-concentration}
 \Prob\left\{
 \left|\frac1q\sum_{a=1}^qh_{J_a}-\mu_h\right|
 >\rho\mu_h+\widetilde\eta
 \right\}
 \leq
 2\exp\left(-\frac{2^jq\rho\widetilde\eta}{18}\right).
\end{equation}
This bound applies when \(\mu_h=0\) and for every ratio
\(\widetilde\eta/\mu_h\).

Combining \eqref{fourier:eq:h-cardinality} with
\eqref{fourier:eq:h-concentration}, a union bound shows that
\begin{equation}\label{fourier:eq:finite-net-sample-threshold}
 q\geq
 C_4\rho^{-3}\widetilde\eta^{-1}
 \log n\log\frac1\gamma
\end{equation}
makes the probability of failure for any \(h\in\mathcal H_j\) and
any \(1\leq j\leq t\) at most
\begin{equation}\label{fourier:eq:finite-net-failure}
 \begin{aligned}
 &\sum_{j=1}^t
 2\exp\left(-c_1\rho^{-2}2^j\log n\log\frac1\gamma\right)\\
 &\hspace{2em}\leq
 \exp\left(-c_2\rho^{-2}\log n\log\frac1\gamma\right)
 \end{aligned}
\end{equation}
for all sufficiently large \(n\).
For the last inequality, extend the sum to \(j\geq1\) and use
\(2^j\geq j+1\); reducing \(c_2\) absorbs the prefactor of the
resulting geometric series.
Thus \eqref{fourier:eq:finite-net-sample-threshold} permits every
integer \(q\) above the displayed threshold.

Write \(\operatorname{avg}_{[n]}\) for the average over all rows and
\(\operatorname{avg}_{Q}\) for the sample average.
On the complementary event in
\eqref{fourier:eq:finite-net-failure}, summing the individual
deviation bounds over \(j\) gives, for every allowable tuple,
\begin{equation}\label{fourier:eq:h-sum-comparison}
 \left|
 \operatorname{avg}_{Q}\sum_{j=1}^th_j
 -
 \operatorname{avg}_{[n]}\sum_{j=1}^th_j
 \right|
 \leq
 \rho\operatorname{avg}_{[n]}\sum_{j=1}^th_j+\frac{\eta}{2}.
\end{equation}
Here \(\sum_{j=1}^t\widetilde\eta=\eta/2\), and the supports of the
\(h_j\) are disjoint.

\medskip
\noindent\textbf{Step 4.}
We pass from this finite event to a quadratic estimate for every
complex vector.

For any \(x\) with \(\norm{x}_1=1\), choose each \(g_{j+r}\) using
\eqref{fourier:eq:coordinate-approximation}.
Call a row good if all \(t\) approximation inequalities hold there.
The bad fraction is at most \(t\gamma\), both in \([n]\) and in \(Q\).

At a good row assigned to \(\mathcal B_j\), the defining threshold
gives \(|g_{j+r,u}|\geq2\cdot2^{-j/2}\).
We also have
\[
 |g_{j+r,u}|\leq3\cdot2^{-j/2}.
\]
For \(j=1\), this follows from \(|g_{1+r,u}|\leq\sqrt2\).
For \(j\geq2\), the threshold at \(j-1\) failed.
The approximation inequalities at \(j-1\) and \(j\), together with
\(2^{-r/2}\leq\rho\), give
\[
 |g_{j+r,u}|
 <
 \bigl[2\sqrt2+(\sqrt2+1)\rho\bigr]2^{-j/2}
 \leq3\cdot2^{-j/2},
\]
where the last inequality uses \(\rho\leq1/100\).
Thus the truncation in \(h_{j,u}\) has no effect at a good assigned
row, and every other \(h_{\ell,u}\) vanishes there.
Moreover,
\[
 \bigl||(Bx)_u|-|g_{j+r,u}|\bigr|
 \leq\rho2^{-j/2}
 \leq\frac{\rho}{2}|g_{j+r,u}|.
\]
Squaring these bounds yields
Writing $a=|(Bx)_u|$ and $b=|g_{j+r,u}|$, the preceding bound gives
$a\leq(1+\rho/2)b$ and $|a-b|\leq(\rho/2)b$. Hence
\[
 |a^2-b^2|=|a-b|(a+b)
 \leq \frac{\rho}{2}b\,(2+\rho/2)b
 \leq 3\rho b^2,
\]
where we used $0<\rho\leq1/100$. At a good assigned row,
$h_{j,u}=b^2$ and all other $h_{\ell,u}$ vanish. \begin{equation}\label{fourier:eq:good-assigned-row}
 \left|| (Bx)_u|^2-\sum_{j=1}^th_{j,u}\right|
 \leq3\rho\sum_{j=1}^th_{j,u}.
\end{equation}

At a good unassigned row, the last threshold fails.
Consequently,
\[
 |(Bx)_u|<3\cdot2^{-t/2}\leq3\sqrt\eta,
 \qquad
 \sum_{j=1}^th_{j,u}=0.
\]
At a bad row, \(|(Bx)_u|^2\leq1\), since \(\norm{x}_1=1\) and
\(|B_{uj}|=1\).
At most one \(h_{j,u}\) is nonzero, and \(|g_{j+r,u}|^2\leq2\);
hence \(\sum_jh_{j,u}\leq2\).
The absolute difference between these two nonnegative quantities is
therefore at most \(2\).

For \(P=Q\) or \(P=[n]\), write
\[
 A_P=\operatorname{avg}_{P}|Bx|^2,
 \qquad
 H_P=\operatorname{avg}_{P}\sum_{j=1}^th_j.
\]
Averaging \eqref{fourier:eq:good-assigned-row} and the bounds for
unassigned and bad rows gives
\begin{equation}\label{fourier:eq:deterministic-approximation}
 |A_P-H_P|
 \leq3\rho H_P+9\eta+2t\gamma
 =3\rho H_P+10\eta.
\end{equation}
Together with \eqref{fourier:eq:h-sum-comparison}, this implies
\begin{equation}\label{fourier:eq:uniform-quadratic-approximation}
 |A_Q-A_{[n]}|\leq100\rho A_{[n]}+100\eta.
\end{equation}
To check the constants, first use the two preceding estimates to
obtain
\[
 H_{[n]}\leq\frac{A_{[n]}+10\eta}{1-3\rho},
 \qquad
 H_Q\leq(1+\rho)H_{[n]}+\frac{\eta}{2}.
\]
The triangle inequality then bounds \(|A_Q-A_{[n]}|\) by
\[
 (7\rho+3\rho^2)H_{[n]}+(20.5+1.5\rho)\eta.
\]
Substituting the bound for \(H_{[n]}\) makes this no larger than the
right side of \eqref{fourier:eq:uniform-quadratic-approximation}
for \(0<\rho\leq1/100\).

The event is simultaneous over all \(x\) with \(\norm{x}_1=1\).
Although the selected tuple can depend on both \(x\) and \(Q\),
\eqref{fourier:eq:h-sum-comparison} holds for every tuple from the
finite sets.
Scaling \(x\), and treating \(x=0\) directly, extends the estimate
to all complex vectors, with additive term
\(100\eta\norm{x}_1^2\).

\medskip
\noindent\textbf{Step 5.}
We choose the parameters to obtain restricted isometry at
\(q=\lceil n/4\rceil\).

From \eqref{fourier:eq:net-parameters},
\[
 t=O\bigl(\log(1/\eta)\bigr),
 \qquad
 \log(1/\gamma)=O\bigl(\log(1/\eta)\bigr)
 \quad(0<\eta<1/4).
\]
Since \(\widetilde\eta=\eta/(2t)\),
condition \eqref{fourier:eq:finite-net-sample-threshold} follows from
\begin{equation}\label{fourier:eq:simplified-sample-threshold}
 q\geq
 C_5\rho^{-3}\eta^{-1}\log n
 \left[\log\frac1\eta\right]^2.
\end{equation}
Also \(\log(1/\gamma)\geq\log(1/\eta)\), so the failure bound in
\eqref{fourier:eq:finite-net-failure} is at most
\[
 \exp\left(-c_2\rho^{-2}\log n\log\frac1\eta\right).
\]

Fix a sufficiently small absolute \(\theta\in(0,1/3)\), and set
\[
 \rho=\frac{\theta}{200},
 \qquad
 \eta=\frac{\theta}{200k}.
\]
For any complex vector supported on at most \(k\) coordinates,
Cauchy--Schwarz gives \(\norm{x}_1^2\leq k\norm{x}_2^2\).
Using \eqref{fourier:eq:fourier-orthogonality} in the homogeneous
form of \eqref{fourier:eq:uniform-quadratic-approximation} therefore
gives, simultaneously for these vectors,
\begin{equation}\label{fourier:eq:independent-row-rip}
 (1-\theta)\norm{x}_2^2
 \leq\frac1q\sum_{a=1}^q|(Fx)_{J_a}|^2
 \leq(1+\theta)\norm{x}_2^2.
\end{equation}
The failure probability is at most
\begin{equation}\label{fourier:eq:rip-failure}
 \exp\left(-c_3\theta^{-2}\log n\log\frac{k}{\theta}\right),
\end{equation}
provided
\begin{equation}\label{fourier:eq:rip-sample-threshold}
 q\geq
 C_6\theta^{-4}k
 \left[\log\frac{k}{\theta}\right]^2\log n.
\end{equation}
Indeed, substitution into
\eqref{fourier:eq:simplified-sample-threshold} gives this condition,
after absorbing fixed numerical factors into \(C_6\).
The same substitution in the preceding failure estimate gives
\eqref{fourier:eq:rip-failure}.

Thus the rows \(F_{J_a,:}/\sqrt q\), sampled independently with
replacement, satisfy \eqref{fourier:eq:independent-row-rip} for
complex \(k\)-sparse vectors.
Equivalently, these are rows of the unitary matrix \(F/\sqrt n\)
multiplied by \(\sqrt{n/q}\).
This is the normalization required by
Lemma~\ref{fourier:lem:paired-transfer}.

Choose
\begin{equation}\label{fourier:eq:k-choice}
 q=\left\lceil\frac n4\right\rceil,
 \qquad
 k=\left\lfloor\frac{\alpha n}{(\log n)^3}\right\rfloor,
\end{equation}
where \(\alpha>0\) is fixed small enough that
\(C_6\theta^{-4}\alpha\leq1/8\).
For all sufficiently large \(n\),
\[
 k\geq\frac{\alpha n}{2(\log n)^3},
 \qquad
 n^{1/2}\leq\frac{k}{\theta}\leq n.
\]
In particular, \(1\leq k\leq n\).
The sufficient sample size in
\eqref{fourier:eq:rip-sample-threshold} satisfies
\[
 C_6\theta^{-4}k
 \left[\log\frac{k}{\theta}\right]^2\log n
 \leq C_6\theta^{-4}\alpha n
 \leq\frac n8
 \leq q.
\]
Since \(\log(k/\theta)\geq\frac12\log n\), the probability in
\eqref{fourier:eq:rip-failure} is at most
\[
 \exp\bigl(-c_4(\log n)^2\bigr)
\]
for an absolute \(c_4>0\).

\medskip
\noindent\textbf{Step 6.}
We transfer the simultaneous estimate to both paired sets.

Apply Lemma~\ref{fourier:lem:paired-transfer} with restricted
isometry constant \(\theta\) and failure probability
\(\exp(-c_4(\log n)^2)\).
Outside an event of probability at most
\[
 \sqrt2\left[e^{-c_4(\log n)^2}+e^{-c_*n}\right],
\]
every \(x\in\ker F_D\) satisfies
\[
 \norm{x}_2
 \leq D_\theta\frac{\norm{x}_1}{\sqrt k},
 \qquad
 D_\theta=1+\sqrt{\frac{1+\theta}{1-\theta}}.
\]
The transfer includes repeated sampled rows and excludes the zero
row through its conditioning event.
Because its kernel conclusion is simultaneous, \(x\) may depend
on the sampled graph.

The lower bound for \(k\) in \eqref{fourier:eq:k-choice} gives
\[
 \frac1{\sqrt k}
 \leq\sqrt{\frac2\alpha}\,
       \frac{(\log n)^{3/2}}{\sqrt n}.
\]
Thus the first assertion holds with
\(C=D_\theta\sqrt{2/\alpha}\).
After decreasing an absolute constant \(c_5>0\) and increasing the
lower threshold for \(n\), its exceptional probability is at most
\(\exp(-c_5(\log n)^2)\).

The signs \(-\sigma_1,\ldots,-\sigma_{(n-1)/2}\) are also
independent fair signs, and their negative set is \(D^{\mathrm c}\).
Applying the proved bound to this sign sequence gives the second
assertion with the same constants.
A union bound makes the two conclusions simultaneous.
Its factor of two is absorbed by decreasing the exponent constant
and increasing \(n_0\).

The orthogonality calculation and the independent-row argument
hold for every integer \(n\).
Only the reflected-pair transfer requires \(n\) to be odd.
One fixed \(n_0\) absorbs all the preceding thresholds, proving the
proposition uniformly over all sufficiently large odd orders.
\end{proof}

\noindent

\medskip
\noindent

\medskip
\noindent

\section{Sign-cube tools}\label{sec:sign-cube-tools}

In this section, we prove the moment and sign-flip estimates used to compare polynomial observables with Gaussian variables, together with the moment bounds for one polynomial Fourier-feedback layer.

\subsection{Walsh moments and sign flips}

\begin{notation}\label{trees:notation:sign-cube}
Let \(\epsilon_1,\ldots,\epsilon_m\) be independent fair signs, where \(m\geq1\). For \(B\subseteq\{1,\ldots,m\}\), write
\[
 \chi_B=\prod_{v\in B}\epsilon_v,
 \qquad \chi_\varnothing=1.
\]
The Walsh expansion of a real polynomial on the sign cube is
\(U=\sum_Ba_B\chi_B\). Its degree is the largest \(\abs B\) with \(a_B\neq0\), and
\[
 D_vU=\sum_{B\ni v}a_B\chi_{B\setminus\{v\}}.
\]
For a random variable \(X\), write
\(\norm{X}_{L^r}=(\E\abs X^r)^{1/r}\).
For deterministic vectors, \(\norm{\cdot}_2\) denotes the Euclidean norm.
Independence of the signs gives
\[
 \E[\chi_B\chi_C]=\ind_{\{B=C\}},
 \qquad
 \E U^2=\sum_Ba_B^2.
\]
\end{notation}

\begin{lem}[Sign hypercontractivity]\label{trees:lem:sign-hypercontractivity}
Let \(U\) be a real Walsh polynomial of degree at most \(D\). For every integer \(p\geq1\),
\begin{equation}
 \norm{U}_{L^{2p}}
 \leq(2p-1)^{D/2}\norm{U}_{L^2}.
 \label{trees:eq:sign-hypercontractivity}
\end{equation}
\end{lem}

\begin{proof}
The case \(p=1\) is equality. Assume that \(p\geq2\), and put
\(\rho=(2p-1)^{-1/2}\). Let \(T_\rho\) multiply the coefficient of \(\chi_B\) by \(\rho^{\abs B}\).

For \(0\leq j\leq p\),
\[
 \frac{\binom{2p}{2j}}{\binom pj}
 =
 \prod_{\nu=0}^{j-1}
 \frac{2p-2\nu-1}{2\nu+1}
 \leq(2p-1)^j,
\]
where the empty product is one. Comparing the even terms in the binomial expansions therefore gives
\[
 \E_\epsilon(a+\rho\epsilon c)^{2p}
 \leq
 \bigl(a^2+(2p-1)\rho^2c^2\bigr)^p
 \qquad(a,c\in\R).
\]

We prove \(\norm{T_\rho U}_{L^{2p}}\leq\norm{U}_{L^2}\) by induction on the number of signs. The assertion for a constant is equality. Write
\(U=G+\epsilon_mH\), where \(G,H\) depend on the first \(m-1\) signs. Apply the scalar inequality conditionally on these signs. Since
\((2p-1)\rho^2=1\), Minkowski's inequality and the induction hypothesis give
\begin{align*}
 \norm{T_\rho U}_{L^{2p}}^2
 &\leq
 \norm{(T_\rho G)^2+(T_\rho H)^2}_{L^p}\\
 &\leq
 \norm{T_\rho G}_{L^{2p}}^2+
 \norm{T_\rho H}_{L^{2p}}^2\\
 &\leq
 \norm{G}_{L^2}^2+\norm{H}_{L^2}^2
 =\norm{U}_{L^2}^2.
\end{align*}
Finally, set \(V=\sum_B\rho^{-\abs B}a_B\chi_B\). Then
\(T_\rho V=U\), while Walsh orthogonality gives
\[
 \norm{V}_{L^2}^2
 =\sum_B\rho^{-2\abs B}a_B^2
 \leq\rho^{-2D}\norm{U}_{L^2}^2.
\]
Applying the preceding inequality to \(V\) proves
\eqref{trees:eq:sign-hypercontractivity}.
\end{proof}

\begin{lem}[Weighted fourth moments and jump bounds]
\label{trees:lem:weighted-jumps}
Let \(P=\sum_Ba_B\chi_B\) be a real Walsh polynomial of degree at most \(D\). Then
\begin{equation}
 \E P^4
 \leq
 \left(\sum_B3^{\abs B}a_B^2\right)^2
 \leq3^{2D}(\E P^2)^2.
 \label{trees:eq:weighted-fourth-moment}
\end{equation}
Let \(Y_1,\ldots,Y_q\) be homogeneous real Walsh polynomials of positive degrees \(d_1,\ldots,d_q\), and put
\(I_{i,v}=\E(D_vY_i)^2\). Then
\begin{equation}
 \E\abs{D_vY_i}^3
 \leq3^{d_i-1}I_{i,v}^{3/2},
 \qquad
 \E(D_vY_i)^4
 \leq3^{2d_i-2}I_{i,v}^2.
 \label{trees:eq:additive-derivative-moments}
\end{equation}
Choose \(J\) uniformly from \(\{1,\ldots,m\}\), independently of the signs, and flip \(\epsilon_J\). If
\[
 \Delta_i=Y_i'-Y_i=-2\epsilon_JD_JY_i,
 \qquad
 \lambda_i=\frac{2d_i}{m},
\]
then
\begin{equation}
 \begin{aligned}
 \lambda_i^{-1}\E[\abs{\Delta_i}\Delta_j^2]
 &=\frac4{d_i}\sum_v
   \E[\abs{D_vY_i}(D_vY_j)^2]\\
 &\leq
 \frac4{d_i}3^{d_j-1}
 \sqrt{\max_v I_{i,v}}\sum_v I_{j,v}.
 \end{aligned}
 \label{trees:eq:mixed-degree-jump}
\end{equation}
\end{lem}

\begin{proof}
For the first inequality in \eqref{trees:eq:weighted-fourth-moment}, induct on the number of signs. The constant case is equality. Write
\(P=A+\epsilon_mB\), and denote the weighted coefficient sums of \(A,B\) by
\[
 \mathcal W(A)=\sum_S3^{\abs S}a_S^2,
 \qquad
 \mathcal W(B)=\sum_S3^{\abs S}b_S^2.
\]
The induction hypothesis and Cauchy--Schwarz imply
\[
 \E A^4\leq\mathcal W(A)^2,\qquad
 \E B^4\leq\mathcal W(B)^2,\qquad
 \E A^2B^2\leq\mathcal W(A)\mathcal W(B).
\]
Averaging first over \(\epsilon_m\) gives
\begin{align*}
 \E P^4
 &=\E A^4+6\E A^2B^2+\E B^4\\
 &\leq
 \mathcal W(A)^2+
 6\mathcal W(A)\mathcal W(B)+\mathcal W(B)^2\\
 &\leq\bigl(\mathcal W(A)+3\mathcal W(B)\bigr)^2.
\end{align*}
The last weighted sum is the one for \(P\). The second inequality in
\eqref{trees:eq:weighted-fourth-moment} follows from \(\abs B\leq D\) and Walsh orthogonality.

The polynomial \(D_vY_i\) has degree \(d_i-1\). Applying
\eqref{trees:eq:weighted-fourth-moment} proves its fourth-moment bound. Cauchy--Schwarz then gives
\[
 \E\abs{D_vY_i}^3
 \leq
 \sqrt{\E(D_vY_i)^2\,\E(D_vY_i)^4},
\]
which proves the first inequality in
\eqref{trees:eq:additive-derivative-moments}.

Conditioning on \(J\) proves the equality in
\eqref{trees:eq:mixed-degree-jump}. For each \(v\), Cauchy--Schwarz gives
\[
 \E[\abs{D_vY_i}(D_vY_j)^2]
 \leq
 \sqrt{I_{i,v}\,\E(D_vY_j)^4}
 \leq3^{d_j-1}\sqrt{I_{i,v}}\,I_{j,v}.
\]
Summing and taking the maximum of the first influence factor proves the jump bound.
\end{proof}

\subsection{Fourier integration and mixed-degree comparison}

We first pass from characteristic functions to absolute moments, and then prove the sign-flip comparison with its degree dependence explicit.

\begin{lem}[Absolute moments from characteristic functions]
\label{trees:lem:absolute-moment-fourier}
For every integrable real random variable \(X\),
\begin{equation}
 \E\abs X
 =
 \frac2\pi\int_0^\infty
 \frac{1-\E\cos(tX)}{t^2}\,dt.
 \label{trees:eq:absolute-moment-fourier}
\end{equation}
Let \(X'\) be another integrable real random variable. Assume that \(a,b\geq0\) satisfy
\[
 \abs{\E e^{\mathrm{i}tX}-\E e^{\mathrm{i}tX'}}
 \leq at^2+b\abs t^3
 \qquad(t\in\R).
\]
Then, for every \(T>0\),
\begin{equation}
 \abs{\E\abs X-\E\abs{X'}}
 \leq
 \frac2\pi\left(aT+\frac b2T^2+\frac2T\right).
 \label{trees:eq:absolute-moment-truncation}
\end{equation}
For a standard real normal variable \(N\) and \(\sigma\geq0\),
\[
 \E e^{\mathrm{i}t\sigma N}=e^{-\sigma^2t^2/2},
 \qquad
 \E\abs{\sigma N}=\sqrt{\frac2\pi}\,\sigma.
\]
\end{lem}

\begin{proof}
Integration by parts gives
\[
 \int_0^\infty\frac{1-\cos u}{u^2}\,du
 =
 \int_0^\infty\frac{\sin u}{u}\,du.
\]
The boundary terms vanish because \(1-\cos u=O(u^2)\) at zero and is bounded at infinity. To evaluate the last improper integral, let \(\eta>0\). The identity
\(\sin u/u=\int_0^1\cos(su)\,ds\) and Fubini's theorem give
\[
 \int_0^\infty e^{-\eta u}\frac{\sin u}{u}\,du
 =
 \int_0^1\frac{\eta}{\eta^2+s^2}\,ds
 =
 \arctan(1/\eta).
\]
Integration by parts bounds the tail from \(R\) to infinity by \(C/R\), uniformly for \(\eta\geq0\), since \(e^{-\eta u}/u\) decreases to zero. We may therefore let \(\eta\) decrease to zero, obtaining \(\pi/2\). Scaling by \(\abs x\) proves
\[
 \int_0^\infty\frac{1-\cos(tx)}{t^2}\,dt
 =\frac{\pi}{2}\abs x
 \qquad(x\in\R).
\]
Tonelli's theorem now proves
\eqref{trees:eq:absolute-moment-fourier}.

Subtract the two absolute-moment identities. On \([0,T]\), the assumed bound contributes at most
\(aT+bT^2/2\) before multiplication by \(2/\pi\). On \([T,\infty)\), use
\(\abs{\E\cos(tX)-\E\cos(tX')}\leq2\). Its contribution is at most \(2/T\), which proves
\eqref{trees:eq:absolute-moment-truncation}.

For \(\sigma>0\), Gaussian integration by parts gives the differential equation
\[
 \frac{d}{dt}\E e^{\mathrm{i}t\sigma N}
 =
 -\sigma^2t\,\E e^{\mathrm{i}t\sigma N}.
\]
Its value at zero is one, so its solution is \(e^{-\sigma^2t^2/2}\). Direct integration of the normal density gives
\[
 \E\abs{\sigma N}
 =
 \frac{2\sigma}{\sqrt{2\pi}}
 \int_0^\infty x e^{-x^2/2}\,dx
 =
 \sqrt{\frac2\pi}\,\sigma.
\]
Both identities also hold when \(\sigma=0\).
\end{proof}

\begin{thm}[Mixed-degree Gaussian comparison]
\label{trees:thm:sign-cube-comparison}
\label{trees:thm:additive-gaussian-comparison}
Let \(Y_1,\ldots,Y_q\) be homogeneous real Walsh polynomials in \(m\) independent fair signs, of degrees \(d_i\geq1\). Define
\begin{equation}
 G_{ij}=\frac1{d_i}\sum_v(D_vY_i)(D_vY_j),
 \qquad
 I_{i,v}=\E(D_vY_i)^2.
 \label{trees:eq:additive-Gij}
\end{equation}
Put
\[
 a_i=\max_vI_{i,v},\qquad
 b_i=\sum_vI_{i,v},\qquad
 c_{ij}=\sqrt{\Var(G_{ij})}.
\]
For \(\alpha=(\alpha_1,\ldots,\alpha_q)\in\R^q\), set
\[
 Y_\alpha=\sum_i\alpha_iY_i,\qquad
 \sigma_\alpha^2=\E Y_\alpha^2,
\]
and
\[
 \mathcal A_\alpha
 =\sum_{i,j}\abs{\alpha_i\alpha_j}c_{ij},
 \qquad
 \mathcal B_\alpha
 =q\sum_{i,j}
 \frac{\abs{\alpha_i}\alpha_j^2}{d_i}
 3^{d_j-1}\sqrt{a_i}\,b_j.
\]
Then \(\E G_{ij}=\E[Y_iY_j]\), and
\begin{equation}
 \abs{\E e^{\mathrm{i}tY_\alpha}
       -e^{-\sigma_\alpha^2t^2/2}}
 \leq
 \frac{\mathcal A_\alpha}{2}t^2
 +\frac{\mathcal B_\alpha}{3}\abs t^3
 \qquad(t\in\R).
 \label{trees:eq:mixed-degree-cf}
\end{equation}
For every \(T>0\),
\begin{equation}
 \abs{\E\abs{Y_\alpha}
       -\sqrt{\frac2\pi}\sigma_\alpha}
 \leq
 \frac{\mathcal A_\alpha}{\pi}T
 +\frac{\mathcal B_\alpha}{3\pi}T^2
 +\frac4{\pi T}.
 \label{trees:eq:mixed-degree-absolute}
\end{equation}
These estimates require no positive lower bound on \(\sigma_\alpha\).

For the following additive coefficient class, the estimates hold with constants uniform in the odd modulus. Fix \(q\geq1\), \(M>0\), and degrees \(d_1,\ldots,d_q\), each equal to \(1\) or at least \(3\). Let \(n\geq3\) be odd and \(m=(n-1)/2\). For \(d\geq3\), let \(\mathcal H_{n,d}\) consist of the \(d\)-element subsets \(e\subseteq\{1,\ldots,m\}\) admitting signs \(s_v\in\{-1,1\}\) such that
\[
 \sum_{v\in e}s_vv=0\pmod n.
\]
Assume that the deterministic real coefficients satisfy
\[
 \begin{array}{lll}
 Y_i=\displaystyle\sum_{e\in\mathcal H_{n,d_i}}w_{i,e}\chi_e,
 &\abs{w_{i,e}}\leq Mn^{-(d_i-1)/2},
 &d_i\geq3,\\[6pt]
 Y_i=\displaystyle\sum_vw_{i,v}\epsilon_v,
 &\abs{w_{i,v}}\leq Mn^{-1/2},
 &d_i=1.
 \end{array}
\]
There is a constant \(C\), depending only on \(q,M\) and the fixed degrees, such that
\begin{equation}
 \E G_{ij}=\E[Y_iY_j],
 \qquad
 \Var(G_{ij})\leq\frac Cn.
 \label{trees:eq:additive-Gij-bound}
\end{equation}
For every fixed \(\alpha\in\R^q\), there is a constant \(C_\alpha\), depending only on these parameters and \(\alpha\), such that
\begin{align}
 \abs{\E e^{\mathrm{i}tY_\alpha}
       -e^{-\sigma_\alpha^2t^2/2}}
 &\leq C_\alpha\frac{t^2+\abs t^3}{\sqrt n},
 \label{trees:eq:additive-general-cf}\\
 \abs{\E\abs{Y_\alpha}
       -\sqrt{\frac2\pi}\sigma_\alpha}
 &\leq C_\alpha n^{-1/6}.
 \label{trees:eq:additive-general-absolute}
\end{align}
The constants are uniform in \(n\) and in all permitted deterministic weights. These additive estimates concern fixed degrees; they assert neither a constant uniform in growing degree nor an approximation theorem for adaptive nonlinear iterations.
\end{thm}

\begin{proof}
We first prove the comparison in terms of the influences and covariance fluctuations. We then establish the required bounds for additive coefficients by counting solutions of congruences with determinant-one minors.

\noindent\textbf{Step 1. The sign flip gives the comparison with explicit parameters.}
Walsh orthogonality gives
\[
 \sum_v I_{i,v}=d_i\E Y_i^2,
 \qquad
 \E\sum_v(D_vY_i)(D_vY_j)
 =d_i\E[Y_iY_j].
\]
For the second identity, both sides vanish when \(d_i\neq d_j\). When the degrees agree, each coincident monomial is counted once for each member of its support. Hence \(\E G_{ij}=\E[Y_iY_j]\).

Flip a uniformly selected sign as in
Lemma~\ref{trees:lem:weighted-jumps}, and write
\[
 \Delta_i=-2\epsilon_JD_JY_i,\qquad
 \lambda_i=\frac{2d_i}{m},\qquad
 \Delta_\alpha=\sum_i\alpha_i\Delta_i.
\]
Every degree-\(d_i\) monomial changes sign under exactly \(d_i\) choices of \(J\). Thus
\[
 \E[\Delta_i\mid\epsilon]=-\lambda_iY_i.
\]
The original and flipped sign samples are exchangeable. For any bounded complex function \(f\), exchanging the samples gives
\[
 \E[\Delta_i f(Y_\alpha+\Delta_\alpha)]
 =-\E[\Delta_i f(Y_\alpha)].
\]
Combining this identity with the regression formula proves
\begin{equation}
 \E[Y_if(Y_\alpha)]
 =
 \frac1{2\lambda_i}
 \E\!\left[
 \Delta_i\{f(Y_\alpha+\Delta_\alpha)-f(Y_\alpha)\}
 \right].
 \label{trees:eq:additive-exchange}
\end{equation}

For real \(y,h,t\), integrating the second derivative of
\(u\mapsto e^{\mathrm{i}tu}\) twice gives
\[
 \abs{
 e^{\mathrm{i}t(y+h)}-e^{\mathrm{i}ty}
 -\mathrm{i}th\,e^{\mathrm{i}ty}}
 \leq\frac{t^2h^2}{2}.
\]
Also,
\[
 \E[\Delta_i\Delta_j\mid\epsilon]
 =\frac4m\sum_v(D_vY_i)(D_vY_j).
\]
Apply \eqref{trees:eq:additive-exchange} with
\(f(y)=e^{\mathrm{i}ty}\), multiply by \(\alpha_i\), and sum over \(i\). The linear terms yield
\begin{equation}
 \E[Y_\alpha e^{\mathrm{i}tY_\alpha}]
 =
 \mathrm{i}t\,\E[V_\alpha e^{\mathrm{i}tY_\alpha}]
 +\rho_\alpha(t),
 \qquad
 V_\alpha=\sum_{i,j}\alpha_i\alpha_jG_{ij}.
 \label{trees:eq:additive-Stein}
\end{equation}
Indeed, the factor \(1/(2\lambda_i)\) times \(4/m\) equals \(1/d_i\). The remainder satisfies
\[
 \abs{\rho_\alpha(t)}
 \leq
 \frac{t^2}{4}\sum_i
 \abs{\alpha_i}\lambda_i^{-1}
 \E[\abs{\Delta_i}\Delta_\alpha^2].
\]
Since
\(\Delta_\alpha^2\leq q\sum_j\alpha_j^2\Delta_j^2\),
\eqref{trees:eq:mixed-degree-jump} gives
\[
 \abs{\rho_\alpha(t)}\leq\mathcal B_\alpha t^2.
\]
The mean identities and Cauchy--Schwarz give
\[
 \E V_\alpha=\sigma_\alpha^2,\qquad
 \E\abs{V_\alpha-\sigma_\alpha^2}
 \leq\mathcal A_\alpha.
\]

Let \(\phi(t)=\E e^{\mathrm{i}tY_\alpha}\). Differentiation is valid on the finite sign space. It follows from
\eqref{trees:eq:additive-Stein} that
\[
 \abs{\phi'(t)+\sigma_\alpha^2t\phi(t)}
 \leq\mathcal A_\alpha\abs t+\mathcal B_\alpha t^2.
\]
For \(t\geq0\), multiply by \(e^{\sigma_\alpha^2t^2/2}\) and integrate from zero. The resulting kernel satisfies
\[
 \exp\!\left[-\frac{\sigma_\alpha^2}{2}(t^2-s^2)\right]
 \leq1
 \qquad(0\leq s\leq t).
\]
Consequently,
\[
 \abs{\phi(t)-e^{-\sigma_\alpha^2t^2/2}}
 \leq
 \int_0^t(\mathcal A_\alpha s+\mathcal B_\alpha s^2)\,ds
 =
 \frac{\mathcal A_\alpha}{2}t^2+
 \frac{\mathcal B_\alpha}{3}t^3.
\]
Complex conjugation handles negative \(t\). This proves
\eqref{trees:eq:mixed-degree-cf}, including \(\sigma_\alpha=0\).
Lemma~\ref{trees:lem:absolute-moment-fourier}, with
\(a=\mathcal A_\alpha/2\) and \(b=\mathcal B_\alpha/3\), proves
\eqref{trees:eq:mixed-degree-absolute}.

\medskip
\noindent\textbf{Step 2. Additive coefficients have bounded total influence and small individual influences.}
Ordering the vertices of an additive edge and choosing their orientations gives
\begin{equation}
 \abs{\mathcal H_{n,d}}\leq A_dn^{d-1},
 \qquad
 \#\{e\in\mathcal H_{n,d}:v\in e\}
 \leq A_dn^{d-2}.
 \label{trees:eq:additive-hyperedge-counts}
\end{equation}
For the first bound, choose \(d-1\) residues and all orientation signs. The final residue is determined by a coefficient \(1\) or \(-1\) in the zero-sum equation. For the second bound, fix the prescribed vertex and choose \(d-2\) further residues. Allowing arbitrary residues, repetitions, and all orderings enlarges both counts.

The weight bounds and Walsh orthogonality now give
\begin{equation}
 \max_v I_{i,v}\leq\frac An,
 \qquad
 \sum_v I_{i,v}=d_i\E Y_i^2\leq A,
 \label{trees:eq:additive-influences}
\end{equation}
where \(A\) depends only on \(M\) and the fixed degrees. If \(d_i=1\), the first estimate is \(w_{i,v}^2\leq M^2/n\), and summing over \(m\leq n/2\) proves the second.

\medskip
\noindent\textbf{Step 3. Modular counts bound the covariance fluctuations.}
First assume that \(r=d_i\geq3\) and \(s=d_j\geq3\). The unscaled covariance polynomial is
\begin{equation}
 \sum_v(D_vY_i)(D_vY_j)
 =
 \sum_{e,f}\abs{e\cap f}\,
 w_{i,e}w_{j,f}\chi_{e\symdiff f}.
 \label{trees:eq:additive-covariance-polynomial}
\end{equation}
Its constant terms have \(e=f\), which is possible only when \(r=s\). After removing those terms and squaring, Walsh orthogonality leaves quadruples
\((e_1,e_2,e_3,e_4)\) of sizes \((r,s,r,s)\) such that every label occurs an even number of times. In addition,
\[
 e_1\neq e_2,\quad e_1\cap e_2\neq\varnothing,
 \qquad
 e_3\neq e_4,\quad e_3\cap e_4\neq\varnothing.
\]
We prove the bound
\begin{equation}
 \#\{\text{such additive quadruples}\}
 \leq A_{r,s}n^{r+s-3}.
 \label{trees:eq:additive-quadruple-count}
\end{equation}

There are finitely many identification patterns among the \(2r+2s\) positions, with the number depending only on \(r,s\). Choose an orientation for each of the four edges; at most \(2^{2r+2s}\) choices cover all possibilities. Fix a pattern and these orientations. Let \(c\) be the number of labels occurring four times. Every other label occurs twice, so the number of distinct labels is
\[
 v=r+s-c.
\]
The four additive equations form a homogeneous congruence system on these \(v\) labels. Each column has either two or four nonzero entries, all equal to \(1\) or \(-1\).

If \(c=0\), regard each column as an edge joining its two supporting equation rows. The resulting multigraph on four rows is connected. It contains connections between rows \(1,2\) and between rows \(3,4\). If it were disconnected, its components would be exactly \(\{1,2\}\) and \(\{3,4\}\). Every label in \(e_1\) would then belong to \(e_2\), and conversely, contradicting \(e_1\neq e_2\). Choose three columns forming a spanning tree and delete one row. Removing a nonroot leaf successively shows that the resulting \(3\)-by-\(3\) minor has determinant \(1\) or \(-1\). Thus assigning the other \(v-3\) labels determines these three modulo \(n\). There are at most \(n^{v-3}=n^{r+s-3}\) solutions.

If \(c=1\), choose the four-support column and a two-support column. Such a two-support column exists because each row has at least three entries. Select one row in the support of the latter column and one outside its support. These rows and columns form a \(2\)-by-\(2\) minor with determinant \(1\) or \(-1\). The number of solutions is therefore at most
\[
 n^{v-2}=n^{r+s-3}.
\]
If \(c\geq2\), any nonzero entry determines one variable after the others are fixed. This bounds the number of solutions by
\[
 n^{v-1}\leq n^{r+s-3}.
\]
All the determining minors are units for every modulus. Summing over the patterns and orientations proves
\eqref{trees:eq:additive-quadruple-count}; restricting the labels to distinct representatives in \(\{1,\ldots,m\}\) can only decrease the count.

Each product of four weights in the variance expansion is bounded by
\[
 M^4n^{-(r-1)-(s-1)}
 =M^4n^{-(r+s-2)}.
\]
The two intersection factors are bounded in terms of \(r,s\). Hence
\eqref{trees:eq:additive-quadruple-count} gives
\[
 \Var\!\left(\sum_vD_vY_iD_vY_j\right)\leq\frac Cn.
\]
Dividing by \(d_i^2\) proves the required estimate for \(G_{ij}\).

If one degree is one and the other is \(s\geq3\), write the linear coefficients as \(u_v\). The unscaled covariance polynomial is
\[
 \sum_vu_vD_vY_j.
\]
For a fixed \((s-1)\)-element support \(B\), its coefficient is the sum over completions \(v\) for which \(B\cup\{v\}\) is additive. Each orientation determines \(v\), so there are at most \(2^s\) completions. The coefficient has magnitude at most
\[
 2^sM^2n^{-s/2}.
\]
There are at most \(n^{s-1}\) supports \(B\). Walsh orthogonality therefore bounds the variance by \(C/n\). Reversing the order of the two degrees changes only the fixed divisor \(d_i\). If both degrees are one, the covariance polynomial is deterministic. This proves
\eqref{trees:eq:additive-Gij-bound} in all cases.

By \eqref{trees:eq:additive-influences} and
\eqref{trees:eq:additive-Gij-bound},
\[
 \mathcal A_\alpha\leq C_\alpha n^{-1/2},
 \qquad
 \mathcal B_\alpha\leq C_\alpha n^{-1/2}.
\]
The factors \(q\), \(d_i^{-1}\), and \(3^{d_j-1}\) are retained explicitly in the definition of \(\mathcal B_\alpha\). Substitution into
\eqref{trees:eq:mixed-degree-cf} proves
\eqref{trees:eq:additive-general-cf}. Taking \(T=n^{1/6}\) in
\eqref{trees:eq:mixed-degree-absolute} gives
\[
 C_\alpha\frac{T+T^2}{\sqrt n}+\frac C T
 \leq C_\alpha n^{-1/6},
\]
which proves \eqref{trees:eq:additive-general-absolute}.
\end{proof}

\subsection{Polynomial Fourier feedback}

The polynomial replacement estimate below treats the one-, two-, and four-coordinate calculations in a common form.

\begin{lem}[Polynomial replacement]\label{trees:lem:polynomial-replacement}
Fix a positive integer \(s\) and a real polynomial \(A\) on \(\R^s\). Let \(n\geq3\), let \(u_1,\ldots,u_m\in\R^s\), and let \(\eta\in\R^s\). Assume that
\[
 \abs{\eta_j}\leq1,\qquad
 \abs{(u_a)_j}\leq\frac2{\sqrt n},\qquad
 \sum_a (u_a)_j^2\leq2
 \quad(1\leq j\leq s).
\]
Let \(\epsilon_a\) be independent fair signs and \(g_a\) independent standard real normal variables. Set
\[
 X=\eta+\sum_a u_a\epsilon_a,
 \qquad
 G=\eta+\sum_a u_ag_a.
\]
There are constants \(C_A,C_A'\), depending only on \(A\) and \(s\), such that
\begin{equation}
 \abs{\E A(X)-\E A(G)}
 \leq C_A\sum_a\norm{u_a}_2^4
 \leq\frac{C_A'}n.
 \label{trees:eq:empirical-replacement}
\end{equation}
Every fixed coordinate moment remains uniformly bounded throughout successive replacements of signs by normal variables.
\end{lem}

\begin{proof}
The inequality \(\cosh x\leq e^{x^2/2}\) follows by integrating
\(\tanh x\leq x\) for \(x\geq0\) and using evenness. Thus, at every replacement stage, a centered coordinate \(H\) satisfies
\[
 \E e^{tH}
 \leq
 \exp\!\left(\frac{t^2}{2}
             \sum_a(u_a)_j^2\right)
 \leq e^{t^2}.
\]
The Gaussian summands satisfy the first bound with equality in their individual factors. Chernoff's inequality gives
\(\Prob(\abs H>x)\leq2e^{-x^2/4}\). Integrating this tail bounds every fixed absolute moment. Adding a shift of magnitude at most one preserves these bounds. The bounds also hold after omitting any one summand.

For the replacement of the \(a\)-th sign, let \(x\) be the sum of all other terms, including the deterministic shift. This base point is independent of the increment. Taylor's formula for \(t\mapsto A(x+tu_a)\), through order three, has remainder bounded by
\[
 C_A\norm{u_a}_2^4\abs t^4
 \left(1+\norm{x}_2^M+\norm{u_a}_2^M\abs t^M\right)
\]
for some nonnegative integer \(M\) depending only on \(A\). The sign and normal increments have matching moments of orders zero through three. The corresponding Taylor terms therefore cancel in expectation. The moment bounds just proved, together with the fixed Gaussian increment moments, bound the expected remainder by
\(C_A\norm{u_a}_2^4\).

Summing the replacement errors proves the first inequality in
\eqref{trees:eq:empirical-replacement}. The coefficient assumptions give
\[
 \max_a\norm{u_a}_2^2\leq\frac{4s}{n},
 \qquad
 \sum_a\norm{u_a}_2^2\leq2s,
 \qquad
 \sum_a\norm{u_a}_2^4\leq\frac{8s^2}{n}.
\]
This proves the second inequality.
\end{proof}

\begin{prop}[One polynomial feedback layer]
\label{trees:prop:feedback-layer}
Let \(n\geq3\) be odd, let \(m=(n-1)/2\), and let
\(\epsilon_1,\ldots,\epsilon_m\) be independent fair signs. Index vectors by \(\Z_n\), and define
\[
 b_0=1,\qquad b_j=b_{-j}=\epsilon_j
 \quad(1\leq j\leq m),\qquad
 F_{kl}=\exp\left(-\frac{2\pi\mathrm{i}kl}{n}\right).
\]
Let \(\1\) be the all-ones vector and \(\ezero\) the coordinate vector at zero. Put
\[
 h=Fb-\1,\qquad Z=\frac h{\sqrt n},\qquad
 R_0=\frac1nF\diag(b-\ezero)F^*,
\]
where \(F^*\) is the conjugate transpose of \(F\). Let \(P\) be a real polynomial of degree \(d\geq0\), applied coordinatewise, and define
\[
 \mu_P=\E[NP(N)],\qquad
 W=R_0P(Z),\qquad
 \widetilde W=W-\sqrt n\,\mu_P\ezero,
\]
where \(N\) is a standard real normal variable. There is a constant \(C_P\), depending only on \(P\), such that
\begin{equation}
 \sup_{\substack{n\geq3\ {\rm odd}\\ k\in\Z_n}}
 \E\abs{\widetilde W_k}^2\leq C_P.
 \label{trees:eq:feedback-second-moment}
\end{equation}
For every integer \(p\geq1\),
\begin{equation}
 \norm{\widetilde W_k}_{L^{2p}}
 \leq(2p-1)^{(d+1)/2}\sqrt{C_P},
 \label{trees:eq:feedback-high-moment}
\end{equation}
uniformly in \(n,k\). Consequently,
\begin{equation}
 \E\max_{k\in\Z_n}\abs{\widetilde W_k}^2
 \leq C_P'\bigl(\log(2n)\bigr)^{d+1},
 \label{trees:eq:feedback-maximum}
\end{equation}
and every fixed absolute moment of \(\widetilde W_k\) is uniformly bounded. If \(\E[NP(N)]=0\), as for every even polynomial \(P\), the macroscopic zero-mode term vanishes. These conclusions concern a single polynomial layer; they assert neither universality for a full adaptive iteration nor control of an adaptive optimizer.
\end{prop}

\begin{proof}
We estimate nonzero coordinates by a four-coordinate replacement. At zero, we retain the deterministic Fourier shift and bound the bias and variance of the empirical polynomial average. Sign hypercontractivity then supplies all higher moments.

\noindent\textbf{Step 1. Fourier coordinates have uniform moments.}
Pairing the terms at \(a\) and \(-a\) gives the real expansion
\begin{equation}
 Z_k=\sum_{a=1}^m
 \frac{2\cos(2\pi ka/n)}{\sqrt n}\epsilon_a.
 \label{trees:eq:linear-fourier-field}
\end{equation}
For every nonzero \(j\in\Z_n\), the finite geometric-series identity gives
\[
 1+2\sum_{a=1}^m\cos(2\pi ja/n)=0.
\]
Applying this identity after the cosine product formula yields
\[
 \E Z_k^2=
 \begin{cases}
 2-2/n,&k=0,\\
 1-2/n,&k\neq0,
 \end{cases}
 \qquad
 \E[Z_kZ_l]=-\frac2n
 \quad(k,l\neq0,\ l\notin\{k,-k\}).
\]
The nonzero variance formula uses that \(2k\neq0\) when \(n\) is odd. No frequency is required to be a unit modulo \(n\).

The sign generating function satisfies
\[
 \E e^{tZ_k}
 \leq
 \exp\!\left(
 \frac{t^2}{2}
 \sum_{a=1}^m\frac{4\cos^2(2\pi ka/n)}n
 \right)
 \leq e^{t^2}.
\]
The tail integration in
Lemma~\ref{trees:lem:polynomial-replacement} therefore bounds every fixed absolute moment of \(Z_k\), uniformly in \(n,k\).

The Fourier formula for \(R_0\) gives
\((R_0)_{kl}=h_{k-l}/n\). Hence
\begin{equation}
 W_k=\frac1{\sqrt n}\sum_{l\in\Z_n}Z_{k-l}P(Z_l).
 \label{trees:eq:feedback-coordinate}
\end{equation}

\medskip
\noindent\textbf{Step 2. Nonzero feedback coordinates have bounded second moments.}
Fix \(k\neq0\). Expanding the square in
\eqref{trees:eq:feedback-coordinate} gives
\begin{equation}
 \E W_k^2
 =
 \frac1n\sum_{l,r\in\Z_n}
 \E\!\left[Z_{k-l}P(Z_l)Z_{k-r}P(Z_r)\right].
 \label{trees:eq:feedback-square}
\end{equation}
Call \((l,r)\) exceptional if one of
\(k-l,l,k-r,r\) is zero or two of these frequencies agree up to sign. There are \(O(n)\) exceptional pairs. A zero condition fixes \(l\) or \(r\). An equality between frequencies from different pairs determines one of \(l,r\) from the other. Within the first pair, equality gives \(2l=k\), with a unique solution because \(n\) is odd; conjugacy would give \(k=0\). The second pair has the same two alternatives with \(r\) in place of \(l\). H\"older's inequality and the uniform coordinate moments bound every exceptional summand by a constant depending only on \(P\).

For a nonexceptional pair, collect the four coordinates into a vector. Its coefficient vectors \(v_a\in\R^4\) satisfy
\[
 \norm{v_a}_2\leq\frac4{\sqrt n},
 \qquad
 \sum_a\norm{v_a}_2^2\leq8.
\]
Each coordinate also satisfies the assumptions of
Lemma~\ref{trees:lem:polynomial-replacement}. Thus replacing the signs by independent standard normals changes the expectation of any fixed polynomial \(A\) of these four coordinates by \(O_A(n^{-1})\).

The resulting centered Gaussian vector has covariance
\[
 \Sigma=I_4-\frac2n\1\1^{\mathsf T}.
\]
For \(n\geq9\), its square root is
\[
 \Sigma^{1/2}
 =
 I_4+
 \frac{\sqrt{1-8/n}-1}{4}\1\1^{\mathsf T}.
\]
Its distance from \(I_4\) in operator norm is \(O(n^{-1})\).
If \(N_4\) has four independent standard normal coordinates, the polynomial mean-value formula and Gaussian moments therefore give
\[
 \abs{\E A(\Sigma^{1/2}N_4)-\E A(N_4)}
 \leq\frac{C_A}{n}.
\]
Take
\[
 A(x_1,x_2,x_3,x_4)=x_1P(x_2)x_3P(x_4).
\]
Its expectation under \(N_4\) is zero, because the first and third coordinates are centered and independent of the other factors. Each nonexceptional summand in
\eqref{trees:eq:feedback-square} is therefore \(O_P(n^{-1})\). It follows that
\[
 \E W_k^2
 \leq
 \frac1n\left(O_P(n)+O_P(n^2)n^{-1}\right)
 \leq C_P.
\]
The finitely many smaller odd values of \(n\) are absorbed into \(C_P\); their sign sample spaces are finite.

\medskip
\noindent\textbf{Step 3. The shifted empirical polynomial average has squared error of order \(n^{-1}\).}
Evenness gives \(Z_{-l}=Z_l\), so
\begin{equation}
 W_0=\sqrt n\,\frac1n\sum_l Z_lP(Z_l).
 \label{trees:eq:feedback-zero-mode}
\end{equation}
Let
\[
 Q(z)=zP(z),\qquad
 \mu_n=n^{-1/2},\qquad
 Y_l=\frac{(Fb)_l}{\sqrt n}=Z_l+\mu_n,
 \qquad
 m_Q=\E Q(N).
\]
Write
\[
 a_{la}=\frac{2\cos(2\pi la/n)}{\sqrt n},
 \qquad
 Y_l=\mu_n+\sum_{a=1}^m a_{la}\epsilon_a.
\]
For \(l\neq0\),
\[
 \sum_a a_{la}^2=v_n:=1-\frac2n.
\]
If \(l,r\neq0\) and \(r\notin\{l,-l\}\), the product-to-sum identity gives
\begin{align*}
 \sum_a a_{la}a_{ra}
 &=
 \frac2n\sum_{a=1}^m
 \left[
 \cos\left(\frac{2\pi(l-r)a}{n}\right)
 +
 \cos\left(\frac{2\pi(l+r)a}{n}\right)
 \right]\\
 &=-\frac2n=:c_n.
\end{align*}
Both half sums are \(-1/2\), since \(l-r\) and \(l+r\) are nonzero modulo \(n\). This calculation applies to composite moduli as well.

For \(s\in\{1,2\}\), the coefficient vectors \(u_a\) of \(s\) selected modes satisfy
\[
 \max_a\norm{u_a}_2^2\leq\frac{4s}{n},
 \qquad
 \sum_a\norm{u_a}_2^2\leq2s.
\]
The common deterministic shift is \(\mu_n\leq1\).
Lemma~\ref{trees:lem:polynomial-replacement} consequently applies to every polynomial of one or two selected coordinates, with matching Gaussian mean and covariance.

For one nonzero mode, the matching Gaussian variable is
\[
 G_l=\mu_n+\sqrt{v_n}\,N.
\]
The polynomial mean-value formula gives
\[
 \abs{\E Q(G_l)-m_Q}
 \leq
 C_Q\left(\mu_n+\abs{\sqrt{v_n}-1}\right)
 \leq\frac{C_Q}{\sqrt n}.
\]
Together with \eqref{trees:eq:empirical-replacement}, this proves
\begin{equation}
 \abs{\E Q(Y_l)-m_Q}
 \leq\frac{C_Q}{\sqrt n}
 \qquad(l\neq0).
 \label{trees:eq:empirical-bias-mode}
\end{equation}
The zero mode has uniformly bounded polynomial moments by Step 1 and the bound on \(\mu_n\). Thus, for
\[
 A_n=\frac1n\sum_lQ(Y_l),
\]
we obtain
\begin{equation}
 \abs{\E A_n-m_Q}\leq\frac{C_Q}{\sqrt n}.
 \label{trees:eq:empirical-bias}
\end{equation}

We next estimate covariances. Assume that \(n\geq5\), \(l,r\neq0\), and \(r\notin\{l,-l\}\). A matching Gaussian pair has the representation
\[
 G_1=\mu_n+\sqrt{v_n}\,N_1,
 \qquad
 G_2=\mu_n+\frac{c_n}{\sqrt{v_n}}N_1
       +\sqrt{v_n-\frac{c_n^2}{v_n}}\,N_2,
\]
where \(N_1,N_2\) are independent standard normal variables. Set
\[
 G_2'=\mu_n+\sqrt{v_n}\,N_2.
\]
Then \(G_1,G_2'\) are independent, and \(G_2,G_2'\) have the same marginal distribution. Since \(v_n\geq3/5\),
\[
 \abs{\frac{c_n}{\sqrt{v_n}}}\leq\frac Cn,
 \qquad
 \abs{
 \sqrt{v_n-\frac{c_n^2}{v_n}}-\sqrt{v_n}}
 \leq\frac C{n^2}.
\]
The second estimate follows by rationalizing the difference of square roots. Every fixed \(L^p\)-norm of \(G_2-G_2'\) is therefore \(O(n^{-1})\). Polynomial derivative bounds and H\"older's inequality imply
\[
 \abs{\E[Q(G_1)(Q(G_2)-Q(G_2'))]}
 \leq\frac{C_Q}{n}.
\]
Independence and equality of the indicated marginals identify the expression inside the absolute value as
\(\Cov(Q(G_1),Q(G_2))\). Hence
\[
 \abs{\Cov(Q(G_1),Q(G_2))}\leq\frac{C_Q}{n}.
\]

Apply \eqref{trees:eq:empirical-replacement} to
\(A(x_1,x_2)=Q(x_1)Q(x_2)\) and to the two one-coordinate means. The means are uniformly bounded, so replacing their product also costs \(O_Q(n^{-1})\). We conclude that
\begin{equation}
 \abs{\Cov(Q(Y_l),Q(Y_r))}
 \leq\frac{C_Q}{n}
 \quad(l,r\neq0,\ r\notin\{l,-l\}).
 \label{trees:eq:empirical-generic-covariance}
\end{equation}

At most \(4n\) ordered pairs \((l,r)\) have a zero coordinate or satisfy \(r=l\) or \(r=-l\). Cauchy--Schwarz and the uniform polynomial second moments bound their covariances by \(C_Q\). For all remaining pairs, use
\eqref{trees:eq:empirical-generic-covariance}. It follows that
\[
 \Var(A_n)
 =
 \frac1{n^2}\sum_{l,r}\Cov(Q(Y_l),Q(Y_r))
 \leq\frac{C_Q}{n}.
\]
Combining this with the squared bias in
\eqref{trees:eq:empirical-bias} proves
\begin{equation}
 \E\abs{\frac1n\sum_lQ(Y_l)-\E Q(N)}^2
 \leq\frac{C_Q}{n}.
 \label{trees:eq:empirical-polynomial-fourier}
\end{equation}
Increasing \(C_Q\) includes \(n=3\).

\medskip
\noindent\textbf{Step 4. Removing the deterministic shift gives the zero-mode second moment.}
The polynomial mean-value formula and uniform coordinate moments yield
\[
 \E\abs{Q(Y_l-n^{-1/2})-Q(Y_l)}^2
 \leq\frac{C_P}{n}
\]
uniformly in \(l\). Jensen's inequality gives
\[
 \E\abs{
 \frac1n\sum_l\bigl(Q(Z_l)-Q(Y_l)\bigr)}^2
 \leq\frac{C_P}{n}.
\]
Since \(m_Q=\mu_P\), the squared triangle inequality and
\eqref{trees:eq:empirical-polynomial-fourier} imply
\[
 \E\abs{\frac1n\sum_lQ(Z_l)-\mu_P}^2
 \leq\frac{C_P}{n}.
\]
Multiplying by \(n\) in
\eqref{trees:eq:feedback-zero-mode} proves
\[
 \E\abs{W_0-\sqrt n\,\mu_P}^2\leq C_P.
\]
Together with Step 2, this establishes
\eqref{trees:eq:feedback-second-moment}.

\medskip
\noindent\textbf{Step 5. Hypercontractivity gives the higher moments and the maximum.}
The polynomial \(\widetilde W_k\) has Walsh degree at most \(d+1\).
In \eqref{trees:eq:feedback-coordinate}, the first factor has degree one and \(P(Z_l)\) has ordinary degree at most \(d\). Reducing powers by \(\epsilon_a^2=1\) can only decrease the degree. Subtracting the deterministic zero-mode term does not increase it. Applying
\eqref{trees:eq:sign-hypercontractivity} to
\eqref{trees:eq:feedback-second-moment} proves
\eqref{trees:eq:feedback-high-moment}. Every fixed absolute moment is then bounded by an even moment of at least that order.

Take \(p=\lceil\log(2n)\rceil\). Pointwise,
\[
 \max_k\abs{\widetilde W_k}^2
 \leq
 \left(\sum_k\abs{\widetilde W_k}^{2p}\right)^{1/p}.
\]
Jensen's inequality for the concave function \(x^{1/p}\), followed by
\eqref{trees:eq:feedback-high-moment}, gives
\[
 \E\max_k\abs{\widetilde W_k}^2
 \leq
 \left(\sum_k\E\abs{\widetilde W_k}^{2p}\right)^{1/p}
 \leq n^{1/p}(2p-1)^{d+1}C_P.
\]
Here \(n^{1/p}\leq e\) and \(2p-1\leq C\log(2n)\). This proves
\eqref{trees:eq:feedback-maximum}.

If \(P\) is even, then \(zP(z)\) is odd, so its expectation under the centered normal law is zero. More generally, whenever \(\E[NP(N)]=0\), the subtracted zero-mode term vanishes by its definition.
\end{proof}

\section{Canonical tree flows and Wick grafting}
\label{sec:tech-trees}\label{trees:sec:canonical-fourier-trees}

Let \(T\) be a finite rooted tree in which each internal node has at least two children; the one-node leaf tree is allowed.  Write \(L(T)\) for its number of leaves and \(d(T)\) for its number of nodes.  If \(x=(x_\ell)_{\ell\ {\rm leaf}}\in\Z_n^{L(T)}\), define the oriented flow at a node \(v\) by
\begin{equation}
 r_v(x)=\sum_{\ell\ {\rm descendant\ of}\ v}x_\ell.
 \label{trees:eq:tree-flow}
\end{equation}
The assignment is admissible when every \(r_v(x)\) is nonzero and the unsigned labels \([r_v(x)]\) are pairwise distinct.

\begin{thm}[Tree-flow Gaussian comparison and canonical covariance]
\label{trees:thm:tree-flow-comparison}
Fix finitely many rooted trees \(T_i\) of the preceding kind.  Put
\(L_i=L(T_i)\) and \(d_i=d(T_i)\).  Let \(\mathcal H_{n,T_i}\) consist of the \(d_i\)-element unsigned label sets obtained from admissible flows on \(T_i\).  For deterministic weights satisfying
\[
 \abs{w_{i,e}}\leq Mn^{-L_i/2},
\]
define
\[
 Y_i=\sum_{e\in\mathcal H_{n,T_i}}w_{i,e}\chi_e,
 \qquad
 G_{ij}=\frac1{d_i}\sum_v(D_vY_i)(D_vY_j).
\]
Then
\begin{equation}
 \E G_{ij}=\E[Y_iY_j],
 \qquad
 \Var(G_{ij})\leq\frac{C}{n},
 \label{trees:eq:tree-Gij-bound}
\end{equation}
where \(C\) depends only on the fixed trees and \(M\).  For fixed real coefficients \(\alpha_i\), if
\[
 Y_\alpha=\sum_i\alpha_iY_i,\qquad
 \sigma_\alpha^2=\E Y_\alpha^2,
\]
then
\begin{align}
 \abs{\E e^{\mathrm{i}tY_\alpha}
       -e^{-\sigma_\alpha^2t^2/2}}
 &\leq C_\alpha\frac{t^2+\abs t^3}{\sqrt n},
 \label{trees:eq:tree-cf}\\
 \abs{\E\abs{Y_\alpha}
       -\sqrt{\frac2\pi}\sigma_\alpha}
 &\leq C_\alpha n^{-1/6}.
 \label{trees:eq:tree-absolute}
\end{align}

For a single tree \(T\), define its canonical field by
\begin{equation}
 Z_{T,k}
 =
 n^{-L(T)/2}
 \sum_{\substack{x\in\Z_n^{L(T)}\\ x\ {\rm admissible}}}
 \cos\!\left(\frac{2\pi k r_{\rm root}(x)}{n}\right)
 \prod_{v\in V(T)}\epsilon_{[r_v(x)]}.
 \label{trees:eq:canonical-tree-field}
\end{equation}
Let \(A_T\) be the integer matrix whose row at \(v\) is the indicator of the descendant leaves of \(v\).  Let \(\Gamma_T\) be the finite group of signed permutation matrices \(P\) on the node coordinates satisfying
\[
 P(\im A_T)=\im A_T.
\]
Write \(g_T=\abs{\Gamma_T}\), and let \(h_T\) be the number of elements whose underlying permutation fixes the root coordinate.  Uniformly over odd \(n\) and \(k,l\in\Z_n\),
\begin{equation}
 \E[Z_{T,k}Z_{T,l}]
 =
 V_T(k,l)+O_T(n^{-1}),
 \label{trees:eq:canonical-tree-covariance}
\end{equation}
where
\[
 V_T(0,0)=g_T,
\]
and, for nonzero \(k,l\),
\[
 V_T(k,l)=
 \begin{cases}
 h_T/2,&l=k\ \text{or}\ l=-k,\\
 0,&l\neq\pm k,
 \end{cases}
\]
while \(V_T(k,l)=0\) if exactly one frequency is zero.  Moreover \(h_T\geq2\).  Every fixed collection of distinct nonzero nonconjugate modes converges jointly, uniformly over the permitted mode choices, to independent centered Gaussian coordinates of variance \(h_T/2\).  These conclusions apply to the homogeneous distinct-label fields only.
\end{thm}

\begin{proof}
For a node \(v\), let \(a_v\) denote its row in \(A_T\).  These rows form a laminar family: two descendant-leaf sets are nested or disjoint.  The leaf rows are the standard basis, so \(A_T\) has column rank \(L=L(T)\).

We first establish two rank properties.  Given a set \(S\) of selected node rows, subtract from each selected row the rows of its immediate selected descendants.  With a tree-compatible ordering, this is a triangular integer row operation with diagonal entries one.  The resulting nonempty residual leaf sets are pairwise disjoint; the remaining residual rows are zero.  Hence
\begin{equation}
 \rho_T(S):=\rank\{a_v:v\in S\}
 =
 \#\{\text{nonempty residual sets}\}.
 \label{trees:eq:laminar-rank}
\end{equation}
Choosing one leaf from each nonempty residual set gives an identity minor.  Therefore prescribed modular values of the forms in \(S\) admit at most
\begin{equation}
 n^{L-\rho_T(S)}
 \label{trees:eq:tree-completion-count}
\end{equation}
leaf assignments for every modulus \(n\).

For every nonempty proper subset \(S\) of the node rows,
\begin{equation}
 \rho_T(S)+\rho_T(S^c)\geq L+1.
 \label{trees:eq:laminar-rank-inequality}
\end{equation}
The sum is at least \(L\), because all rows together span \(\R^L\).
If equality held, the two row spans would have zero intersection.  Each internal node and its children satisfy the minimally supported relation
\[
 a_{\rm parent}=\sum_{\rm child}a_{\rm child}.
\]
The child rows have disjoint nonempty supports and are independent, and every displayed coefficient is nonzero.  If this circuit met both \(S\) and \(S^c\), it would give a nonzero vector in the intersection of the two spans.  Thus every parent-child circuit would have to lie entirely on one side.  Connectivity of the tree would put every node on the same side, contradicting that \(S\) is nonempty and proper.  This proves \eqref{trees:eq:laminar-rank-inequality}.  For the one-node tree there is no such partition.

\medskip
\noindent\textbf{Composite-modulus congruence lemma.}
We will repeatedly use the following elementary fact, valid for every positive
integer modulus.  If $a\in\mathbb Z\setminus\{0\}$, then for fixed $b\in\mathbb Z$
the congruence
\[
 ax=b\pmod n
\]
has either no solutions or exactly $\gcd(a,n)$ solutions modulo $n$.  In
particular,
\begin{equation}
 \#\{x\in\mathbb Z_n:ax=b\}
 \leq \gcd(a,n)\leq |a|.
 \label{trees:eq:composite-congruence}
\end{equation}
After fixing all but one coordinate in a system containing such a nonzero
coefficient, the number of completions is therefore at most $|a|n^{L-1}$.
Thus all subsequent modular counting bounds are uniform over odd composite
$n$ as well as prime $n$.

A fixed admissible label set has at most \(d!2^d\) realizations by node positions and orientations.  Allowing all these realizations, the total number of tree hyperedges is \(O_T(n^L)\).  If one unsigned label is prescribed, choose its node role and orientation and use
\eqref{trees:eq:tree-completion-count} with a rank-one row.  There are
\(O_T(n^{L-1})\) completions.  More generally, prescribed values on a node-role set \(S\) leave
\(O_T(n^{L-\rho_T(S)})\) completions.

These counts imply
\begin{equation}
 \E Y_i^2\leq C,\qquad
 I_{i,v}:=\E(D_vY_i)^2,\qquad
 \max_vI_{i,v}\leq\frac{C}{n},\qquad
 \sum_vI_{i,v}=d_i\E Y_i^2\leq C.
 \label{trees:eq:tree-influences}
\end{equation}
Walsh orthogonality also gives the expectation identity in
\eqref{trees:eq:tree-Gij-bound}.

Before division by \(d_i\), the covariance polynomial is
\[
 \sum_{e,f}\abs{e\cap f}\,w_{i,e}w_{j,f}\chi_{e\symdiff f}.
\]
After its constant terms are removed, every contributing pair is distinct and has nonempty intersection.  Fix selected realizations of the two edges and an overlap pattern, including the node-role subsets, their bijection, orientations, and the placement of outside labels.  Only a tree-dependent number of patterns is possible.

For each nonconstant pair, one of the two trees, say \(A\), has a nonempty proper overlap-role set \(S\).  Let the other tree be \(B\), with leaf counts \(L_A,L_B\).  Choose a full flow on \(B\), which fixes the shared values in \(A\), and then use
\eqref{trees:eq:tree-completion-count}.  The number of edge pairs in this pattern is at most
\begin{equation}
 Cn^{L_A+L_B-\rho_A(S)}.
 \label{trees:eq:tree-pair-count}
\end{equation}
For a fixed resulting Walsh support, its labels and orientations have only finitely many placements in the outside node positions.  Those positions in \(A\) are \(S^c\).  Once they are fixed, the shared labels have at most
\begin{equation}
 Cn^{L_A-\rho_A(S^c)}
 \label{trees:eq:tree-fixed-support-count}
\end{equation}
possible values.

Each pair weight is bounded by
\(Cn^{-(L_A+L_B)/2}\).  For the polynomial belonging to this pattern, the maximum absolute Walsh coefficient is therefore at most
\[
 Cn^{L_A-\rho_A(S^c)-(L_A+L_B)/2},
\]
while the sum of the absolute coefficients is at most
\[
 Cn^{(L_A+L_B)/2-\rho_A(S)}.
\]
The product bounds the sum of coefficient squares.  By
\eqref{trees:eq:laminar-rank-inequality}, it is at most \(C/n\).  Orthogonality converts this into the squared \(L^2\)-norm of the pattern polynomial.  Adding the finitely many patterns proves
\[
 \Var\!\left(\sum_vD_vY_iD_vY_j\right)\leq\frac{C}{n}.
\]
This also covers containment of one edge in another, because the larger tree supplies the nonempty proper overlap set.  If both trees have one node, the covariance polynomial is constant.  This proves
\eqref{trees:eq:tree-Gij-bound}.

The Gaussian comparison now follows from the same finite exchangeable-pair calculation as in Theorem~\ref{trees:thm:additive-gaussian-comparison}.  For completeness, \eqref{trees:eq:weighted-fourth-moment} and
\eqref{trees:eq:tree-influences} give
\[
 \E(D_vY_i)^4\leq3^{2d_i-2}I_{i,v}^2,
 \qquad
 \E\abs{D_vY_i}^3\leq3^{d_i-1}I_{i,v}^{3/2}.
\]
After a uniformly selected sign is flipped, put
\[
 \Delta_i=-2\epsilon_JD_JY_i,\qquad
 \lambda_i=\frac{2d_i}{m}.
\]
Then
\[
 \E[\Delta_i\mid\epsilon]=-\lambda_iY_i
\]
and
\[
 \E[Y_if(Y_\alpha)]
 =
 \frac1{2\lambda_i}
 \E[\Delta_i\{f(Y_\alpha+\Delta_\alpha)-f(Y_\alpha)\}].
\]
Taylor expansion of the exponential yields
\[
 \E[Y_\alpha e^{\mathrm{i}tY_\alpha}]
 =
 \mathrm{i}t\,\E[V_\alpha e^{\mathrm{i}tY_\alpha}]
 +R_\alpha(t),
 \qquad
 V_\alpha=\sum_{i,j}\alpha_i\alpha_jG_{ij}.
\]
The variance estimate gives
\[
 \E V_\alpha=\sigma_\alpha^2,\qquad
 \E\abs{V_\alpha-\sigma_\alpha^2}\leq C_\alpha n^{-1/2}.
\]
Furthermore,
\begin{align*}
 \lambda_i^{-1}\E[\abs{\Delta_i}\Delta_j^2]
 &=
 \frac4{d_i}\sum_v
 \E[\abs{D_vY_i}(D_vY_j)^2]\\
 &\leq
 C\sum_v\sqrt{I_{i,v}}\,I_{j,v}
 \leq
 C\sqrt{\max_vI_{i,v}}\sum_vI_{j,v}
 \leq Cn^{-1/2}.
\end{align*}
Thus \(\abs{R_\alpha(t)}\leq C_\alpha t^2n^{-1/2}\).
The integrating-factor argument proves \eqref{trees:eq:tree-cf}, and
\eqref{trees:eq:absolute-moment-fourier} with the cutoff \(T=n^{1/6}\)
proves \eqref{trees:eq:tree-absolute}.

It remains to calculate the covariance of the canonical field.  Its Walsh coefficient at any label set has magnitude at most
\[
 d!2^d n^{-L/2}.
\]
Consider admissible leaf assignments \(x,y\) whose node flows have the same unsigned label set.  There is a unique signed node permutation \(P\) such that
\begin{equation}
 A_Ty=PA_Tx.
 \label{trees:eq:signed-flow-matching}
\end{equation}
Uniqueness uses distinct unsigned labels and the fact that a nonzero residue cannot equal its own negative in odd order.

Let \(B_T\) have one row
\(r_{\rm parent}-\sum r_{\rm child}\) for each internal node.  Its real nullspace is \(\im A_T\), and its modular nullspace is exactly the set of flows \(A_Tx\), since the leaf values determine all internal values.  If \(P\) does not preserve \(\im A_T\), then
\(B_TPA_T\) is a nonzero integer matrix with tree-dependent bounded entries.  Choose a nonzero coefficient \(c\).  After all other leaf coordinates are fixed, the remaining congruence has at most
By the composite-modulus congruence bound \eqref{trees:eq:composite-congruence}, after all other leaf coordinates are fixed, the remaining congruence has at most
\(\gcd(c,n)\leq|c|\) solutions.  Since the entries of \(B_TPA_T\) are bounded by a constant depending only on \(T\), this gives a uniform \(O_T(n^{L-1})\) bound, with no primality assumption on \(n\).
Hence nonpreserving \(P\) account for
\(O_T(n^{L-1})\) pairs and contribute \(O_T(n^{-1})\) to the normalized covariance.

If \(P\in\Gamma_T\), then \(B_TPA_T=0\) as an integer identity.
Every \(x\) therefore determines a flow \(y\), and the inverse signed permutation gives a bijection.  Removing the nonzero and distinct-label restrictions costs \(O_T(n^{L-1})\) assignments.  A node value equal to zero imposes a primitive indicator-row equation.  Equality or opposition of two distinct node values imposes a nonzero equation having a coefficient \(1\) or \(-1\): for two laminar sets, their symmetric difference or their proper nested difference is nonempty.  Each equation has at most \(n^{L-1}\) solutions, and only finitely many equations occur.

We may therefore evaluate the unrestricted cosine sum for each
\(P\in\Gamma_T\).  The root form is the sum of all leaf coordinates.
If the underlying permutation fixes the root, the two observed forms agree up to sign.  Character orthogonality gives normalized cosine-product average
\[
 1\quad\text{at }(k,l)=(0,0),
 \qquad
 \frac12\quad\text{when }k,l\neq0,\ l=\pm k,
\]
and zero otherwise.  If the root is sent to a proper subtree, choose one leaf inside and one outside that subtree.  The whole-tree sum and subtree sum then have a determinant-one \(2\)-by-\(2\) minor, so they range independently over \(\Z_n^2\).  Their cosine-product average is one only at \((0,0)\), and zero otherwise.  Summation over \(\Gamma_T\) proves
\eqref{trees:eq:canonical-tree-covariance}.  The signed identity and its negative belong to \(\Gamma_T\) and preserve the root, so \(h_T\geq2\).

Finally, apply \eqref{trees:eq:tree-cf} to any fixed collection of canonical fields evaluated at distinct nonzero nonconjugate modes.  Their covariance matrix converges uniformly to \((h_T/2)I\), which proves the asserted independent Gaussian characteristic-function limit.
\end{proof}

The next theorem supplies a moment interface that remains quantitative when the degree parameter is exposed.

\begin{thm}[Polynomial moment comparison]
\label{trees:thm:polynomial-moment-comparison}
Let \(\epsilon_1,\ldots,\epsilon_m\) be independent fair signs, and let
\(Y_i\), \(1\leq i\leq q\), be real homogeneous Walsh polynomials of positive degrees \(d_i\leq D\).  Suppose \(B>0\), \(u\geq0\), and
\(v_{\mathrm{cov}}\geq0\) satisfy
\[
 \E Y_i^2\leq B^2,\qquad
 \max_{i,v}\E(D_vY_i)^2\leq u,
\]
and
\[
 \Var(G_{ij})\leq v_{\mathrm{cov}}^2,
 \qquad
 G_{ij}=\frac1{d_i}\sum_v(D_vY_i)(D_vY_j).
\]
For real \(a_i\), put
\[
 A=\sum_i\abs{a_i},\qquad
 Y=\sum_ia_iY_i,
\]
and let \(G\) be centered Gaussian with variance
\(\sigma^2=\E Y^2\).  For every integer \(r\geq2\),
\begin{equation}
 \abs{\E Y^r-\E G^r}
 \leq
 C_{D,r}A^r
 \left(
 B^{r-2}v_{\mathrm{cov}}
 +B^{r-1}\sqrt u
 \right),
 \qquad
 C_{D,r}=(32r)^{(D+6)r}.
 \label{trees:eq:polynomial-moment-comparison}
\end{equation}
The bound is independent of \(m,q\), and of the Walsh supports.
\end{thm}

\begin{proof}
If \(A=0\), both sides vanish.  Assume \(A>0\).  Every \(Y_i\) is centered.  Since \(D_vY_i\) retains only coefficients of monomials containing \(v\),
\[
 \E(D_vY_i)^2\leq\E Y_i^2\leq B^2.
\]
We may replace \(u\) by \(\min(u,B^2)\).  Put
\(I_{i,v}=\E(D_vY_i)^2\).  Homogeneity gives
\begin{equation}
 \sum_vI_{i,v}=d_i\E Y_i^2\leq DB^2.
 \label{trees:eq:moment-interface-influences}
\end{equation}
Walsh orthogonality gives
\begin{equation}
 \E G_{ij}=\E[Y_iY_j].
 \label{trees:eq:moment-interface-mean}
\end{equation}

Flip a uniformly selected sign \(J\), and set
\[
 \Delta_i=-2\epsilon_JD_JY_i,\qquad
 \Delta=\sum_ia_i\Delta_i,\qquad
 \lambda_i=\frac{2d_i}{m}.
\]
The exact exchangeability identity is
\begin{equation}
 \E[Y_if(Y)]
 =
 \frac1{2\lambda_i}
 \E[\Delta_i\{f(Y+\Delta)-f(Y)\}]
 \label{trees:eq:moment-interface-exchange}
\end{equation}
for every polynomial \(f\).  Apply it with \(f(y)=y^{r-1}\), multiply by \(a_i\), and sum over \(i\).  The linear Taylor term gives
\begin{equation}
 \E Y^r=(r-1)\E[V_aY^{r-2}]+R_r,
 \qquad
 V_a=\sum_{i,j}a_ia_jG_{ij},
 \label{trees:eq:moment-recurrence-start}
\end{equation}
where
\begin{equation}
 \abs{R_r}
 \leq
 \sum_{j=2}^{r-1}\binom{r-1}{j}
 \sum_i\frac{\abs{a_i}}{2\lambda_i}
 \E\!\left[
 \abs{\Delta_i}\abs{\Delta}^j
 \abs Y^{r-1-j}
 \right].
 \label{trees:eq:moment-Taylor-remainder}
\end{equation}
For \(r=2\), this sum is empty.  The covariance hypothesis and the
\(L^2\) triangle inequality give
\begin{equation}
 \norm{V_a-\sigma^2}_2\leq A^2v_{\mathrm{cov}}.
 \label{trees:eq:moment-random-variance}
\end{equation}

By \eqref{trees:eq:sign-hypercontractivity}, every Walsh polynomial \(U\) of degree at most \(D\) satisfies, for \(p\leq2r\),
\begin{equation}
 \norm U_p\leq(2r)^{D/2}\norm U_2.
 \label{trees:eq:moment-interface-hypercontractivity}
\end{equation}
In particular, \(\norm Y_2\leq AB\).  Cauchy--Schwarz and
\eqref{trees:eq:moment-random-variance} yield
\begin{equation}
 (r-1)\abs{\E[(V_a-\sigma^2)Y^{r-2}]}
 \leq
 (r-1)A^rv_{\mathrm{cov}}B^{r-2}
 (2r)^{D(r-2)/2}.
 \label{trees:eq:moment-random-variance-error}
\end{equation}

For the Taylor remainder, put
\[
 s_v=\sum_i\abs{a_i}\sqrt{I_{i,v}}.
\]
Then
\begin{equation}
 \max_vs_v\leq A\sqrt u,
 \qquad
 \sum_vs_v^2
 \leq
 A\sum_i\abs{a_i}\sum_vI_{i,v}
 \leq A^2DB^2.
 \label{trees:eq:sv-bounds}
\end{equation}
Condition on \(J=v\).  The factor \(m/(4d_i)\) from
\(1/(2\lambda_i)\) cancels the probability \(1/m\), while the increments contribute \(2^{j-1}/d_i\leq2^{j-1}\).
Hölder's inequality with total exponent \(r\), followed by
\eqref{trees:eq:moment-interface-hypercontractivity}, gives
\begin{align}
 &\sum_i\abs{a_i}
 \E\!\left[
 \abs{D_vY_i}
 \abs{\textstyle\sum_\ell a_\ell D_vY_\ell}^{\,j}
 \abs Y^{r-1-j}
 \right]\notag\\
 &\hspace{35mm}\leq
 (2r)^{Dr/2}s_v^{j+1}(AB)^{r-1-j}.
 \label{trees:eq:moment-holder}
\end{align}
The last factor is omitted when its exponent is zero.  From
\eqref{trees:eq:sv-bounds},
\[
 \sum_vs_v^{j+1}
 \leq(A\sqrt u)^{j-1}A^2DB^2.
\]
Since \(u\leq B^2\) and \(j\geq2\),
\[
 u^{(j-1)/2}B^{r+1-j}
 \leq\sqrt u\,B^{r-1}.
\]
Each term in \eqref{trees:eq:moment-Taylor-remainder} is therefore at most
\[
 2^{j-1}\binom{r-1}{j}
 (2r)^{Dr/2}A^rD\sqrt u\,B^{r-1}.
\]
The sum of the binomial factors is at most \(3^r\), so
\begin{equation}
 \abs{R_r}
 \leq
 3^r(2r)^{Dr/2}A^rD\sqrt u\,B^{r-1}.
 \label{trees:eq:moment-remainder-final}
\end{equation}

Combining \eqref{trees:eq:moment-recurrence-start},
\eqref{trees:eq:moment-random-variance-error}, and
\eqref{trees:eq:moment-remainder-final} gives
\begin{align}
 \abs{\E Y^r-(r-1)\sigma^2\E Y^{r-2}}
 \leq
 K_{D,r}A^r
 \left(B^{r-2}v_{\mathrm{cov}}+B^{r-1}\sqrt u\right),
 \label{trees:eq:moment-approximate-recurrence}
\end{align}
where \(K_{D,r}=(8r)^{(D+4)r}\) dominates the displayed factors.

Gaussian integration by parts gives
\[
 \E G^r=(r-1)\sigma^2\E G^{r-2}.
\]
The identity also holds when \(\sigma=0\).  Moreover,
\(\sigma^2\leq A^2B^2\), and the moments agree in orders zero, one, and two.  Divide the recurrence for the moment difference by
\((AB)^r\).  If the normalized error is \(e_r\), then
\[
 e_r\leq(r-1)e_{r-2}
 +K_{D,r}
 \left(\frac{v_{\mathrm{cov}}}{B^2}
       +\frac{\sqrt u}{B}\right).
\]
Unrolling requires at most \(r\) terms, and the product of recurrence coefficients is at most \(r^{r/2}\).  Since
\(K_{D,j}\leq K_{D,r}\) for \(2\leq j\leq r\),
\[
 e_r
 \leq
 r^{r/2+1}K_{D,r}
 \left(\frac{v_{\mathrm{cov}}}{B^2}
       +\frac{\sqrt u}{B}\right).
\]
For \(D\geq1\) and \(r\geq2\),
\[
 r^{r/2+1}(8r)^{(D+4)r}
 \leq(32r)^{(D+6)r}.
\]
Multiplying by \((AB)^r\) proves
\eqref{trees:eq:polynomial-moment-comparison}.
\end{proof}

We now identify the canonical Gaussian tree basis and the exact top-degree mechanism behind Wick grafting.

\begin{thm}[Canonical Gaussian tree basis and Wick grafting]
\label{trees:thm:wick-grafting}
For every rooted tree \(T\) considered above,
\begin{equation}
 h_T=2\aut(T),
 \label{trees:eq:root-symmetry-automorphism}
\end{equation}
where \(\aut(T)\) is the number of rooted-tree automorphisms with unordered children.  If \(T_1,\ldots,T_q\) are pairwise nonisomorphic, then their canonical fields at any fixed collection of nonzero nonconjugate Fourier modes converge jointly, uniformly over the mode choices, to independent Gaussian blocks.  The coordinate corresponding to \(T_i\) is a centered Gaussian \(G_i\) satisfying
\begin{equation}
 \E G_i^2=\aut(T_i).
 \label{trees:eq:canonical-Gaussian-variance}
\end{equation}
For every continuous \(H\colon\R^q\to\R\) of polynomial growth,
\begin{equation}
 \frac1n\sum_{k\in\Z_n}
 H(Z_{T_1,k},\ldots,Z_{T_q,k})
 \longrightarrow
 \E H(G_1,\ldots,G_q)
 \quad\text{in }L^2.
 \label{trees:eq:tree-empirical-law}
\end{equation}
Every fixed absolute moment of each root coordinate \(Z_{T,0}\) is uniformly bounded.

Let \(r_1,\ldots,r_q\) be nonnegative integers with
\(r=\sum_ir_i\geq2\).  Form a tree \(T\) by adding a new root whose children comprise \(r_i\) copies of \(T_i\).  Let
\(H_j(x;\sigma^2)\) be defined by
\begin{equation}
 \exp(tx-\sigma^2t^2/2)
 =
 \sum_{j=0}^\infty H_j(x;\sigma^2)\frac{t^j}{j!},
 \label{trees:eq:Hermite-generating-function}
\end{equation}
and set
\[
 W_k=\prod_iH_{r_i}(Z_{T_i,k};\aut(T_i)).
\]
Then
\begin{equation}
 \E\left[\frac1n\norm{K_0W-Z_T}_2^2\right]\longrightarrow0
 \label{trees:eq:Wick-graft-isometry}
\end{equation}
along all odd integers \(n\).

If every \(T_i\) has even numbers of children at all internal nodes and
\(P\) is a fixed real polynomial satisfying
\(P(-x)=P(x)\), then
\(K_0P(Z_{T_1},\ldots,Z_{T_q})\) is approximated in the normalized
\(L^2\) sense of \eqref{trees:eq:Wick-graft-isometry} by a finite linear combination of canonical fields whose trees also have even internal arities.  The constant term is annihilated, and no unary graft is asserted.
\end{thm}

\begin{proof}
Call an antichain of proper descendants a frontier when its descendant-leaf sets partition all leaves.  We first characterize the minimal dependent row sets of \(A_T\) that contain the root.  They are exactly
\[
 \{\mathrm{root}\}\cup\mathcal F,
\]
where \(\mathcal F\) is a frontier, and the unique relation up to scale is
\begin{equation}
 a_{\rm root}=\sum_{v\in\mathcal F}a_v.
 \label{trees:eq:frontier-circuit}
\end{equation}
Indeed, the laminar row operation used in
\eqref{trees:eq:laminar-rank} makes the residual supports disjoint.  A selected row set is dependent precisely when a residual support is empty.  If this occurs at \(v\), its leaf set is covered by its immediate selected descendants, yielding a dependence supported on \(v\) and those descendants.  Minimality forces the original row set to be exactly that support.  If the root belongs to it, then \(v\) is the root and the other nodes form a frontier.  Conversely, frontier rows have disjoint nonempty supports, are independent, and sum to the root row.

Let a signed permutation \(P\) preserve \(\im A_T\) and suppose its underlying permutation fixes the root.  Since \(A_T\) has full column rank, there is an invertible real matrix \(M\) such that
\[
 PA_T=A_TM.
\]
Thus \(P\) preserves the row circuits and their coefficient relations.
Apply it to \eqref{trees:eq:frontier-circuit}.  Uniqueness of the relation shows that every node in the frontier has the same sign as the root.  Every proper node belongs to a frontier obtained by taking that node together with the sibling subtrees encountered along its path to the root.  Hence all signs are equal.

Remove this common sign.  The frontier of all leaves is the unique frontier of largest cardinality \(L\): replacing the children below any internal node by that node lowers the size because every internal node has at least two children.  The unsigned permutation therefore preserves the leaves.  Since leaf rows are the standard basis, \(M\) is a leaf permutation.  It preserves the family of descendant-leaf sets and its containment relation, hence preserves parenthood.  It is a rooted-tree automorphism.  Conversely, every rooted-tree automorphism, multiplied by either global sign, preserves the flow subspace and the root.  For the one-node tree the same statement is immediate.  This proves \eqref{trees:eq:root-symmetry-automorphism}.

The same circuit argument applies to a signed coordinate bijection between the flow spaces of two trees that sends root to root.  After removal of a global sign, it induces a rooted-tree isomorphism.  We use this to calculate mixed covariances.  Trees with different node counts have exactly zero covariance because their Walsh degrees differ.  If two trees have equal node counts but different leaf counts \(L,L'\), a matching pair of distinct-label flows is specified by a signed node bijection and either flow.  There are at most
\(Cn^{\min(L,L')}\) such pairs, so normalization gives covariance
\[
 O\!\left(n^{-\abs{L-L'}/2}\right).
\]

Suppose the node and leaf counts agree.  A signed node bijection that does not map the first real flow space to the second imposes a nonzero bounded-coefficient congruence on the leaf parameters.  Its solutions number
\(O(n^{L-1})\), by the same greatest-common-divisor bound used in the proof of
Theorem~\ref{trees:thm:tree-flow-comparison}.  A preserving bijection supplies all
\(n^L\) flows before the nonzero and distinct-label restrictions, whose removal costs \(O(n^{L-1})\).  At nonzero output modes, a leading term can occur only when root maps to root.  Otherwise the whole-leaf sum and a proper-subtree sum have a determinant-one minor and range independently.  Root-to-root maps exist only for isomorphic trees.  In that case the two canonical fields agree after relabeling leaf assignments.  Thus nonisomorphic trees have asymptotically zero mixed covariance, while
\eqref{trees:eq:canonical-tree-covariance} and
\eqref{trees:eq:root-symmetry-automorphism} give variance
\(\aut(T_i)\) at a nonzero mode.  Distinct nonconjugate modes have zero limiting covariance.

Theorem~\ref{trees:thm:tree-flow-comparison} now gives joint characteristic-function convergence to the Gaussian coordinates in
\eqref{trees:eq:canonical-Gaussian-variance}.  We next transfer this convergence to polynomial-growth tests.  Every coordinate has uniformly bounded second moment, including the root, by the flow count and coefficient bound in the proof of that theorem.  Each coordinate is a fixed-degree sign polynomial, so
\eqref{trees:eq:sign-hypercontractivity} supplies all fixed higher moments uniformly.

For compactly supported continuous \(H\), convolve with a narrow Gaussian.
The convolution approximates \(H\) uniformly and has an integrable Fourier representation.  Joint characteristic-function convergence and dominated convergence transfer its expectations.  For polynomial-growth \(H\), insert a large compact cutoff and bound the discarded expectation by a sufficiently high uniform moment divided by a power of the cutoff radius.  This proves uniform convergence of one-block and two-block expectations.

Expand the square of the difference in \eqref{trees:eq:tree-empirical-law}.
For nonzero nonconjugate ordered mode pairs, the expectation of the product converges uniformly to the square of the Gaussian mean.  Only \(O(n)\) ordered pairs contain the zero mode or matching or conjugate modes.  Their contributions are bounded by the uniform second moment of the test.  Division by \(n^2\) proves the \(L^2\) convergence.  The root moment assertion follows from the same fixed-degree moment bound.

We turn to Wick grafting.  Because simultaneous negation of all leaf residues preserves every sign label and pairs the complex exponential with its conjugate, the cosine in \eqref{trees:eq:canonical-tree-field} can be replaced by
\[
 \exp(-2\pi\mathrm{i}kr_{\rm root}/n)
\]
without changing the field.  Multiply these exponential forms for the
\(r\) distinguished child-tree copies.  Applying \(K_0\), whose kernel is
\[
 (K_0)_{kl}
 =
 \frac1n\sum_{s\in\Z_n}
 a_s\exp\!\left(-\frac{2\pi\mathrm{i}(k-l)s}{n}\right),
\]
and summing over \(l\) imposes
\[
 s=\sum_{\text{child copies}}r_{\rm child}.
\]
The factor \(1/n\) cancels the coordinate sum.  The normalization is
\[
 n^{-\frac12\sum_ir_iL(T_i)},
\]
which is exactly the normalization of the grafted tree.  The extra sign
\(a_s\) supplies the new root.

Let
\[
 D=1+\sum_ir_id(T_i)
\]
be the node count of the grafted tree.  The degree-\(D\) Walsh terms are precisely those for which every sign label among all child nodes and the new root is nonzero and distinct.  Repetition of a sign label lowers the reduced Walsh degree by at least two.  Since each child field already imposes its internal distinct-label restriction, the degree-\(D\) component is exactly \(Z_T\).

Each \(H_{r_i}(x;\aut(T_i))\) has leading term \(x^{r_i}\), while every lower term has exponent smaller by a positive even integer.  Thus all other terms in \(W\) have sign degree strictly below
\(\sum_ir_id(T_i)\).  Applying \(K_0\) raises Walsh degree by at most one.  Therefore the degree-\(D\) Walsh projection of \(K_0W\) is exactly \(Z_T\).  Orthogonality gives, coordinatewise,
\[
 \E\abs{(K_0W)_k-Z_{T,k}}^2
 =
 \E\abs{(K_0W)_k}^2-\E\abs{Z_{T,k}}^2.
\]
After summing over \(k\),
\begin{equation}
 \E\left[\frac1n\norm{K_0W-Z_T}_2^2\right]
 =
 \E\left[\frac1n\norm{K_0W}_2^2\right]
 -
 \E\left[\frac1n\norm{Z_T}_2^2\right].
 \label{trees:eq:orthogonal-graft-difference}
\end{equation}

Since \(K_0^2=I-J\) by \eqref{eq:projection-decomposition},
\begin{equation}
 \frac1n\norm{K_0W}_2^2
 =
 \frac1n\sum_kW_k^2
 -
 \left(\frac1n\sum_kW_k\right)^2.
 \label{trees:eq:K0-energy}
\end{equation}
Apply \eqref{trees:eq:tree-empirical-law} to \(W\) and \(W^2\).  At least one \(r_i\) is positive, so \(\E W(G)=0\).  Hermite orthogonality gives
\begin{equation}
 \E W(G)^2
 =
 \prod_i r_i!\,\aut(T_i)^{r_i}.
 \label{trees:eq:Hermite-product-energy}
\end{equation}
Indeed, multiplying two copies of
\eqref{trees:eq:Hermite-generating-function} and integrating against a centered normal of variance \(\sigma^2\) gives \(e^{\sigma^2ts}\).  Coefficient comparison yields
\[
 \E[H_r(G;\sigma^2)H_s(G;\sigma^2)]
 =
 \mathbf1_{\{r=s\}}r!(\sigma^2)^r.
\]
Independence gives \eqref{trees:eq:Hermite-product-energy}.

On the other hand,
\eqref{trees:eq:canonical-tree-covariance} and
\eqref{trees:eq:root-symmetry-automorphism} show that
\[
 \E\left[\frac1n\norm{Z_T}_2^2\right]\longrightarrow\aut(T).
\]
The root contribution is negligible.  The automorphism count of the grafted tree is
\begin{equation}
 \aut(T)=\prod_ir_i!\,\aut(T_i)^{r_i}.
 \label{trees:eq:grafted-automorphism-count}
\end{equation}
An automorphism permutes the \(r_i\) isomorphic children and acts independently inside them; nonisomorphic child types cannot be interchanged.  The two limiting energies in
\eqref{trees:eq:orthogonal-graft-difference} therefore agree.  Their difference is nonnegative and tends to zero, proving
\eqref{trees:eq:Wick-graft-isometry}.  The assumption \(r\geq2\) ensures that the new root has at least two children.

Finally, suppose all input trees have even internal arities.  Their node counts are odd, since the number of edges is the sum of all internal child counts and is even.  Expand the globally even polynomial \(P\) in the multivariate Hermite basis associated with the independent variances \(\aut(T_i)\).  This expansion is finite because every univariate Hermite polynomial is monic.  The parity identity
\[
 H_r(-x;\sigma^2)=(-1)^rH_r(x;\sigma^2)
\]
shows that only multi-indices of even total degree occur.  The total-degree-zero term is constant and is annihilated by \(K_0\).  Every remaining term has even total degree at least two, and its grafted tree again has even internal arities.  Apply
\eqref{trees:eq:Wick-graft-isometry} to the finitely many terms and use the triangle inequality for the normalized \(L^2\) norm.  This proves the final assertion.
\end{proof}

\section{Uniform growing-degree bounds}

The preceding arguments admit explicit constants large enough to remain valid when the degree cutoff varies with \(n\).

\begin{thm}[Growing-degree covariance and Gaussian comparison]
\label{trees:thm:growing-degree-comparison}
For \(D\geq1\), define
\begin{equation}
 A_D=(4D)^{80D}.
 \label{trees:eq:AD-definition}
\end{equation}
Let \(T_1,\ldots,T_q\) be pairwise nonisomorphic rooted trees with at most \(D\) nodes and at least two children at every internal node.  For real coefficients \(c_i\) satisfying
\[
 V=\sum_i\aut(T_i)c_i^2\leq1,
\]
put
\[
 X_n(k)=\sum_ic_iZ_{T_i,k}.
\]
For every odd \(n\geq3\), uniformly over nonzero \(k\),
\begin{align}
 \abs{\E X_n(k)^2-V}
 &\leq\frac{A_D}{\sqrt n},
 \label{trees:eq:growing-variance}\\
 \abs{\E e^{\mathrm{i}tX_n(k)}-e^{-Vt^2/2}}
 &\leq
 \frac{A_D(t^2+\abs t^3)}{\sqrt n}
 \qquad(t\in\R).
 \label{trees:eq:growing-cf}
\end{align}
Moreover,
\begin{align}
 \abs{\E\frac{\norm{X_n}_2^2}{n}-V}
 &\leq\frac{A_D}{\sqrt n},
 \label{trees:eq:growing-energy}\\
 \E\abs{X_n(0)}^2&\leq A_D.
 \label{trees:eq:growing-root}
\end{align}

If \(k,l\) are nonzero and nonconjugate, and two coefficient lists
\(c_i,e_i\) satisfy
\[
 V_2=\sum_i\aut(T_i)(c_i^2+e_i^2)\leq1,
\]
then
\[
 Y=\sum_i\{c_iZ_{T_i,k}+e_iZ_{T_i,l}\}
\]
satisfies
\begin{equation}
 \abs{\E e^{\mathrm{i}tY}-e^{-V_2t^2/2}}
 \leq
 \frac{A_D(t^2+\abs t^3)}{\sqrt n}.
 \label{trees:eq:growing-two-mode-cf}
\end{equation}
All these bounds remain valid when \(D\) and the deterministic coefficient lists depend on \(n\).  If
\[
 D\log(2D)=o(\log n),
\]
the comparison errors are at most \(n^{-1/2+o(1)}\).  For
\(v_n=X_n/\sqrt n\),
\begin{equation}
 \E\abs{v_n(0)}^2
 \leq\frac{A_D}{n}
 \leq n^{-1+o(1)}.
 \label{trees:eq:growing-scaled-root}
\end{equation}
This is a one-sided upper bound and does not assert a nonzero lower bound.  No growing-degree nonlinear Wick-graft estimate is included.
\end{thm}

\begin{proof}
Put
\[
 P=(4D)^{4D},
 \qquad A_D=P^{20}.
\]
The number \(q\) of rooted-tree types with at most \(D\) nodes is at most \(4^D\).  Indeed, every unordered rooted tree admits a plane representative, and a depth-first parentheses word for a plane tree with \(d\) nodes has length \(2(d-1)\).  Thus there are at most \(4^{d-1}\) representatives of size \(d\).  Since \(\aut(T)\geq1\), Cauchy--Schwarz gives
\begin{equation}
 \sum_i\abs{c_i}\leq2^D
 \label{trees:eq:growing-coefficient-l1}
\end{equation}
when \(V\leq1\).  For the two coefficient lists, their combined absolute sum is at most
\(\sqrt2\,2^D\leq P\).  We use \(P\) as a common upper bound.

Consider a tree with \(d\leq D\) nodes and \(L\) leaves.  It has at most
\(n^L\) leaf assignments.  If an unsigned nonzero label is assigned to a chosen node, its orientation can be selected in two ways and
\eqref{trees:eq:tree-completion-count} leaves at most \(n^{L-1}\) flows.
Thus there are at most \(2Dn^{L-1}\) flows through a fixed label.
A fixed label set has at most
\[
 2^dd!\leq P
\]
oriented node realizations.  It follows that a canonical field, viewed as a Walsh polynomial, has at most \(n^L\) monomials and each coefficient is at most
\(Pn^{-L/2}\).  Consequently,
\begin{equation}
 \E Z_{T,k}^2\leq P^2,
 \qquad
 \max_vI_v\leq\frac{P^3}{n},
 \qquad
 \sum_vI_v\leq P^3,
 \label{trees:eq:growing-influences}
\end{equation}
where \(I_v=\E(D_vZ_{T,k})^2\).  These bounds hold at every output coordinate.

Before tracking the covariance-polynomial variance, we make the dependence on the cutoff $D$ explicit.  The number of choices of two node-role subsets is at most $2^{2D}$; a matching between selected roles contributes at most $D!$; the orientation choices contribute at most $2^{2D}$; and the possible overlap size contributes at most $D+1$.  Hence the total number of overlap patterns is bounded by
\begin{equation}
 N_D\leq 2^{4D}D!(D+1)
 \leq \exp\!\bigl(CD\log(2D)\bigr).
 \label{trees:eq:uniform-pattern-count}
\end{equation}
for an absolute constant $C$.  The same type of bound applies to the placements of outside labels and orientations in a fixed resulting Walsh support.  Likewise, all integer coefficients arising from the flow equations are bounded by $D$, and \eqref{trees:eq:composite-congruence} converts each nonzero equation into an $O(Dn^{L-1})$ counting bound.  Consequently the entire fixed-degree argument remains valid uniformly for $D\to\infty$, with its constants bounded by $\exp(CD\log(2D))$ before the deliberately coarser envelope $P=(4D)^{4D}$ is applied.

We next track the covariance-polynomial variance.  For two canonical fields \(Y_i,Y_j\), with leaf counts \(L_i,L_j\), the contribution of an intersecting edge pair to
\(\sum_v(D_vY_i)(D_vY_j)\) has magnitude at most
\[
 P^3n^{-(L_i+L_j)/2}.
\]
An overlap pattern consists of two subsets of at most \(D\) node positions, their bijection, relative orientations, and the overlap size.  The number of patterns is at most
\[
 2^{2D}D!2^D(D+1)\leq P.
\]
For a fixed resulting Walsh support, assigning its outside labels and orientations to at most \(2D\) positions has at most
\[
 (2D)!2^{2D}\leq P
\]
possibilities.

Choose the tree \(A\) for which the overlap node set \(S\) is nonempty and proper, and call the other tree \(B\).  The rank argument in
\eqref{trees:eq:tree-pair-count} bounds the number of full pairs for one pattern by
\[
 Pn^{L_A+L_B-\rho_A(S)}.
\]
Once the outside support is fixed, the number of shared-label choices is at most
\[
 Pn^{L_A-\rho_A(S^c)}.
\]
The maximum absolute Walsh coefficient of the pattern is therefore at most
\[
 P^4
 n^{L_A-\rho_A(S^c)-(L_A+L_B)/2},
\]
and its total absolute coefficient sum is at most
\[
 P^4
 n^{(L_A+L_B)/2-\rho_A(S)}.
\]
Their product is at most \(P^8/n\) by
\eqref{trees:eq:laminar-rank-inequality}.  Adding at most \(P\) patterns by the squared triangle inequality gives
\begin{equation}
 \Var(G_{ij})\leq\frac{P^{10}}{n},
 \qquad
 G_{ij}=\frac1{d_i}\sum_v(D_vY_i)(D_vY_j).
 \label{trees:eq:growing-Gij}
\end{equation}
If both trees have one node, the variance is zero.  The exact mean identity is
\[
 \E G_{ij}=\E[Y_iY_j].
\]

We also need an explicit covariance comparison.  Different node counts give exact zero covariance.  If the node counts agree but the leaf counts differ, a matching label set determines a signed node bijection, of which there are at most \(2^DD!\leq P\).  There are at most
\(n^{\min(L,L')}\) matching flow pairs for each bijection.  The normalized covariance is at most
\[
 Pn^{-\abs{L-L'}/2}\leq\frac{P}{\sqrt n}.
\]
For equal leaf counts, a nonpreserving signed bijection imposes a nonzero row congruence from \(B_{T'}PA_T\).  Each coefficient has magnitude at most \(D\).  Choosing a nonzero coefficient bounds its solution count by
\(Dn^{L-1}\) for every modulus.  For a preserving bijection, exclusion of zero or repeated unsigned node labels removes at most
\(D^2n^{L-1}\) assignments.  Summing over at most \(P\) bijections gives normalized error at most
\[
 \frac{P(D+D^2)}{n}\leq\frac{P^3}{n}.
\]
The unrestricted cosine averages were calculated in
Theorem~\ref{trees:thm:wick-grafting}: at nonzero modes, only root-to-root mappings survive; their number is \(2\aut(T)\), and nonisomorphic trees have no such mapping.  Hence, for the one-mode or two-mode linear combination in the statement,
\begin{equation}
 \abs{\E Y^2-V_{\rm target}}
 \leq\frac{P^6}{\sqrt n}.
 \label{trees:eq:growing-target-covariance}
\end{equation}

For a single combination \(X_n\), \eqref{trees:eq:growing-influences} and
\eqref{trees:eq:growing-coefficient-l1} give
\[
 \E\abs{X_n(0)}^2\leq P^4.
\]
Averaging \eqref{trees:eq:growing-target-covariance} over the \(n-1\) nonzero modes, and adding the root contribution and the omitted \(V/n\), gives
\[
 \abs{\E\frac{\norm{X_n}_2^2}{n}-V}
 \leq
 \frac{P^6}{\sqrt n}+\frac{P^4+1}{n}
 \leq\frac{P^8}{\sqrt n}.
\]
This proves \eqref{trees:eq:growing-variance},
\eqref{trees:eq:growing-energy}, and \eqref{trees:eq:growing-root} with room to spare.

For the characteristic function, write the one-mode or two-mode variable as
\[
 Y=\sum_ia_iY_i,
 \qquad
 \sum_i\abs{a_i}\leq P.
\]
Flip a uniformly selected sign and use
\[
 \Delta_i=-2\epsilon_JD_JY_i,\qquad
 \lambda_i=\frac{2d_i}{m}.
\]
The exchangeability identity gives
\[
 \E[Ye^{\mathrm{i}tY}]
 =
 \mathrm{i}t\,\E[V_ae^{\mathrm{i}tY}]+R_a(t),
 \qquad
 V_a=\sum_{i,j}a_ia_jG_{ij}.
\]
From \eqref{trees:eq:growing-Gij},
\begin{equation}
 \E\abs{V_a-\E Y^2}
 \leq
 \left(\sum_i\abs{a_i}\right)^2
 \frac{P^5}{\sqrt n}
 \leq\frac{P^7}{\sqrt n}.
 \label{trees:eq:growing-random-variance}
\end{equation}
The fourth-moment inequality
\eqref{trees:eq:weighted-fourth-moment} gives
\[
 \E[\abs{D_vY_i}(D_vY_j)^2]
 \leq
 3^{d_j-1}\sqrt{I_{i,v}}\,I_{j,v}.
\]
Using \(3^{D-1}\leq P\) and \eqref{trees:eq:growing-influences},
\begin{align}
 \lambda_i^{-1}\E[\abs{\Delta_i}\Delta_j^2]
 &=
 \frac4{d_i}\sum_v
 \E[\abs{D_vY_i}(D_vY_j)^2]\notag\\
 &\leq\frac{P^7}{\sqrt n}.
 \label{trees:eq:growing-flip-bound}
\end{align}
The weighted inequality
\[
 \left(\sum_ia_i\Delta_i\right)^2
 \leq
 \left(\sum_i\abs{a_i}\right)
 \sum_j\abs{a_j}\Delta_j^2
\]
and \eqref{trees:eq:growing-flip-bound} give
\[
 \abs{R_a(t)}\leq\frac{P^{10}t^2}{\sqrt n}.
\]
If \(\phi(t)=\E e^{\mathrm{i}tY}\), then
\[
 \abs{\phi'(t)+(\E Y^2)t\phi(t)}
 \leq
 \frac{P^{10}(\abs t+t^2)}{\sqrt n}.
\]
The integrating-factor calculation gives
\[
 \abs{\phi(t)-e^{-(\E Y^2)t^2/2}}
 \leq
 \frac{P^{10}(t^2+\abs t^3)}{\sqrt n}.
\]
Replacing \(\E Y^2\) by \(V_{\rm target}\) costs at most
\[
 \frac{t^2}{2}
 \abs{\E Y^2-V_{\rm target}}
\]
by \eqref{trees:eq:growing-target-covariance}.  Enlarging the coefficient to
\(P^{20}=A_D\) proves \eqref{trees:eq:growing-cf} and
\eqref{trees:eq:growing-two-mode-cf}.

Finally,
\[
 \log A_D=80D\log(4D).
\]
If \(D\log(2D)=o(\log n)\), then
\[
 \frac{A_D}{\sqrt n}=n^{-1/2+o(1)}
 \quad\text{and}\quad
 \frac{A_D}{n}=n^{-1+o(1)}.
\]
Together with \eqref{trees:eq:growing-root}, this proves
\eqref{trees:eq:growing-scaled-root}.  The estimate permits the root moment to be zero, including when every coefficient is zero.
\end{proof}

The following all-coordinate bound is independent of the Gaussian comparison and uses only the exact label-matching count.

\begin{prop}[All-coordinate bound]
\label{trees:prop:all-coordinate-bound}
Let \(D\geq1\), and let \(c_T\) be deterministic real coefficients indexed by pairwise nonisomorphic rooted trees with at most \(D\) nodes and at least two children at every internal node.  Define
\[
 Q_c=\sum_Tc_T^2\aut(T),
 \qquad
 v_k=\frac1{\sqrt n}\sum_Tc_TZ_{T,k}.
\]
For every odd \(n\geq3\),
\begin{equation}
 \max_{k\in\Z_n}\E\abs{v_k}^2
 \leq
 \frac{Q_cD!2^{3D+2}}{n}.
 \label{trees:eq:all-coordinate-bound}
\end{equation}
The result allows \(D\) and the coefficients to depend deterministically on \(n\).  If \(Q_c\) is uniformly bounded and
\[
 D_n\log(D_n+1)=o(\log n),
\]
then
\begin{equation}
 \max_k\E\abs{v_k}^2\leq n^{-1+o(1)}.
 \label{trees:eq:all-coordinate-asymptotic}
\end{equation}
No estimate for the full squared vector norm or for a nonlinear residual is asserted.
\end{prop}

\begin{proof}
Fix a tree \(T\) with \(d\leq D\) nodes and \(L\) leaves.  From
\eqref{trees:eq:canonical-tree-field},
\[
 Z_{T,k}
 =
 n^{-L/2}
 \sum_{x\ {\rm admissible}}
 \cos(2\pi kr_{\rm root}(x)/n)\chi_{E_T(x)},
\]
where \(E_T(x)\) is the set of \(d\) distinct unsigned node labels.
There are at most \(n^L\) assignments \(x\).  In the squared expectation, a pair \(x,y\) contributes only when
\(E_T(x)=E_T(y)\), by Walsh orthogonality.

For a fixed admissible \(x\), there are at most \(d!2^d\) assignments
\(y\) with the same label set.  Assign the \(d\) unsigned labels to the
\(d\) node positions and then choose one of the two residue orientations at each node.  The leaf-node values determine \(y\).  Some choices fail the flow equations or admissibility, which only lowers the count.  Since each cosine product has magnitude at most one,
\begin{equation}
 \E\abs{Z_{T,k}}^2
 \leq
 n^{-L}n^Ld!2^d
 \leq D!2^D.
 \label{trees:eq:single-tree-coordinate}
\end{equation}
This bound uses neither a limiting argument nor a modulus-dependent inverse.

The number of unordered rooted-tree types with \(j\) nodes is at most \(4^j\), by the depth-first traversal encoding used in
Theorem~\ref{trees:thm:growing-degree-comparison}.  Therefore the number with at most \(D\) nodes is less than \(4^{D+1}\).  If
\(M=\sum_T\abs{c_T}\), then \(\aut(T)\geq1\) and Cauchy--Schwarz give
\begin{equation}
 M^2
 \leq
 4^{D+1}\sum_Tc_T^2
 \leq4^{D+1}Q_c.
 \label{trees:eq:coefficient-count-bound}
\end{equation}
The \(L^2\) triangle inequality and
\eqref{trees:eq:single-tree-coordinate} yield
\[
 \bigl(\E\abs{v_k}^2\bigr)^{1/2}
 \leq
 \frac1{\sqrt n}
 \sum_T\abs{c_T}
 \bigl(\E\abs{Z_{T,k}}^2\bigr)^{1/2}
 \leq
 \frac{M\sqrt{D!2^D}}{\sqrt n}.
\]
Squaring and using \eqref{trees:eq:coefficient-count-bound} proves
\eqref{trees:eq:all-coordinate-bound}.

If \(Q_c\) is bounded, then
\[
 \log\!\left(D!2^{3D+2}\right)
 \leq D\log D+(3D+2)\log2
 =o(\log n)
\]
under the stated growth condition.  This proves
\eqref{trees:eq:all-coordinate-asymptotic}, including the case \(Q_c=0\).
\end{proof}

\section{Smoothed Gaussian fixed point and degree tails}
\label{sec:tech-gaussian}
\label{gaussian:sec:smoothed-gaussian-tree}

In this section, we construct the smoothed Gaussian fixed point and estimate its canonical tree-degree tail. The contraction estimate gives existence and uniqueness; a uniform Hermite estimate identifies a subcritical offspring law and controls its total population.

\begin{notation}[Gaussian tree coordinates]
\label{gaussian:setup:gaussian-tree-space}
We use the canonical Gaussian coordinates and Wick grafting from
Theorems~\ref{trees:thm:tree-flow-comparison} and
\ref{trees:thm:wick-grafting}.
Let \(\cT\) denote the finite rooted trees with positive even arity at every internal vertex, including the one-vertex leaf \(L\).
Write \(d(T)\) for the number of vertices and \(\aut(T)\) for the rooted automorphism count.
The coordinates \(G_T\) are independent centered Gaussians with
\(\E G_T^2=\aut(T)\).
Let \(\mathcal H\) be their closed real linear span in \(L^2\), and let
\(\mathcal E_{\mathrm{even}}\) be the closed subspace of mean-zero \(L^2\) functions invariant under simultaneous negation of all coordinates.

Use the Hermite convention
\[
 \exp\left(tx-\frac{\sigma^2t^2}{2}\right)
 =
 \sum_{r=0}^{\infty}H_r(x;\sigma^2)\frac{t^r}{r!},
 \qquad
 \He_r(x)=H_r(x;1).
\]
For a finitely supported family \((r_A)_{A\in\cT}\) of nonnegative integers with positive even sum, let \(T\) be the tree obtained by grafting a root above \(r_A\) copies of each \(A\).
The corresponding Wick monomials form an orthogonal basis of
\(\mathcal E_{\mathrm{even}}\), and
\begin{equation}
 \left\|\prod_A H_{r_A}(G_A;\aut(A))\right\|_2^2
 =
 \prod_A r_A!\,\aut(A)^{r_A}
 =
 \aut(T).
 \label{gaussian:eq:grafting-norm}
\end{equation}
We denote the resulting grafting isometry by
\[
 J_{\tree}\colon
 \mathcal E_{\mathrm{even}}
 \longrightarrow
 \mathcal H\ominus\operatorname{span}\{G_L\},
 \qquad
 \prod_A H_{r_A}(G_A;\aut(A))\longmapsto G_T.
\]
Throughout this section, \(\tau>0\) is fixed,
\(\phi_\tau(x)=\sqrt{x^2+4\tau}\), and \(Z\) denotes a standard normal variable.
All \(L^2\) norms refer to the Gaussian probability space.
\end{notation}

For finite-dimensional Gaussian polynomial density, let \(\gamma\) be the law of \(m\) coordinates and let \(h\in L^2(\gamma)\) be orthogonal to every polynomial. Cauchy--Schwarz and Gaussian exponential moments justify the entire power series of \(M(z)=\int h(x)e^{z\cdot x}\,d\gamma(x)\), uniformly on compact subsets of \(\C^m\). Every coefficient vanishes. Taking \(z=it\) shows that the integrable density \(h\,d\gamma/dx\) has zero Fourier transform. Its convolution with every centered Gaussian density is zero, by the absolutely convergent Gaussian Fourier integral and Fubini's theorem. These convolutions converge to the density in \(L^1\) as their variance decreases to zero: uniform continuity and tail estimates prove this for continuous compactly supported functions, and approximation by such functions proves it for every \(L^1\) function. Hence \(h=0\), proving polynomial density. Tensor Hermite products span all polynomials because the univariate Hermite polynomials are monic. Their orthogonality therefore gives completeness in finite dimension.

For finite-cylinder approximation, enumerate \(G=(G_{T_j})_{j\geq1}\) and let \(\mathcal F_m\) be the sigma field of its first \(m\) coordinates. Events approximable in probability by cylinder events form a sigma field: complements preserve approximation, and countable unions reduce to finite unions by continuity from below. This sigma field contains the generating cylinders, so cylinder indicators have dense linear span in product \(L^2\). The orthogonal projections \(f_m=\E[f\mid\mathcal F_m]\) consequently converge to \(f\) in \(L^2\), since their ranges increase to a dense union. For \(f\in\mathcal E_{\mathrm{even}}\), approximate \(f_m\) by finite-coordinate polynomials and apply \(P\mapsto(P(G)+P(-G))/2-\E P(G)\). Global negation preserves the product law, so symmetrization and centering are commuting orthogonal projections whose composition fixes \(f\) and cannot increase the approximation error. The resulting Hermite expansions contain only positive even total orders, proving completeness of the stated finite Wick monomials. Their child multisets correspond bijectively to the nonleaf even-arity rooted trees. Orthogonality and \eqref{gaussian:eq:grafting-norm} therefore make the prescribed map an isometry on their linear span. It extends continuously to \(\mathcal E_{\mathrm{even}}\), including its nonpolynomial functions. Its range is closed and contains every nonleaf coordinate, so it equals \(\mathcal H\ominus\operatorname{span}\{G_L\}\).

\begin{lem}[Strict contraction under Gaussian smoothing]
\label{gaussian:lem:strict-gaussian-contraction}
Let \(X,Y\) be centered jointly Gaussian real variables, and let \(Q_0>0\).
Assume that \(\Var(X),\Var(Y)\leq Q_0\).
Then
\begin{equation}
 \E\abs{\phi_\tau(X)-\phi_\tau(Y)}^2
 \leq
 \frac{3Q_0}{3Q_0+4\tau}\,\E\abs{X-Y}^2.
 \label{gaussian:eq:strict-gaussian-contraction}
\end{equation}
\end{lem}

\begin{proof}
Set \(D=Y-X\) and \(X_s=(1-s)X+sY\) for \(0\leq s\leq1\).
The \(L^2\) triangle inequality gives
\(\Var(X_s)\leq Q_0\).
The fundamental theorem of calculus and Cauchy--Schwarz yield
\[
 \abs{\phi_\tau(Y)-\phi_\tau(X)}^2
 \leq
 \int_0^1D^2\phi_\tau'(X_s)^2\,ds.
\]

Fix \(s\), and write \(\sigma^2=\Var(X_s)\).
When \(\sigma>0\), Gaussian orthogonal projection gives
\[
 X_s=\sigma Z,\qquad D=aZ+bW,
 \qquad a^2+b^2=\E D^2,
\]
where \(Z,W\) are independent standard normals, \(a\in\R\), and \(b\geq0\).
The residual after projection is independent of \(X_s\), since their joint Gaussian characteristic function factors.
Using
\(1-\phi_\tau'(x)^2=4\tau/(x^2+4\tau)\), we obtain
\[
 \E\!\left[D^2\bigl(1-\phi_\tau'(X_s)^2\bigr)\right]
 =
 4\tau\left(
 a^2\E\frac{Z^2}{\sigma^2Z^2+4\tau}
 +
 b^2\E\frac{1}{\sigma^2Z^2+4\tau}
 \right).
\]
The mixed term vanishes by independence and \(\E W=0\).
Cauchy--Schwarz, together with \(\E Z^2=1\) and \(\E Z^4=3\), gives
\[
 \E\frac{Z^2}{\sigma^2Z^2+4\tau}
 \geq
 \frac{(\E Z^2)^2}
      {\E[Z^2(\sigma^2Z^2+4\tau)]}
 =
 \frac{1}{3\sigma^2+4\tau}.
\]
Jensen's inequality gives
\[
 \E\frac{1}{\sigma^2Z^2+4\tau}
 \geq
 \frac{1}{\sigma^2+4\tau}.
\]
Both lower bounds are at least \((3Q_0+4\tau)^{-1}\).
Consequently,
\[
 \E[D^2\phi_\tau'(X_s)^2]
 \leq
 \frac{3Q_0}{3Q_0+4\tau}\,\E D^2.
\]
When \(\sigma=0\), the left-hand side is zero because \(X_s=0\) almost surely and \(\phi_\tau'(0)=0\).
Taking expectations in the integral inequality and integrating over \(s\) proves the result.
\end{proof}

\begin{thm}[The smoothed Gaussian fixed point]
\label{gaussian:thm:smoothed-gaussian-fixed-point}
In the Gaussian space of Notation~\ref{gaussian:setup:gaussian-tree-space}, define
\begin{equation}
 \mathcal B_\tau(X)
 =
 2G_L+
 J_{\tree}
 \bigl(\phi_\tau(X)-\E\phi_\tau(X)\bigr).
 \label{gaussian:eq:smoothed-map}
\end{equation}
The map \(\mathcal B_\tau\) preserves the closed ball
\[
 \mathbb B=\{X\in\mathcal H\mid \|X\|_2^2\leq2\pi\}
\]
and is a contraction there with factor \(\sqrt{\lambda_\tau}\), where
\begin{equation}
 \lambda_\tau=\frac{6\pi}{6\pi+4\tau}<1.
 \label{gaussian:eq:lambda-tau}
\end{equation}
It has a unique fixed point \(H_\tau\) in all of \(\mathcal H\).
This fixed point belongs to \(\mathbb B\), and its variance
\(Q_\tau=\E H_\tau^2\) lies in \([4,2\pi]\).

Let \(H_{\tau,0}=0\) and
\(H_{\tau,t+1}=\mathcal B_\tau(H_{\tau,t})\).
For every integer \(t\geq0\),
\begin{align}
 \|H_{\tau,t+1}-H_{\tau,t}\|_2^2
 &\leq4\lambda_\tau^t,
 \label{gaussian:eq:picard-increments}\\
 \|H_\tau-H_{\tau,t}\|_2
 &\leq
 \frac{2\lambda_\tau^{t/2}}{1-\sqrt{\lambda_\tau}}.
 \label{gaussian:eq:picard-tail}
\end{align}
The variances \(Q_{\tau,t}=\E H_{\tau,t}^2\) satisfy
\begin{equation}
 Q_{\tau,0}=0,\qquad
 Q_{\tau,t+1}
 =
 4+\Var\!\left(\sqrt{Q_{\tau,t}Z^2+4\tau}\right),
 \label{gaussian:eq:variance-recursion}
\end{equation}
and converge to the unique \(Q_\tau\in[4,2\pi]\) such that
\begin{equation}
 \E\sqrt{Q_\tau Z^2+4\tau}
 =
 2\sqrt{1+\tau}.
 \label{gaussian:eq:scalar-variance-identity}
\end{equation}
\end{thm}

\begin{proof}
We first control the scalar variance map.
For \(q\geq0\), set
\[
 \mu_\tau(q)=\E\sqrt{qZ^2+4\tau},
 \qquad
 F_\tau(q)=4+q+4\tau-\mu_\tau(q)^2.
\]
Differentiation under the expectation is justified by
\[
 \frac{Z^2}{2\sqrt{qZ^2+4\tau}}
 \leq \frac{Z^2}{4\sqrt\tau}.
\]
It follows that
\[
 F_\tau'(q)
 =
 1-\mu_\tau(q)\E\frac{Z^2}{\sqrt{qZ^2+4\tau}}.
\]

Put \(Y=\sqrt{qZ^2+4\tau}\).
Cauchy--Schwarz applied to \(\sqrt Y\) and \(\abs Z/\sqrt Y\) gives
\[
 \E Y\,\E\frac{Z^2}{Y}
 \geq(\E\abs Z)^2=\frac{2}{\pi}.
\]
The functions \(u\mapsto\sqrt{qu+4\tau}\) and
\(u\mapsto u/\sqrt{qu+4\tau}\) are increasing on \([0,\infty)\), as their derivatives show.
For an independent copy \(Z'\) of \(Z\), write \(Y'=\sqrt{q(Z')^2+4\tau}\).
Then
\[
 2\Cov\left(Y,\frac{Z^2}{Y}\right)
 =
 \E\left[
 (Y-Y')
 \left(\frac{Z^2}{Y}-\frac{(Z')^2}{Y'}\right)
 \right]
 \geq0.
\]
Since \(\E[Y(Z^2/Y)]=1\), this proves
\(\E Y\,\E[Z^2/Y]\leq1\).
Hence
\begin{equation}
 0\leq F_\tau'(q)\leq1-\frac{2}{\pi}
 \qquad(q\geq0).
 \label{gaussian:eq:variance-map-lipschitz}
\end{equation}
Together with \(F_\tau(0)=4\), this gives
\[
 4\leq F_\tau(q)
 \leq4+\left(1-\frac{2}{\pi}\right)q,
 \qquad
 F_\tau(q)\leq2\pi\quad(0\leq q\leq2\pi).
\]

Every \(X\in\mathcal H\) is centered Gaussian.
The centered function
\(\phi_\tau(X)-\E\phi_\tau(X)\) belongs to
\(\mathcal E_{\mathrm{even}}\).
The isometry in \eqref{gaussian:eq:grafting-norm} and orthogonality to \(G_L\) imply
\begin{equation}
 \|\mathcal B_\tau(X)\|_2^2
 =
 4+\Var(\phi_\tau(X))
 =
 F_\tau(\|X\|_2^2).
 \label{gaussian:eq:ball-invariance}
\end{equation}
Thus \(\mathcal B_\tau\) preserves \(\mathbb B\).

Centering is an orthogonal projection in \(L^2\).
Applying Lemma~\ref{gaussian:lem:strict-gaussian-contraction} with \(Q_0=2\pi\), followed by the grafting isometry, gives
\[
 \|\mathcal B_\tau(X)-\mathcal B_\tau(Y)\|_2
 \leq\sqrt{\lambda_\tau}\,\|X-Y\|_2
 \qquad(X,Y\in\mathbb B).
\]
The first Picard increment is \(2G_L\), of norm \(2\).
Iteration of this inequality proves \eqref{gaussian:eq:picard-increments}.
The increments have summable norms, so the iterates converge in the closed ball \(\mathbb B\).
Continuity of \(\mathcal B_\tau\) makes their limit a fixed point.
Summing the increment bounds gives \eqref{gaussian:eq:picard-tail}.
The contraction inequality applied to two fixed points proves uniqueness in \(\mathbb B\).

Let \(X\) be any fixed point in \(\mathcal H\), and set \(q=\|X\|_2^2\).
Equations \eqref{gaussian:eq:variance-map-lipschitz} and
\eqref{gaussian:eq:ball-invariance} give
\[
 q=F_\tau(q)
 \leq4+\left(1-\frac{2}{\pi}\right)q.
\]
Therefore \(q\leq2\pi\).
Also \(q=4+\Var(\phi_\tau(X))\geq4\).
Every fixed point lies in \(\mathbb B\), proving uniqueness in all of \(\mathcal H\).

Applying \eqref{gaussian:eq:ball-invariance} to the Picard iterates proves
\eqref{gaussian:eq:variance-recursion}.
By \eqref{gaussian:eq:variance-map-lipschitz}, the scalar map \(F_\tau\) is a strict contraction of \([0,2\pi]\) into itself.
Its iterates converge to its unique fixed point \(Q_\tau\in[4,2\pi]\).
The equation \(F_\tau(Q_\tau)=Q_\tau\) is equivalent to
\[
 \mu_\tau(Q_\tau)^2=4(1+\tau).
\]
Taking the positive square root gives \eqref{gaussian:eq:scalar-variance-identity}.
The function \(\mu_\tau\) is strictly increasing, since its integrand is strictly increasing in \(q\) whenever \(Z\neq0\).
Thus the scalar identity also characterizes \(Q_\tau\) uniquely.
\end{proof}

\begin{prop}[Uniform Hermite approximation of the smoothing function]
\label{gaussian:prop:uniform-hermite-approximation}
There are constants \(C_\tau,c_\tau>0\) such that, for every \(q\in[0,2\pi]\) and every integer \(M\geq2\), the \(L^2(N(0,q))\) orthogonal projection \(P_{M,q}\) of \(\phi_\tau\) onto the polynomials of degree at most \(M\) may be chosen even and satisfies
\begin{equation}
 \E\abs{
 \phi_\tau(\sqrt q\,Z)-P_{M,q}(\sqrt q\,Z)
 }^2
 \leq C_\tau\exp(-c_\tau\sqrt M).
 \label{gaussian:eq:uniform-hermite-approximation}
\end{equation}
For \(q=0\), one may take \(P_{M,0}\equiv2\sqrt\tau\).
\end{prop}

\begin{proof}
The projection formulation and derivative-based Hermite truncation estimate are given in \cite[Definition~2 and Lemma~1]{MFHF17}.
We establish the derivative bounds needed for uniformity over \(q\in[0,2\pi]\), and record the first-derivative identity for later use.

Set \(a=2\sqrt\tau\).
For every integer \(k\geq2\),
\begin{equation}
 \sup_{x\in\R}
 \left|\frac{d^k}{dx^k}\sqrt{x^2+a^2}\right|
 \leq\frac{k!}{2}a^{1-k}.
 \label{gaussian:eq:square-root-derivatives}
\end{equation}
To prove this, factor the second derivative as
\[
 a^2(x^2+a^2)^{-3/2}
 =
 a^2(x+\mathrm{i}a)^{-3/2}(x-\mathrm{i}a)^{-3/2},
\]
using conjugate branches whose product is positive on the real axis.
For the rising factorial
\((b)_j=b(b+1)\cdots(b+j-1)\), with \((b)_0=1\), the \(j\)-th derivative of either factor has absolute value
\((3/2)_j(x^2+a^2)^{-3/4-j/2}\).
Leibniz's formula for the remaining \(k-2\) derivatives therefore bounds the absolute value by
\[
 a^2(x^2+a^2)^{-(k+1)/2}
 \sum_{j=0}^{k-2}
 \binom{k-2}{j}(3/2)_j(3/2)_{k-2-j}.
\]
Comparison of coefficients in
\((1-z)^{-3/2}(1-z)^{-3/2}=(1-z)^{-3}\) gives
\[
 \sum_{j=0}^{k-2}
 \binom{k-2}{j}(3/2)_j(3/2)_{k-2-j}
 =
 (3)_{k-2}=\frac{k!}{2}.
\]
The inequality \(x^2+a^2\geq a^2\) now proves
\eqref{gaussian:eq:square-root-derivatives}.

For \(0<q\leq2\pi\), define
\[
 g_q(z)=\sqrt{qz^2+a^2},
 \qquad
 h_j(z)=\frac{\He_j(z)}{\sqrt{j!}},
 \qquad
 c_j(q)=\E[g_q(Z)h_j(Z)].
\]
The functions \(h_j\) form an orthonormal basis of Gaussian \(L^2\).
The Hermite generating function gives
\[
 \frac{d}{dz}\left(\He_\ell(z)e^{-z^2/2}\right)
 =
 -\He_{\ell+1}(z)e^{-z^2/2}.
\]
Applying this identity in \(k\) integrations by parts yields, for integers \(k\geq1\) and \(\ell\geq0\),
\begin{equation}
 \E[g_q^{(k)}(Z)h_\ell(Z)]
 =
 \sqrt{\frac{(\ell+k)!}{\ell!}}\,c_{\ell+k}(q).
 \label{gaussian:eq:hermite-integration-by-parts}
\end{equation}
The boundary terms vanish: \(g_q\) has at most linear growth, \(g_q'\) is bounded by \(\sqrt q\), and its higher derivatives are bounded by the chain rule and
\eqref{gaussian:eq:square-root-derivatives}.
In particular, completeness of the Hermite basis applied to \(g_q'\) gives
\begin{equation}
 \sum_{j\geq1}j\,c_j(q)^2
 =
 \E|g_q'(Z)|^2.
 \label{gaussian:eq:hermite-derivative-energy}
\end{equation}
For \(k\geq2\), Bessel's inequality and the same derivative bound give
\begin{equation}
 \sum_{j\geq k}\frac{j!}{(j-k)!}\,c_j(q)^2
 \leq
 \E|g_q^{(k)}(Z)|^2
 \leq
 \frac{a^2}{4}(k!)^2
 \left(\frac{2\pi}{a^2}\right)^k.
 \label{gaussian:eq:weighted-hermite-tail}
\end{equation}

Let \(2\leq k\leq M/2\).
For \(j>M\), each factor in \(j!/(j-k)!\) is at least \(M/2\).
Using \eqref{gaussian:eq:weighted-hermite-tail} and \(k!\leq k^k\), we obtain
\begin{equation}
 \sum_{j>M}c_j(q)^2
 \leq
 \frac{a^2}{4}
 \left(\frac{4\pi k^2}{a^2M}\right)^k.
 \label{gaussian:eq:optimized-hermite-tail}
\end{equation}
Set
\[
 b_\tau=\min\left\{\frac14,\frac{a}{4\sqrt{2\pi}}\right\},
 \qquad
 k=\lfloor b_\tau\sqrt M\rfloor.
\]
For sufficiently large \(M\), this integer satisfies
\(2\leq k\leq M/2\) and \(k\geq b_\tau\sqrt M/2\).
The bracket in \eqref{gaussian:eq:optimized-hermite-tail} is at most \(1/8\).
Thus the tail is bounded by
\[
 \frac{a^2}{4}
 \exp\left(-\frac{b_\tau\log 8}{2}\sqrt M\right).
\]
Increasing the multiplicative constant covers the remaining integers \(M\geq2\), uniformly in \(q\), because
\(\E g_q(Z)^2=q+4\tau\leq2\pi+4\tau\).

The odd coefficients \(c_j(q)\) vanish because \(g_q\) is even.
Its Hermite truncation is therefore even, and the substitution \(z=x/\sqrt q\) gives
\[
 P_{M,q}(x)
 =
 \sum_{j=0}^{M}c_j(q)h_j(x/\sqrt q).
\]
Parseval's identity identifies its squared error with
\(\sum_{j>M}c_j(q)^2\), proving
\eqref{gaussian:eq:uniform-hermite-approximation}.
At \(q=0\), the stated constant projection is exact.
\end{proof}

\begin{thm}[Canonical tree-degree tail]
\label{gaussian:thm:smoothed-tree-degree-tail}
Let \(H_\tau,Q_\tau\) be as in
Theorem~\ref{gaussian:thm:smoothed-gaussian-fixed-point}, and write
\begin{equation}
 H_\tau=\sum_{T\in\cT}c_TG_T,
 \qquad
 \sum_{T\in\cT}c_T^2\aut(T)=Q_\tau.
 \label{gaussian:eq:tree-expansion}
\end{equation}
Define
\begin{equation}
 \beta_{\tau,j}
 =
 \frac{1}{Q_\tau}
 \sum_{\substack{T\in\cT\\d(T)=j}}c_T^2\aut(T),
 \qquad
 R_\tau(z)=\sum_{j\geq1}\beta_{\tau,j}z^j.
 \label{gaussian:eq:tree-degree-law}
\end{equation}
Then \((\beta_{\tau,j})_{j\geq1}\) is a probability distribution supported on the positive odd integers.

Let
\[
 \psi_\tau(x)=\sqrt{Q_\tau x^2+4\tau},
 \qquad
 b_k=\frac{\E[\psi_\tau(Z)\He_k(Z)]}{k!},
\]
and define
\begin{equation}
 \alpha_{\tau,0}=\frac{4}{Q_\tau},
 \qquad
 \alpha_{\tau,k}=\frac{b_k^2k!}{Q_\tau}\quad(k\geq1),
 \qquad
 h_\tau(z)=\sum_{k\geq0}\alpha_{\tau,k}z^k.
 \label{gaussian:eq:offspring-law}
\end{equation}
The sequence \((\alpha_{\tau,k})_{k\geq0}\) is a probability distribution supported on the nonnegative even integers.
Its mean satisfies
\begin{equation}
 m_\tau
 =
 \sum_{k\geq0}k\alpha_{\tau,k}
 =
 \E\frac{Q_\tau Z^2}{Q_\tau Z^2+4\tau}
 <1.
 \label{gaussian:eq:subcritical-mean}
\end{equation}
For each \(z\in[0,1]\), the number \(R_\tau(z)\) is the unique solution in \([0,1]\) of
\begin{equation}
 R_\tau(z)=z\,h_\tau(R_\tau(z)).
 \label{gaussian:eq:tree-generating-function}
\end{equation}
Equivalently, \((\beta_{\tau,j})\) is the total-population law of a branching process started from one individual with offspring law \((\alpha_{\tau,k})\).

There are constants \(C_\tau,c_\tau>0\) such that, for every pair of integers \(K,D\geq1\),
\begin{align}
 \sum_{k\geq K}\alpha_{\tau,k}
 &\leq C_\tau\exp(-c_\tau\sqrt K),
 \label{gaussian:eq:offspring-tail}\\
 \sum_{j>D}\beta_{\tau,j}
 &\leq C_\tau\exp(-c_\tau D^{1/3}).
 \label{gaussian:eq:total-progeny-tail}
\end{align}
Consequently, the canonical degree projection
\[
 H_{\tau,\leq D}
 =
 \sum_{\substack{T\in\cT\\d(T)\leq D}}c_TG_T
\]
satisfies
\begin{align}
 \|H_\tau-H_{\tau,\leq D}\|_2^2
 &\leq
 Q_\tau C_\tau\exp(-c_\tau D^{1/3}),
 \label{gaussian:eq:canonical-projection-tail}\\
 \|H_{\tau,\leq D}-\mathcal B_\tau(H_{\tau,\leq D})\|_2^2
 &\leq
 4Q_\tau C_\tau\exp(-c_\tau D^{1/3}).
 \label{gaussian:eq:canonical-residual-tail}
\end{align}
All constants are for fixed \(\tau\).
These assertions concern the Gaussian tree space and make no finite-size claim uniform in a growing degree cutoff.
\end{thm}

\begin{proof}
Write \(H=H_\tau\) and \(Q=Q_\tau\).
We use the fixed-point identity \(\mathcal B_\tau(H)=H\) in
\eqref{gaussian:eq:smoothed-map} to relate Hermite order to tree degree.

\medskip
\noindent\textbf{Step 1.}
The uniform Hermite estimate gives a subcritical offspring law with a quantitative tail.

Orthogonality in \eqref{gaussian:eq:tree-expansion} shows that the coefficients \(\beta_{\tau,j}\) are nonnegative and sum to one.
Every tree in \(\cT\) has an even number of edges, since the number of edges is the sum of its internal arities.
It therefore has an odd number of vertices, proving the support assertion for \(\beta_\tau\).

The Hermite expansion
\[
 \psi_\tau(Z)=\sum_{k\geq0}b_k\He_k(Z)
\]
converges in Gaussian \(L^2\).
Its odd coefficients vanish by evenness.
The variance equation from
Theorem~\ref{gaussian:thm:smoothed-gaussian-fixed-point} gives
\[
 Q
 =
 4+\Var(\psi_\tau(Z))
 =
 4+\sum_{k\geq1}b_k^2k!.
\]
Thus \(\alpha_\tau\) is a probability distribution supported on the nonnegative even integers.

For \(K\geq3\), Parseval's identity and
\eqref{gaussian:eq:uniform-hermite-approximation}, applied with \(q=Q\) and \(M=K-1\), give
\[
 \begin{aligned}
 \sum_{k\geq K}\alpha_{\tau,k}
 &=
 \frac1Q\E\abs{
 \phi_\tau(\sqrt Q\,Z)-P_{K-1,Q}(\sqrt Q\,Z)
 }^2\\
 &\leq
 \frac{C_\tau}{Q}\exp(-c_\tau\sqrt{K-1}).
 \end{aligned}
\]
Since \(Q\geq4\) and \(\sqrt{K-1}\geq\sqrt{K/2}\), adjusting the constants proves
\eqref{gaussian:eq:offspring-tail} for \(K\geq3\).
Increasing the multiplicative constant covers \(K=1,2\).
Summing polynomially weighted tails now shows that the offspring law has finite moments of every order.

The normalized coefficients in
\eqref{gaussian:eq:hermite-derivative-energy}, at \(q=Q\), are
\(c_k(Q)=\sqrt{k!}\,b_k\).
Hence that identity gives
\[
 \sum_{k\geq1}k b_k^2k!
 =
 \E|\psi_\tau'(Z)|^2,
 \qquad
 \psi_\tau'(x)=\frac{Qx}{\sqrt{Qx^2+4\tau}}.
\]
Division by \(Q\) proves \eqref{gaussian:eq:subcritical-mean}.
The integrand there is nonnegative and strictly less than \(1\), so \(m_\tau<1\).

\medskip
\noindent\textbf{Step 2.}
Grafting identifies the tree-degree law with total population for this offspring distribution.

For \(z\in[0,1]\), define a contraction on \(\mathcal H\) by
\[
 D_zG_T=z^{d(T)}G_T.
\]
Let \(G_T'\) be independent copies of the Gaussian coordinates, and set
\[
 G_T^{(z)}
 =
 z^{d(T)}G_T+\sqrt{1-z^{2d(T)}}\,G_T'.
\]
The Gaussian noise operator \(O_z\) replaces \(G_T\) by \(G_T^{(z)}\) and then takes conditional expectation given the original coordinates.
On a Wick monomial indexed by \((r_A)\), it acts by multiplication by
\(z^{\sum_A r_A d(A)}\).
Its grafted tree has
\(d(T)=1+\sum_A r_A d(A)\).
It follows that
\begin{equation}
 D_zJ_{\tree}
 =
 z\,J_{\tree}O_z
 \quad\text{on }\mathcal E_{\mathrm{even}}.
 \label{gaussian:eq:grafting-intertwining}
\end{equation}
This identity holds first on the Wick basis and then on its closure, because the operators are bounded.

Set
\[
 S_\tau(z)=\inner{H}{D_zH}=Q R_\tau(z).
\]
The leaf and nonleaf subspaces remain orthogonal under \(D_z\).
Applying the \(D_z\)-quadratic form to the fixed-point identity and using
\eqref{gaussian:eq:grafting-intertwining} therefore gives
\begin{equation}
 S_\tau(z)
 =
 4z+
 z\,\Cov\!\left(
 \phi_\tau(H(G)),
 \phi_\tau(H(G^{(z)}))
 \right).
 \label{gaussian:eq:noise-covariance}
\end{equation}
The linear variables \(H(G)\) and \(H(G^{(z)})\) are centered jointly Gaussian.
Both have variance \(Q\), and their covariance is \(S_\tau(z)\).
Their correlation is consequently \(R_\tau(z)\).

For standard normals \(Z,Z'\) with correlation \(\rho\), the joint expectation of their Hermite generating functions is
\[
 \E\left[
 e^{sZ-s^2/2}e^{tZ'-t^2/2}
 \right]=e^{\rho st}.
\]
Comparing coefficients yields
\[
 \E[\He_k(Z)\He_\ell(Z')]
 =
 \one_{\{k=\ell\}}\,k!\rho^k.
\]
The \(L^2\) convergence of the Hermite expansions and Cauchy--Schwarz allow passage from polynomial truncations to \(\psi_\tau\).
Thus \eqref{gaussian:eq:noise-covariance} becomes
\[
 Q R_\tau(z)
 =
 4z+z\sum_{k\geq1}b_k^2k!R_\tau(z)^k.
\]
Dividing by \(Q\) proves \eqref{gaussian:eq:tree-generating-function}.

Nonnegativity of the coefficients of \(h_\tau\) gives
\[
 0\leq h_\tau'(r)\leq h_\tau'(1)=m_\tau<1
 \qquad(0\leq r\leq1).
\]
Moreover, \(h_\tau\) maps \([0,1]\) into \([0,1]\).
For each \(z\in[0,1]\), the map \(r\mapsto zh_\tau(r)\) is therefore an \(m_\tau\)-contraction of \([0,1]\).
This proves uniqueness of the solution in
\eqref{gaussian:eq:tree-generating-function}.

Consider a branching process with one initial individual and offspring law \(\alpha_\tau\).
Conditioning on the preceding generation shows that the expected size of generation \(\ell\) is \(m_\tau^\ell\).
Its probability of being nonempty is at most \(m_\tau^\ell\).
Hence extinction occurs after finitely many generations almost surely.
Its total population \(T_{\mathrm{pop}}\) is therefore finite almost surely.

Condition on the offspring of the initial individual.
The descendant processes are independent copies of the original process, so
\[
 \E z^{T_{\mathrm{pop}}}
 =
 z\,h_\tau(\E z^{T_{\mathrm{pop}}})
 \qquad(0\leq z\leq1).
\]
Uniqueness in \eqref{gaussian:eq:tree-generating-function} identifies this generating function with \(R_\tau\).
Equality of its power-series coefficients proves that \(\beta_\tau\) is the total-population law.

\medskip
\noindent\textbf{Step 3.}
A truncated exploration bound gives the total-population tail.

List independent offspring variables \(\xi_1,\xi_2,\ldots\) in breadth-first exploration order, appending independent variables after extinction.
If \(T_{\mathrm{pop}}>D\), the queue after exploring \(D\) individuals is positive.
Since its size is an integer,
\[
 1+\sum_{i=1}^D\xi_i-D>0
 \quad\Longrightarrow\quad
 \sum_{i=1}^D\xi_i\geq D.
\]
For large integers \(D\), put
\[
 B_D=D^{2/3},
 \qquad
 \xi_i'=\xi_i\one_{\{\xi_i\leq B_D\}}.
\]
The union bound and \eqref{gaussian:eq:offspring-tail} give
\[
 \Prob\left(\max_{1\leq i\leq D}\xi_i>B_D\right)
 \leq DC_\tau\exp(-c_\tau D^{1/3}).
\]
On the complementary event, survival beyond \(D\) forces
\(\sum_{i=1}^D\xi_i'\geq D\).

Define
\[
 \mu_D=\E\xi_i'\leq m_\tau,
 \qquad
 \Delta_\tau=1-m_\tau>0,
 \qquad
 Y_i=\xi_i'-\mu_D.
\]
For sufficiently large \(D\),
\[
 \E Y_i=0,\qquad
 |Y_i|\leq2B_D,\qquad
 \E Y_i^2\leq M_{2,\tau}:=\E\xi_i^2<\infty.
\]
Take \(t_D=(2B_D)^{-1}\).
For \(|x|\leq1\), the exponential series gives
\[
 e^x\leq1+x+x^2,
\]
because the absolute sum of its terms of degree at least two is at most
\((e-2)x^2<x^2\).
It follows that
\[
 \E e^{t_DY_i}
 \leq1+t_D^2M_{2,\tau}
 \leq\exp(t_D^2M_{2,\tau}).
\]
Independence and Markov's inequality now yield
\[
 \begin{aligned}
 \Prob\left(\sum_{i=1}^D\xi_i'\geq D\right)
 &\leq
 \Prob\left(\sum_{i=1}^DY_i\geq\Delta_\tau D\right)\\
 &\leq
 \exp\left(-t_D\Delta_\tau D+Dt_D^2M_{2,\tau}\right).
 \end{aligned}
\]
For sufficiently large \(D\), the inequality
\(t_DM_{2,\tau}\leq\Delta_\tau/2\) holds.
The last expression is then at most
\[
 \exp\left(-\frac{\Delta_\tau}{4}D^{1/3}\right).
\]
Combining the two exceptional events gives
\[
 \Prob(T_{\mathrm{pop}}>D)
 \leq
 DC_\tau e^{-c_\tau D^{1/3}}
 +
 e^{-\Delta_\tau D^{1/3}/4}.
\]
Since \(\log D=o(D^{1/3})\), the factor \(D\) is absorbed by decreasing the positive exponential constant.
Increasing the multiplicative constant handles the remaining integers \(D\geq1\).
The identification of \(\beta_\tau\) in Step 2 proves
\eqref{gaussian:eq:total-progeny-tail}.

\medskip
\noindent\textbf{Step 4.}
Orthogonal projection and contraction give the two truncation estimates.

By \eqref{gaussian:eq:tree-expansion} and
\eqref{gaussian:eq:tree-degree-law},
\[
 \|H_\tau-H_{\tau,\leq D}\|_2^2
 =
 Q_\tau\sum_{j>D}\beta_{\tau,j}.
\]
The bound \eqref{gaussian:eq:total-progeny-tail} proves
\eqref{gaussian:eq:canonical-projection-tail}.
Both \(H_\tau\) and its orthogonal projection \(H_{\tau,\leq D}\) belong to \(\mathbb B\).
Since \(\mathcal B_\tau(H_\tau)=H_\tau\), the contraction in
Theorem~\ref{gaussian:thm:smoothed-gaussian-fixed-point} gives
\[
 \begin{aligned}
 \|H_{\tau,\leq D}-\mathcal B_\tau(H_{\tau,\leq D})\|_2
 &\leq
 \|H_{\tau,\leq D}-H_\tau\|_2
 +
 \|\mathcal B_\tau(H_\tau)
       -\mathcal B_\tau(H_{\tau,\leq D})\|_2\\
 &\leq
 (1+\sqrt{\lambda_\tau})
 \|H_{\tau,\leq D}-H_\tau\|_2\\
 &\leq
 2\|H_{\tau,\leq D}-H_\tau\|_2.
 \end{aligned}
\]
Squaring and applying \eqref{gaussian:eq:canonical-projection-tail} proves
\eqref{gaussian:eq:canonical-residual-tail}.
\end{proof}

\section{Growing approximants on the original signs}
\label{sec:tech-transfer}
\label{transfer:sec:growing-approximants}

In this section, we prove Proposition~\ref{transfer:prop:approximants} by combining growing-degree Wick grafting with a quantitative comparison of the smoothed activation.

\begin{notation}[Canonical fields and Fourier operators]
\label{transfer:notation:canonical-conventions}
For every odd integer \(n\geq3\), use the Fourier matrix \(F\),
the multiplier \(a\), the operators \(K_0,J\), and the independent
fair signs \(\epsilon_j=1-2\xi_j\) from Setup~\ref{setup:fourier}.
In particular, \(a_0=0\), and a negative sign represents an edge
difference.
Write \(\1\) for the constant vector and \(e_0\) for the coordinate
vector at zero. The Fourier identities give
\begin{equation}
 K_0^2=\Id-J,\qquad
 \norm{K_0}_{\ell_2\to\ell_2}\leq1.
 \label{transfer:eq:K-contraction}
\end{equation}
Every canonical field uses the same signs as \(K_0\).

For a rooted tree \(T\), let \(d(T)\) be its number of nodes and \(\aut(T)\) its rooted automorphism count. We use the real even canonical fields \(Z_{T,n}\) of Theorem~\ref{trees:thm:tree-flow-comparison}; thus \(Z_{T,n}(k)\) is the coordinate denoted \(Z_{T,k}\) there. Internal vertices have at least two children. The Gaussian tree coordinates and Wick-grafting isometry of Theorem~\ref{trees:thm:wick-grafting} are denoted by \(G_T\) and \(J_{\tree}\), respectively, with
\begin{equation}
 \E[G_TG_S]=\1_{\{T\simeq S\}}\aut(T).
 \label{transfer:eq:gaussian-tree-covariance}
\end{equation}
The one-node tree is denoted by \(\mathrm L\). Its Gaussian coordinate is standard normal and is orthogonal to the range of \(J_{\tree}\).

For an integer \(D\geq1\), put
\begin{equation}
 P_D=(4D)^{4D},\qquad A_D=(4D)^{80D},
 \label{transfer:eq:PD-AD}
\end{equation}
so \(A_D\) agrees with \eqref{trees:eq:AD-definition}. For scalar random variables, \(\norm{\cdot}_p\) denotes the probabilistic \(L^p\) norm; for vectors, \(\norm{\cdot}_2\) denotes the Euclidean norm. Constants \(C_\tau,c_\tau\) are positive and finite, depend only on \(\tau>0\), and may change between occurrences.
\end{notation}

\subsection{Polynomial moments and Wick grafting}

The first proposition quantifies the Fourier-convolution identity of Theorem~\ref{trees:thm:wick-grafting} when both polynomial degrees vary with \(n\).

\begin{prop}[Growing-degree polynomial comparison and Wick grafting]
\label{transfer:prop:growing-wick}
There is an absolute constant \(C>0\) such that the following statements hold for every odd \(n\geq3\) and all integers \(D\geq1\), \(M\geq2\).

Let \(T_i\) be pairwise nonisomorphic rooted trees with at most \(D\) nodes and internal arities at least two, and set
\[
 V_{i,n}(k)=\frac{Z_{T_i,n}(k)}{\sqrt{\aut(T_i)}}.
\]
Let \(R\) be a polynomial of total degree at most \(M\) in these variables at one nonzero output mode, or at two nonzero output modes that are not equal up to sign. Write \(\abs R_{\coef}\) for the sum of the absolute values of its ordinary monomial coefficients. Let \((N_i)\), or two independent copies of \((N_i)\), be the corresponding independent standard Gaussian variables. Then
\begin{equation}
 \abs{\E R(V)-\E R(N)}
 \leq
 \abs R_{\coef}
 \frac{\exp\bigl(C(D+M+2)^4\bigr)}{\sqrt n}.
 \label{transfer:eq:polynomial-comparison}
\end{equation}
For arbitrary choices of the output modes, including zero and conjugate pairs,
\begin{equation}
 \E\abs{R(V)}
 \leq
 \abs R_{\coef}\exp\bigl(C(D+M+2)^4\bigr).
 \label{transfer:eq:arbitrary-mode-moment}
\end{equation}

Let \(r_i\) be nonnegative integers with
\[
 2\leq r=\sum_i r_i\leq M,
\]
and let \(T\) be formed by attaching \(r_i\) copies of \(T_i\) to a new root. For the probabilists' Hermite polynomials \(\He_j\), define
\begin{equation}
 W_n(k)=
 \prod_i\frac{\He_{r_i}(V_{i,n}(k))}{\sqrt{r_i!}}.
 \label{transfer:eq:normalized-wick-field}
\end{equation}
Then
\begin{equation}
 \E\frac1n
 \norm{K_0W_n-\frac{Z_{T,n}}{\sqrt{\aut(T)}}}_2^2
 \leq
 \frac{\exp\bigl(C[(D+1)(M+1)]^4\bigr)}{\sqrt n}.
 \label{transfer:eq:wick-graft-bound}
\end{equation}

Assume further that all \(T_i\) have even internal arities, and let
\[
 X_n=\sum_i c_iZ_{T_i,n},\qquad X=\sum_i c_iG_{T_i},
\]
where \(c_i\) are deterministic real coefficients. For every even polynomial \(P\) of degree at most \(M\), let \(Y_{P,n}\) be the finite canonical field representing
\[
 J_{\tree}\bigl(P(X)-\E P(X)\bigr)
\]
by normalized Wick grafting. Then
\begin{equation}
 \E\frac1n\norm{K_0P(X_n)-Y_{P,n}}_2^2
 \leq
 \E[P(X)^2]\,
 \frac{\exp\bigl(C[(D+1)(M+1)]^4\bigr)}{\sqrt n}.
 \label{transfer:eq:polynomial-graft-bound}
\end{equation}
Every tree in \(Y_{P,n}\) has even internal arities and at most \(1+DM\) nodes. The constants are independent of \(n\), the number of tree types, and the coefficients. All estimates permit deterministic dependence of these quantities, \(D\), and \(M\) on \(n\).
\end{prop}

\begin{proof}
We use the scalar polynomial comparison of Theorem~\ref{trees:thm:polynomial-moment-comparison}, with the explicit coordinate bounds from Theorem~\ref{trees:thm:growing-degree-comparison}. The proof first controls ordinary monomials, then normalized Wick monomials, and finally their finite Hermite expansions.

For the normalized Wick monomial in \eqref{transfer:eq:normalized-wick-field},
we first record the degree bookkeeping used below.  If $U$ is a Walsh
polynomial of maximal leaf-sign degree $m$, then
\begin{equation}
 \deg_{\mathrm{Walsh}}(K_0U)\leq m+1,
 \label{transfer:eq:K0-walsh-degree}
\end{equation}
because each kernel coefficient $a_s$ is linear in the corresponding leaf-sign
variable.  Moreover, the degree-$m+1$ part can only arise from the degree-$m$
part of $U$; all lower Walsh degrees of $U$ remain at most $m$.  Thus the
maximal Walsh component of $K_0W_n$ is determined solely by the top Hermite
terms $\He_{r_i}(x)=x^{r_i}+O(x^{r_i-2})$.

For the normalized Wick monomial in \eqref{transfer:eq:normalized-wick-field},
its top Walsh-degree component is exactly the normalized canonical field of
the grafted tree. Here $\operatorname{TopWalsh}$ denotes the projection onto the maximal Walsh degree in the leaf signs:
\begin{equation}\label{transfer:eq:top-walsh-wick}
 \operatorname{TopWalsh}\!\left(K_0W_n\right)
 =\frac{Z_{T,n}}{\sqrt{\operatorname{aut}(T)}}.
\end{equation}
Indeed, the top Hermite term satisfies
$\He_{r_i}(x)=x^{r_i}+\text{(terms of degree at most $r_i-2$)}$.
Consequently the product of the top-degree pieces is
$\prod_iV_{i,n}^{r_i}$; under $K_0$ this is the flow polynomial obtained by
grafting the $r_i$ copies of $T_i$ to a new root. The normalization in
\eqref{transfer:eq:normalized-wick-field} and the rooted automorphism
count give exactly the factor $1/\sqrt{\operatorname{aut}(T)}$.
All remaining Hermite terms have strictly smaller total degree and are
handled by the finite Hermite expansion below.

\medskip
\noindent\textbf{Step 1. Polynomial moments.}
The coordinate estimates remain uniform at growing degree.
For each standardized canonical coordinate, the proof of Theorem~\ref{trees:thm:growing-degree-comparison} gives an \(L^2\) bound \(P_D\), a largest discrete-derivative influence at most \(P_D^3/n\), and a standard deviation at most \(P_D^5/\sqrt n\) for each derivative covariance. Standardization cannot enlarge these bounds because \(\aut(T_i)\geq1\); a derivative covariance is divided by the product of the two normalizing square roots. The sign-cube estimate \eqref{trees:eq:sign-hypercontractivity} therefore gives, at every mode,
For later use in Theorem~\ref{trees:thm:polynomial-moment-comparison},
we take
\[
 B=P_D,\qquad
 u=\frac{P_D^3}{n},\qquad
 v_{\mathrm{cov}}=\frac{P_D^5}{\sqrt n}.
\]
Thus, for a degree-$r$ moment, the two terms in that theorem become
\[
 B^{r-2}v_{\mathrm{cov}}
 =\frac{P_D^{r+3}}{\sqrt n},
 \qquad
 B^{r-1}\sqrt u
 =\frac{P_D^{r+1/2}}{\sqrt n}.
\]
These are precisely the two error scales used below.
\begin{equation}
 \norm{V_{i,n}(k)}_{2s}
 \leq(2s-1)^{D/2}P_D
 \qquad(s\geq1,\ s\ \text{an integer}).
 \label{transfer:eq:canonical-hypercontractivity}
\end{equation}

Consider a monomial of degree \(r\), where \(2\leq r\leq M\), and list its factors as \(V_1,\ldots,V_r\). Repetitions are allowed. The factors belong to either one admissible mode block or two admissible mode blocks. For \(\eta\in\{-1,1\}^r\), put
\[
 S_\eta=\frac1r\sum_{j=1}^r\eta_jV_j.
\]
Its total absolute coefficient weight is at most one. Applying Theorem~\ref{trees:thm:polynomial-moment-comparison} with the coordinate, influence, and derivative-covariance bounds above yields
\begin{equation}
 \abs{\E S_\eta^r-\E G_\eta^r}
 \leq
 \frac{(32r)^{(D+6)r}
       \bigl(P_D^{r+3}+P_D^{r+1/2}\bigr)}{\sqrt n},
 \label{transfer:eq:moment-matching-variance}
\end{equation}
where \(G_\eta\) is centered normal with variance \(\E S_\eta^2\).

Let \(G_{\eta,0}\) be the corresponding combination of the independent comparison Gaussian coordinates, and let \(V_0=\E G_{\eta,0}^2\). The coefficient weight gives \(V_0\leq1\). The one-mode and two-mode covariance bounds in Theorem~\ref{trees:thm:growing-degree-comparison}, together with the \(L^2\) triangle inequality, give
\begin{equation}
 \abs{\E S_\eta^2-V_0}\leq\frac{P_D^{20}}{\sqrt n},
 \qquad
 \E S_\eta^2\leq P_D^2.
 \label{transfer:eq:variance-replacement}
\end{equation}
Both Gaussian \(r\)-th moments vanish for odd \(r\). For \(r=2j\), integration by parts against the Gaussian density gives
\[
 \E G_\eta^r=(2j-1)!!\,(\E S_\eta^2)^j,\qquad
 \E G_{\eta,0}^r=(2j-1)!!\,V_0^j.
\]
Since
\[
 \abs{x^j-y^j}\leq j\abs{x-y}\max(x,y)^{j-1}
 \qquad(x,y\geq0),
\]
the cost of replacing \(G_\eta\) by \(G_{\eta,0}\) in
\eqref{transfer:eq:moment-matching-variance} is at most
\[
 \frac{r^rP_D^{r+20}}{\sqrt n}.
\]
The logarithms of the explicit factors in these two error bounds are bounded by an absolute constant times
\[
 D(r+20)\log(4D)+(D+6)r\log(32r)+r\log r
 \leq C_1(D+r+2)^4.
\]

To recover the original product, use polarization:
\begin{equation}
 \prod_{j=1}^r x_j
 =
 \frac1{2^rr!}
 \sum_{\eta\in\{-1,1\}^r}
 \left(\prod_{j=1}^r\eta_j\right)
 \left(\sum_{j=1}^r\eta_jx_j\right)^r.
 \label{transfer:eq:polarization}
\end{equation}
Expansion of the powers proves this identity. The sign sum removes every term unless each variable has odd exponent. Total degree \(r\) then forces all exponents to equal one, with multinomial coefficient \(r!\). Applying \eqref{transfer:eq:polarization} to both families of random variables multiplies the scalar error by at most \(r^r/r!\leq r^r\). Thus
\[
 \abs{\E\prod_{j=1}^rV_j-\E\prod_{j=1}^rN_j}
 \leq
 \frac{\exp\bigl(C_2(D+r+2)^4\bigr)}{\sqrt n}.
\]
The degree-zero and degree-one moments agree exactly. Summing against the coefficients of \(R\) proves \eqref{transfer:eq:polynomial-comparison}.

For arbitrary modes, H\"older's inequality and
\eqref{transfer:eq:canonical-hypercontractivity} give
\[
 \E\abs{\prod_{j=1}^rV_j}
 \leq\left[(2M-1)^{D/2}P_D\right]^r
 \qquad(1\leq r\leq M).
\]
The logarithm of this bound is at most \(C(D+M+2)^4\), while the constant monomial has expectation one. Summation against the absolute coefficients of \(R\) proves \eqref{transfer:eq:arbitrary-mode-moment}.

\medskip
\noindent\textbf{Step 2. A normalized Wick monomial.}
The coefficient bounds allow us to compare both the spatial energy and the spatial mean of \(W_n\).
The generating function \(\exp(tx-t^2/2)\) gives
\begin{equation}
 \He_j(x)
 =
 j!\sum_{\ell=0}^{\lfloor j/2\rfloor}
 \frac{(-1)^\ell x^{j-2\ell}}
      {2^\ell\ell!(j-2\ell)!}.
 \label{transfer:eq:hermite-expansion}
\end{equation}
After division by \(j!\), the absolute coefficient sum is the coefficient of \(t^j\) in \(\exp(t+t^2/2)\), and hence is at most \(e^{3/2}\). Therefore the coefficient sum of \(\He_j/\sqrt{j!}\) is at most \(e^{3/2}\sqrt{j!}\). At most \(r\) indices in \eqref{transfer:eq:normalized-wick-field} are active, and \(\prod_i r_i!\leq r!\). It follows that
\begin{equation}
 \abs{W_n}_{\coef}\leq e^{3r/2}\sqrt{r!},
 \qquad
 \abs{W_n^2}_{\coef},\
 \abs{W_n(k)W_n(\ell)}_{\coef}
 \leq e^{3r}r!.
 \label{transfer:eq:wick-coefficient-sums}
\end{equation}
These bounds are at most \(\exp(C_3M\log(M+2))\), and the squared and two-block polynomials have degree at most \(2M\).

Hermite orthogonality gives
\[
 \E[\He_j(Z)\He_\ell(Z)]=\1_{\{j=\ell\}}j!,
\]
for a standard normal \(Z\). This identity follows by comparing coefficients in the expectation of the product of two Hermite generating functions. Independence then gives mean zero and second moment one for a normalized Wick monomial of positive degree.

Apply \eqref{transfer:eq:polynomial-comparison} with degree parameter \(2M\), using \eqref{transfer:eq:wick-coefficient-sums}. For nonzero \(k\), and for nonzero \(k,\ell\) with \(\ell\notin\{k,-k\}\), respectively, we obtain
\[
 \abs{\E W_n(k)^2-1}
 \leq\frac{\exp(C_4(D+M+2)^4)}{\sqrt n},
 \qquad
 \abs{\E[W_n(k)W_n(\ell)]}
 \leq\frac{\exp(C_4(D+M+2)^4)}{\sqrt n}.
\]
The second comparison Gaussian expectation vanishes because its two blocks are independent and centered. Estimate \eqref{transfer:eq:arbitrary-mode-moment} controls the exceptional modes. There are at most \(4n\) exceptional ordered pairs: at most \(2n\) have a zero coordinate, and at most \(2n\) satisfy \(\ell=k\) or \(\ell=-k\). Their contribution after division by \(n^2\) is at most the same exponential factor divided by \(n\). Consequently,
\begin{align}
 \abs{\E\left[\frac1n\sum_kW_n(k)^2\right]-1}
 &\leq
 \frac{\exp(C_5(D+M+2)^4)}{\sqrt n},
 \label{transfer:eq:wick-mean-square}\\
 \E\left[\left(\frac1n\sum_kW_n(k)\right)^2\right]
 &\leq
 \frac{\exp(C_5(D+M+2)^4)}{\sqrt n}.
 \label{transfer:eq:wick-spatial-mean}
\end{align}
For the second bound, sum absolute expectations over all ordered pairs; its left-hand side is nonnegative.

The automorphisms of the grafted tree act by permuting isomorphic children and by automorphisms within each child. Thus the normalization of Theorem~\ref{trees:thm:wick-grafting} is
\begin{equation}
 \aut(T)=\prod_i r_i!\,\aut(T_i)^{r_i},
 \qquad
 J_{\tree}\left(
   \prod_i
   \frac{\He_{r_i}(G_{T_i}/\sqrt{\aut(T_i)})}{\sqrt{r_i!}}
 \right)
 =
 \frac{G_T}{\sqrt{\aut(T)}}.
 \label{transfer:eq:wick-isometry}
\end{equation}
The exact Fourier-convolution calculation in that theorem identifies the highest Walsh-degree component of \(K_0W_n\) with
\(Z_{T,n}/\sqrt{\aut(T)}\). Its degree is
\[
 1+\sum_i r_i d(T_i).
\]
Lower Hermite terms and repeated sign labels have smaller Walsh degree. The convolution calculation is finite and remains valid when the degrees change. Walsh orthogonality therefore gives
\begin{equation}
 \E\frac1n
 \norm{K_0W_n-\frac{Z_{T,n}}{\sqrt{\aut(T)}}}_2^2
 =
 \E\frac1n\norm{K_0W_n}_2^2
 -
 \E\frac{\norm{Z_{T,n}}_2^2}{n\aut(T)}.
 \label{transfer:eq:exact-wick-defect}
\end{equation}
By \eqref{transfer:eq:K-contraction},
\[
 \E\frac1n\norm{K_0W_n}_2^2
 =
 \E\left[\frac1n\sum_kW_n(k)^2\right]
 -
 \E\left[\left(\frac1n\sum_kW_n(k)\right)^2\right].
\]
Equations \eqref{transfer:eq:wick-mean-square} and
\eqref{transfer:eq:wick-spatial-mean} bound its difference from one by twice their common error.

The grafted tree has at most \(K=1+DM\) nodes. The averaged-energy estimate
\eqref{trees:eq:growing-energy}, applied to its normalized canonical field, bounds the difference between the final term in
\eqref{transfer:eq:exact-wick-defect} and one by
\[
 \frac{A_K}{\sqrt n}=\frac{(4K)^{80K}}{\sqrt n}.
\]
Since
\[
 K\leq(D+1)(M+1),\qquad
 D+M+2\leq2(D+1)(M+1),
\]
both errors are bounded by the right-hand side of
\eqref{transfer:eq:wick-graft-bound}, after increasing \(C\). The left-hand side of \eqref{transfer:eq:exact-wick-defect} is nonnegative, so this proves the normalized Wick estimate.

\medskip
\noindent\textbf{Step 3. An even polynomial.}
A finite Hermite expansion reduces the polynomial assertion to Step 2.
Set \(N_i=G_{T_i}/\sqrt{\aut(T_i)}\). These variables are independent standard normals. The triangular monic Hermite expansion and orthogonality give
\begin{equation}
 P(X)=
 \sum_\alpha b_\alpha
 \prod_i\frac{\He_{\alpha_i}(N_i)}{\sqrt{\alpha_i!}},
 \qquad
 \sum_\alpha b_\alpha^2=\E[P(X)^2],
 \qquad
 \abs\alpha\leq M.
 \label{transfer:eq:multivariate-hermite-expansion}
\end{equation}
Simultaneous negation of all coordinates leaves \(P(X)\) unchanged, so only even total degrees occur. The constant term is killed by \(K_0\). Every remaining degree is at least two, and
\eqref{transfer:eq:wick-isometry} maps its normalized Wick monomial to the normalized coordinate of the grafted tree. Replacing these Gaussian coordinates by their canonical fields defines \(Y_{P,n}\). Every resulting tree has even internal arities and at most \(1+DM\) nodes.

The rooted-tree count in Theorem~\ref{trees:thm:growing-degree-comparison} is \(q_D^{\tree}\leq4^D\). Hence the number of multi-indices in
\eqref{transfer:eq:multivariate-hermite-expansion} is at most
\[
 \binom{q_D^{\tree}+M}{M}
 \leq(q_D^{\tree}+M)^M
 \leq\exp\bigl(C_6M(D+\log(M+2))\bigr).
\]
Cauchy--Schwarz bounds \(\sum_\alpha\abs{b_\alpha}\) by the square root of this count times \(\E[P(X)^2]^{1/2}\). Apply
\eqref{transfer:eq:wick-graft-bound} to each nonconstant term and use the triangle inequality in the norm
\[
 \left(\E\frac1n\norm{\cdot}_2^2\right)^{1/2}.
\]
After squaring, the extra factor is at most the multi-index count. Increasing \(C\) absorbs it into
\(\exp(C[(D+1)(M+1)]^4)\), proving
\eqref{transfer:eq:polynomial-graft-bound}. All counts and moment estimates are finite and uniform, so deterministic changes of the degrees and coefficients with \(n\) are permitted.
\end{proof}

\subsection{Transfer of the smoothed activation}

The next proposition compares the scalar activation with its Gaussian polynomial projection, including the zero output coordinate after spatial averaging.

\begin{prop}[Quantitative smooth activation transfer]
\label{transfer:prop:smooth-transfer}
Fix \(\tau>0\), and let \(D\geq1\), \(M\geq2\) be integers. For deterministic real coefficients indexed by pairwise nonisomorphic rooted trees with at most \(D\) nodes and internal arities at least two, put
\[
 X_n(k)=\sum_Tc_TZ_{T,n}(k),\qquad
 q=\sum_Tc_T^2\aut(T),
 \qquad 1\leq q\leq2\pi.
\]
Let \(\phi_\tau(x)=\sqrt{x^2+4\tau}\), and let \(P_{M,q}\) be its Gaussian \(L^2\) projection onto polynomials of degree at most \(M\) under \(N(0,q)\). If \(n\geq3\) is odd and \(\sqrt n\geq20A_D\), then, for a standard normal \(Z\),
\begin{align}
 &\left|
 \frac1n\sum_{k=0}^{n-1}
 \E\abs{P_{M,q}(X_n(k))-\phi_\tau(X_n(k))}^2
 -
 \E\abs{P_{M,q}(\sqrt q\,Z)-\phi_\tau(\sqrt q\,Z)}^2
 \right|
 \notag\\
 &\hspace{25mm}\leq
 \exp\bigl(C_\tau DM\log(DM+2)\bigr)n^{-1/6}.
 \label{transfer:eq:activation-transfer}
\end{align}
Consequently,
\begin{align}
 \frac1n\sum_{k=0}^{n-1}
 \E\abs{P_{M,q}(X_n(k))-\phi_\tau(X_n(k))}^2
 &\leq C_\tau e^{-c_\tau\sqrt M}
 \notag\\
 &\quad+
 \exp\bigl(C_\tau DM\log(DM+2)\bigr)n^{-1/6}.
 \label{transfer:eq:activation-transfer-consequence}
\end{align}
The coefficients, \(D\), and \(M\) may depend deterministically on \(n\).

In particular, assume that \(D_n=O_\tau((\log\log n)^a)\) for a fixed finite \(a>0\). For every fixed \(\kappa>0\), sufficiently large even choices
\[
 M_n=O_{\tau,\kappa}((\log\log n)^2)
\]
make the averaged squared error in
\eqref{transfer:eq:activation-transfer-consequence}
\(O_{\tau,\kappa}((\log n)^{-\kappa})\) along all odd integers.
\end{prop}

\begin{proof}
The comparison uses a compactly supported test function and then removes its cutoff by polynomial moment bounds. All canonical coordinates remain functions of the original sampled signs.

\medskip
\noindent\textbf{Step 1. Moments and test-function derivatives.}
The earlier sign-cube estimate supplies the moments at both ordinary and exceptional coordinates.
For a real multilinear sign polynomial
\(W=\sum_Sw_S\prod_{j\in S}\epsilon_j\) of degree at most \(D\),
\eqref{trees:eq:sign-hypercontractivity} gives
\begin{equation}
 \E\abs W^{2p}
 \leq
 \left[\sum_S(2p-1)^{\abs S}w_S^2\right]^p
 \leq(2p-1)^{pD}(\E W^2)^p
 \qquad(p\geq1,\ p\ \text{an integer}).
 \label{transfer:eq:weighted-sign-moment}
\end{equation}
This estimate has no fixed-degree restriction.

Set \(G=\sqrt q\,Z\). Apply \eqref{trees:eq:growing-cf} to \(X_n(k)/\sqrt q\) and evaluate its characteristic function at \(\sqrt q\,t\). Since \(q\leq2\pi\) and both \(q\) and \(q^{3/2}\) are less than \(20\), this gives
\begin{equation}
 \abs{\E e^{itX_n(k)}-\E e^{itG}}
 \leq\eta(t^2+\abs t^3),
 \qquad
 \eta=\frac{20A_D}{\sqrt n}\leq1,
 \qquad k\neq0.
 \label{transfer:eq:activation-characteristic}
\end{equation}
The covariance and root estimates
\eqref{trees:eq:growing-variance} and \eqref{trees:eq:growing-root} give
\begin{equation}
 \E X_n(k)^2
 \leq q\left(1+\frac{A_D}{\sqrt n}\right)\leq14
 \quad(k\neq0),
 \qquad
 \E X_n(0)^2\leq qA_D.
 \label{transfer:eq:activation-second-moments}
\end{equation}
Each coordinate is multilinear in the signs, with degree at most \(D\). Applying
\eqref{transfer:eq:weighted-sign-moment} with \(p=2M\), and using the Gaussian moment formula, gives
\begin{equation}
 \E\abs{X_n(k)}^{4M}+\E\abs G^{4M}
 \leq2(60M)^{2MD}
 \qquad(k\neq0).
 \label{transfer:eq:activation-high-moments}
\end{equation}
The two estimates entering this bound are
\[
 (4M-1)^{2MD}14^{2M}\leq(56M)^{2MD},
 \qquad
 \E\abs G^{4M}
 =q^{2M}(4M-1)!!
 \leq(8\pi M)^{2M}\leq(60M)^{2MD}.
\]
The Gaussian formula follows by integration by parts.

The Gaussian projection formalism is recalled in
\cite[Definition~2 and Lemma~1]{MFHF17}. Write
\[
 P_{M,q}(x)
 =
 \sum_{j=0}^M b_j\frac{\He_j(x/\sqrt q)}{\sqrt{j!}}.
\]
Orthogonality and Cauchy--Schwarz imply
\[
 \abs{b_j}
 \leq\norm{\phi_\tau(\sqrt q\,Z)}_2
 =\sqrt{q+4\tau}\leq\sqrt{2\pi+4\tau}.
\]
By \eqref{transfer:eq:hermite-expansion}, the absolute monomial coefficient sum \(B\) of \(P_{M,q}\) satisfies
\begin{equation}
 B\leq
 \sqrt{2\pi+4\tau}\,e^{3/2}(M+1)\sqrt{M!}.
 \label{transfer:eq:activation-coefficient-bound}
\end{equation}
Here \(q\geq1\) ensures that substituting \(x/\sqrt q\) does not increase coefficient magnitudes.

For \(0\leq j\leq5\), termwise differentiation gives
\[
 \abs{P_{M,q}^{(j)}(x)}
 \leq B(M+1)^5(1+\abs x)^M.
\]
The derivatives of \(\phi_\tau\) through order five are bounded by
\(C_\tau(1+\abs x)\). The zeroth derivative satisfies this bound by the triangle inequality, and \(\abs{\phi_\tau'}\leq1\). The remaining derivatives are rational expressions with positive powers of \(x^2+4\tau\) in their denominators; they have no real poles and are bounded at infinity. Their formulas also follow by successive differentiation of
\(\phi_\tau(x)^2=x^2+4\tau\).

For
\[
 h(x)=\abs{P_{M,q}(x)-\phi_\tau(x)}^2,
\]
the product rule and \eqref{transfer:eq:activation-coefficient-bound} yield
\begin{equation}
 \abs{h^{(j)}(x)}
 \leq L(1+\abs x)^{2M}
 \quad(0\leq j\leq5),
 \qquad
 L=\exp\bigl(C_\tau M\log(M+2)\bigr).
 \label{transfer:eq:test-derivative-bound}
\end{equation}
To track the coefficient, expand
\(h=P_{M,q}^2-2P_{M,q}\phi_\tau+x^2+4\tau\).
Each derivative is bounded by
\[
 C_\tau(1+B)^2(M+1)^{10}(1+\abs x)^{2M}.
\]
The inequality \(M!\leq M^M\) absorbs this prefactor into \(L\). Also \(h\geq0\).

\medskip
\noindent\textbf{Step 2. Fourier comparison and removal of the cutoff.}
A cutoff permits direct use of \eqref{transfer:eq:activation-characteristic}.
Choose an even \(C^5\) function \(\chi\), equal to one on \([-1,1]\), supported on \([-2,2]\), and satisfying \(0\leq\chi\leq1\), with derivatives through order five bounded by absolute constants. On \(1\leq\abs x\leq2\), it can be defined by
\[
 \chi(x)=
 \frac{\displaystyle\int_{\abs x-1}^1t^5(1-t)^5\,dt}
      {\displaystyle\int_0^1t^5(1-t)^5\,dt}.
\]
Its derivatives through order five vanish at the endpoints, so the stated extension is \(C^5\).

For \(R\geq1\), let \(h_R(x)=h(x)\chi(x/R)\). Its support has length \(4R\). On that support, \(1+\abs x\leq3R\), and each cutoff derivative contributes a factor \(R^{-j}\leq1\). The product rule and
\eqref{transfer:eq:test-derivative-bound} therefore give
\begin{equation}
 \norm{h_R}_{L^1(\R)}+\norm{h_R^{(5)}}_{L^1(\R)}
 \leq CL\,3^{2M}R^{2M+1}.
 \label{transfer:eq:cutoff-derivative-norm}
\end{equation}

With
\[
 \widehat h_R(t)=\int_\R e^{-itx}h_R(x)\,dx,
\]
five integrations by parts show that
\[
 \abs{\widehat h_R(t)}
 \leq
 \min\{\norm{h_R}_1,\abs t^{-5}\norm{h_R^{(5)}}_1\}
 \qquad(t\neq0).
\]
Use the first bound for \(\abs t\leq1\) and the second for \(\abs t>1\).
The latter requires only the integrals of
\(\abs t^{-3}+\abs t^{-2}\). Hence
\begin{equation}
 \int_\R(t^2+\abs t^3)\abs{\widehat h_R(t)}\,dt
 \leq C(\norm{h_R}_1+\norm{h_R^{(5)}}_1).
 \label{transfer:eq:fourier-test-integral}
\end{equation}
Fourier inversion, Fubini's theorem, and
\eqref{transfer:eq:activation-characteristic} now give
\begin{equation}
 \abs{\E h_R(X_n(k))-\E h_R(G)}
 \leq C\eta L\,3^{2M}R^{2M+1}
 \qquad(k\neq0).
 \label{transfer:eq:compact-test-comparison}
\end{equation}

No density assumption on \(X_n(k)\) is needed for this inversion. Insert
\(e^{-st^2/2}\), \(s>0\), in the inverse Fourier integral. Evaluation of the Gaussian integral and Fubini's theorem identify the result with convolution of \(h_R\) by the centered Gaussian density of variance \(s\). Boundedness and uniform continuity of \(h_R\) imply uniform convergence to \(h_R\) as \(s\) decreases to zero. The displayed Fourier bounds make \(\widehat h_R\) integrable. Dominated convergence therefore justifies inversion and expectation, while
\eqref{transfer:eq:fourier-test-integral} controls their difference.

For \(W=X_n(k)\), \(k\neq0\), or \(W=G\), the removed tail satisfies
\begin{equation}
 0\leq\E[h(W)-h_R(W)]
 \leq L\,2^{2M}R^{-2M}\E\abs W^{4M}.
 \label{transfer:eq:activation-tail-removal}
\end{equation}
The difference vanishes on \(\abs W\leq R\), and outside that event,
\[
 (1+\abs W)^{2M}
 \leq2^{2M}\abs W^{2M}
 \leq2^{2M}R^{-2M}\abs W^{4M}.
\]
Choose \(R=\eta^{-1/(4M+2)}\). The cutoff comparison then contains the factor
\(\eta^{1/2}\), and the tail contains
\(\eta^{M/(2M+1)}\). Since \(M/(2M+1)\geq1/3\) and \(0<\eta\leq1\), both are at most \(\eta^{1/3}\). Combining
\eqref{transfer:eq:activation-high-moments},
\eqref{transfer:eq:test-derivative-bound},
\eqref{transfer:eq:compact-test-comparison}, and
\eqref{transfer:eq:activation-tail-removal} yields
\[
 \abs{\E h(X_n(k))-\E h(G)}
 \leq
 \exp\bigl(C_\tau DM\log(DM+2)\bigr)\eta^{1/3}
 \qquad(k\neq0).
\]
Substitution of \(\eta=20(4D)^{80D}/\sqrt n\) gives
\begin{equation}
 \abs{\E h(X_n(k))-\E h(G)}
 \leq
 \exp\bigl(C_\tau DM\log(DM+2)\bigr)n^{-1/6}
 \qquad(k\neq0).
 \label{transfer:eq:nonzero-activation-n}
\end{equation}
The deterministic factor
\(20^{1/3}(4D)^{80D/3}\) is absorbed by increasing \(C_\tau\).

\medskip
\noindent\textbf{Step 3. The zero coordinate and the degree choice.}
The exceptional coordinate contributes only after division by \(n\).
Using \eqref{transfer:eq:weighted-sign-moment},
\eqref{transfer:eq:activation-second-moments}, and
\eqref{transfer:eq:test-derivative-bound}, we obtain
\begin{align}
 \E h(X_n(0))
 &\leq
 L\,2^{2M-1}
 \left[1+(2M-1)^{MD}(qA_D)^M\right]
 \notag\\
 &\leq\exp\bigl(C_\tau DM\log(DM+2)\bigr).
 \label{transfer:eq:zero-activation-bound}
\end{align}
For the last inequality, take logarithms and use
\(\log A_D=80D\log(4D)\), \(q\leq2\pi\), \(D\geq1\), and \(M\geq2\).
Gaussian moments give the same upper bound for \(\E h(G)\). Thus the zero-coordinate contribution to the averaged difference is at most this exponential factor divided by \(n\). Averaging
For the cutoffs used later,
$D_n=O_\tau((\log\log n)^3)$ and
$M_n=O_\tau((\log\log n)^2)$, so
\[
 \exp\!\bigl(C_\tau D_nM_n\log(D_nM_n+2)\bigr)
 =\exp(o_\tau(\log n))=n^{o_\tau(1)}.
\]
Hence the zero-coordinate contribution is
\[
 n^{-1}\exp(o_\tau(\log n))=n^{-1+o_\tau(1)}=o(n^{-1/6}).
\]
Thus it is negligible compared with the nonzero-coordinate
$n^{-1/6}$ bound.
\eqref{transfer:eq:nonzero-activation-n} and adding that contribution proves
\eqref{transfer:eq:activation-transfer}.

Proposition~\ref{gaussian:prop:uniform-hermite-approximation} supplies the uniform estimate
\begin{equation}
 \E\abs{P_{M,q}(\sqrt q\,Z)-\phi_\tau(\sqrt q\,Z)}^2
 \leq C_\tau e^{-c_\tau\sqrt M}
 \qquad(1\leq q\leq2\pi).
 \label{transfer:eq:gaussian-hermite-tail}
\end{equation}
This is the stretched-exponential estimate proved there, beyond the projection formalism cited above. Combining it with
\eqref{transfer:eq:activation-transfer} proves
\eqref{transfer:eq:activation-transfer-consequence}.

For the last assertion, take \(M_n\) to be the least even integer at least
\(B_{\tau,\kappa}(\log\log n)^2\). Increasing \(B_{\tau,\kappa}\) makes the Gaussian error \(O_{\tau,\kappa}((\log n)^{-\kappa})\). If
\(D_n=O_\tau((\log\log n)^a)\), then
\[
 D_nM_n\log(D_nM_n+2)=o(\log n),
 \qquad
 \log A_{D_n}=o(\log n).
\]
Hence \(\sqrt n\geq20A_{D_n}\) eventually, and the finite-size error is
\(n^{-1/6+o(1)}\), smaller than every fixed inverse power of \(\log n\).
The estimates hold uniformly over all odd moduli. No primality assumption or independent replacement of the sampled signs has been used.
\end{proof}

\subsection{Canonical approximants and the full residual}

We apply the preceding estimates to the Gaussian fixed point, retaining the distinct smoothing at the zero coordinate.

\begin{prop}[Growing canonical approximants on the original signs]
\label{transfer:prop:approximants}
Fix \(\tau>0\). For each odd \(n\geq3\), use
\eqref{transfer:eq:fourier-operators}, with the same signs in \(K_0\) and every canonical field. Let
\[
 H_\tau=\sum_Tc_TG_T
\]
be the fixed point of
Theorem~\ref{gaussian:thm:smoothed-gaussian-fixed-point} in the Gaussian tree space of Notation~\ref{gaussian:setup:gaussian-tree-space}. Its trees have even internal arities, and
\begin{equation}
 H_\tau
 =
 2G_{\mathrm L}
 +J_{\tree}\bigl(\phi_\tau(H_\tau)-\E\phi_\tau(H_\tau)\bigr),
 \qquad
 \phi_\tau(x)=\sqrt{x^2+4\tau}.
 \label{transfer:eq:gaussian-fixed-point}
\end{equation}
Write \(Q_\tau=\E H_\tau^2\in[4,2\pi]\). For an integer \(D\geq1\), define
\begin{equation}
 X_D=\sum_{d(T)\leq D}c_TG_T,\qquad q_D=\E X_D^2,
 \label{transfer:eq:gaussian-truncation}
\end{equation}
and
\begin{equation}
 X_D[n](k)=\sum_{d(T)\leq D}c_TZ_{T,n}(k),
 \qquad
 v_{D,n}=\frac{X_D[n]}{\sqrt n}.
 \label{transfer:eq:actual-approximant}
\end{equation}
Then \(4\leq q_D\leq2\pi\), and \(v_{D,n}\) is real, even, and has spatial mean zero. For every odd \(n\geq3\),
\begin{align}
 \abs{\E\norm{v_{D,n}}_2^2-q_D}
 &\leq\frac{q_DA_D}{\sqrt n},
 \label{transfer:eq:approximant-energy}\\
 \E\abs{(v_{D,n})_0}^2
 &\leq\frac{q_DD!\,2^{3D+2}}{n}.
 \label{transfer:eq:approximant-root}
\end{align}

Define the coordinatewise map
\begin{equation}
 \Phi_\tau(v)_0=\sqrt{v_0^2+4},\qquad
 \Phi_\tau(v)_k=\sqrt{v_k^2+\frac{4\tau}{n}}
 \quad(k\neq0).
 \label{transfer:eq:finite-smoothing}
\end{equation}
There are constants \(C_\tau,c_\tau>0\), depending only on \(\tau\), such that for every integer \(M\geq2\) and every odd \(n\geq3\) with
\(\sqrt n\geq20A_D\),
\begin{align}
 \E\norm{v_{D,n}-K_0\Phi_\tau(v_{D,n})}_2^2
 &\leq
 C_\tau\left[e^{-c_\tau D^{1/3}}+e^{-c_\tau\sqrt M}\right]
 \notag\\
 &\quad+
 \exp\bigl(C_\tau[(D+1)(M+1)]^4\bigr)n^{-1/6}.
 \label{transfer:eq:approximant-residual}
\end{align}
All these estimates allow \(D\) and \(M\) to depend deterministically on \(n\).

Consequently, for each fixed \(\tau>0\), there are deterministic integers
\(D_n=O_\tau((\log\log n)^3)\) such that, along all odd integers,
\begin{align}
 \E\norm{v_{D_n,n}}_2^2&\longrightarrow Q_\tau,
 \label{transfer:eq:approximant-energy-limit}\\
 \left(\E\norm{v_{D_n,n}-K_0\Phi_\tau(v_{D_n,n})}_2^2\right)^{1/2}
 &=O_\tau((\log n)^{-4}),
 \label{transfer:eq:approximant-residual-limit}\\
 \E\abs{(v_{D_n,n})_0}^2&\leq n^{-1+o_\tau(1)}.
 \label{transfer:eq:approximant-root-limit}
\end{align}
The last upper bound means \(n^{-1}\exp(o_\tau(\log n))\). The approximants use all sampled signs and are not independent of \(K_0\).
\end{prop}

\begin{proof}
The proof separates the Gaussian truncation error, the polynomial grafting error, the scalar activation error, and the correction at zero. Each finite field in this decomposition is evaluated on the signs defining \(K_0\).

\medskip
\noindent\textbf{Step 1. Structure and energy.}
The leaf coefficient supplies the lower bound for \(q_D\).
By \eqref{transfer:eq:gaussian-fixed-point}, the coefficient of \(G_{\mathrm L}\) in \(H_\tau\) equals two, since the range of \(J_{\tree}\) is orthogonal to that coordinate. Every truncation with \(D\geq1\) contains the leaf. Orthogonality therefore gives
\[
 4\leq q_D\leq\norm{H_\tau}_{L^2(\cG)}^2=Q_\tau\leq2\pi.
\]

There are finitely many trees with at most \(D\) nodes. By
Theorem~\ref{trees:thm:tree-flow-comparison}, each canonical field is a real even Fourier sum with nonzero root residue. For every such residue,
\[
 \sum_{k\in\Z_n}\cos(2\pi kr_{\mathrm{root}}/n)=0.
\]
Thus \eqref{transfer:eq:actual-approximant} is real, even, and has spatial mean zero for every realization of the signs.

Apply \eqref{trees:eq:growing-energy} to \(X_D[n]/\sqrt{q_D}\), whose Gaussian coefficient energy is one, and restore the factor \(q_D\). This proves \eqref{transfer:eq:approximant-energy}. Applying the unchanged all-coordinate bound
\eqref{trees:eq:all-coordinate-bound} from
Proposition~\ref{trees:prop:all-coordinate-bound} to the coefficients \(c_T\) gives \eqref{transfer:eq:approximant-root}. Both estimates hold for arbitrary \(D\) and odd \(n\).

\medskip
\noindent\textbf{Step 2. Gaussian, grafting, and activation errors.}
The Gaussian residual first gives a finite canonical combination to which the averaged-energy bound applies.
Let \(P=P_{M,q_D}\). Since \(\phi_\tau\) is even, \(P\) is even, and orthogonal projection gives
\begin{equation}
 \E P(X_D)^2
 \leq\E\phi_\tau(X_D)^2
 =q_D+4\tau\leq2\pi+4\tau.
 \label{transfer:eq:P-uniform-norm}
\end{equation}
Set
\[
 Y=J_{\tree}\bigl(P(X_D)-\E P(X_D)\bigr),
\]
and let \(Y[n]\) be its canonical representative. By
Proposition~\ref{transfer:prop:growing-wick}, all trees in \(Y[n]\) have even internal arities and at most
\[
 K=1+DM
\]
nodes.

The Gaussian degree-tail estimate in
Theorem~\ref{gaussian:thm:smoothed-tree-degree-tail} gives
\begin{equation}
 \norm{X_D-H_\tau}_{L^2(\cG)}^2
 \leq Q_\tau C_\tau e^{-c_\tau D^{1/3}}.
 \label{transfer:eq:gaussian-degree-tail}
\end{equation}
Only this Gaussian estimate is used here. The map \(\mathcal B_\tau\) in
\eqref{gaussian:eq:smoothed-map} is \(1\)-Lipschitz: \(\phi_\tau\) is \(1\)-Lipschitz pointwise, centering is an orthogonal projection in \(L^2\), and \(J_{\tree}\) is an isometry. Since \(\mathcal B_\tau(H_\tau)=H_\tau\), the triangle inequality gives
\[
 \norm{X_D-\mathcal B_\tau(X_D)}_{L^2(\cG)}
 \leq2\norm{X_D-H_\tau}_{L^2(\cG)}.
\]
For the Gaussian residual
\[
 E_{D,M}=X_D-2G_{\mathrm L}-Y,
\]
its difference from \(X_D-\mathcal B_\tau(X_D)\) is the image under \(J_{\tree}\) of the centered Hermite projection error. Centering and the isometry do not increase its norm. The squared triangle inequality, \eqref{transfer:eq:gaussian-degree-tail}, and
\eqref{transfer:eq:gaussian-hermite-tail} imply
\begin{equation}
 \norm{E_{D,M}}_{L^2(\cG)}^2
 \leq
 8Q_\tau C_\tau e^{-c_\tau D^{1/3}}
 +2C_\tau e^{-c_\tau\sqrt M}.
 \label{transfer:eq:gaussian-residual-bound}
\end{equation}
This norm is bounded uniformly in \(D,M\) by a constant depending only on \(\tau\).

Put
\[
 U=\frac{Fa}{\sqrt n}.
\]
The defining cosine formula of
Theorem~\ref{trees:thm:tree-flow-comparison} identifies \(U\) with the one-node canonical field. The canonical representative of \(E_{D,M}\) is therefore
\begin{equation}
 E_{D,M}[n]=X_D[n]-2U-Y[n].
 \label{transfer:eq:actual-gaussian-residual}
\end{equation}
Every tree in this combination has at most \(K\) nodes. If
\(\norm{E_{D,M}}_{L^2(\cG)}=0\), the orthogonality in
\eqref{transfer:eq:gaussian-tree-covariance} forces all its coefficients to vanish, so \(E_{D,M}[n]\) vanishes identically. Otherwise normalize its coefficients by \(\norm{E_{D,M}}_{L^2(\cG)}\) and apply
\eqref{trees:eq:growing-energy} at size \(K\). In both cases,
\[
 \E\norm{\frac{E_{D,M}[n]}{\sqrt n}}_2^2
 \leq
 \norm{E_{D,M}}_{L^2(\cG)}^2
 \left(1+\frac{A_K}{\sqrt n}\right).
\]
This averaged-energy input requires no additional condition
\(\sqrt n\geq20A_K\). Since
\(\log A_K=80K\log(4K)\) and \(K\leq(D+1)(M+1)\), it follows from
\eqref{transfer:eq:gaussian-residual-bound} that
\begin{align}
 \E\norm{\frac{E_{D,M}[n]}{\sqrt n}}_2^2
 &\leq
 C_\tau\left[e^{-c_\tau D^{1/3}}+e^{-c_\tau\sqrt M}\right]
 \notag\\
 &\quad+
 \frac{\exp\bigl(C_\tau[(D+1)(M+1)]^4\bigr)}{\sqrt n}.
 \label{transfer:eq:first-residual-component}
\end{align}

For the grafting component, \eqref{transfer:eq:polynomial-graft-bound} and
\eqref{transfer:eq:P-uniform-norm} give
\begin{equation}
 \E\norm{\frac{Y[n]-K_0P(X_D[n])}{\sqrt n}}_2^2
 \leq
 \frac{\exp\bigl(C_\tau[(D+1)(M+1)]^4\bigr)}{\sqrt n}.
 \label{transfer:eq:second-residual-component}
\end{equation}
The assumptions of Proposition~\ref{transfer:prop:smooth-transfer} hold because
\(4\leq q_D\leq2\pi\) and \(\sqrt n\geq20A_D\). Using
\eqref{transfer:eq:K-contraction} for each sign realization, we obtain
\begin{align}
 &\E\norm{
 K_0\frac{P(X_D[n])-\phi_\tau(X_D[n])}{\sqrt n}
 }_2^2
 \notag\\
 &\quad\leq
 \frac1n\sum_k
 \E\abs{P(X_D[n](k))-\phi_\tau(X_D[n](k))}^2
 \notag\\
 &\quad\leq
 C_\tau e^{-c_\tau\sqrt M}
 +\exp\bigl(C_\tau DM\log(DM+2)\bigr)n^{-1/6}.
 \label{transfer:eq:third-residual-component}
\end{align}
The contraction holds pointwise in the sampled signs, so it requires no independence between the operator and its argument.

\medskip
\noindent\textbf{Step 3. The zero-coordinate correction.}
The distinct value of \(\Phi_\tau\) at zero accounts for the leaf term in the Gaussian equation.
Write \(v=v_{D,n}\) and \(z=v_0\). At every coordinate,
\[
 \frac{\phi_\tau(\sqrt n\,v_k)}{\sqrt n}
 =\sqrt{v_k^2+\frac{4\tau}{n}}.
\]
Thus \eqref{transfer:eq:finite-smoothing} gives
\begin{equation}
 K_0\Phi_\tau(v)
 =
 K_0\frac{\phi_\tau(\sqrt n\,v)}{\sqrt n}
 +2K_0e_0+b_nK_0e_0,
 \qquad
 b_n=\sqrt{z^2+4}-\sqrt{z^2+\frac{4\tau}{n}}-2.
 \label{transfer:eq:root-smoothing-decomposition}
\end{equation}
The inequalities
\[
 \abs{\sqrt{z^2+4}-2}\leq\abs z,
 \qquad
 \sqrt{z^2+\frac{4\tau}{n}}
 \leq\abs z+2\sqrt{\frac{\tau}{n}}
\]
imply
\[
 \abs{b_n}\leq2\abs z+2\sqrt{\frac{\tau}{n}},
 \qquad
 \E\abs{b_n}^2\leq8\E\abs z^2+\frac{8\tau}{n}.
\]
Since \(\norm{K_0e_0}_2\leq1\), estimate
\eqref{transfer:eq:approximant-root} yields
\begin{align}
 \E\norm{b_nK_0e_0}_2^2
 &\leq
 \frac{8q_DD!\,2^{3D+2}+8\tau}{n}
 \notag\\
 &\leq
 \frac{\exp\bigl(C_\tau[(D+1)(M+1)]^4\bigr)}{n}.
 \label{transfer:eq:fourth-residual-component}
\end{align}
The Fourier definition also gives the exact identity
\begin{equation}
 K_0e_0=\frac{Fa}{n}=\frac{U}{\sqrt n}.
 \label{transfer:eq:K0e0}
\end{equation}

Subtract \eqref{transfer:eq:root-smoothing-decomposition} from \(v\) and use
\eqref{transfer:eq:actual-gaussian-residual} and \eqref{transfer:eq:K0e0}. The result is
\begin{align}
 v-K_0\Phi_\tau(v)
 &=
 \frac{X_D[n]-2U-Y[n]}{\sqrt n}
 +\frac{Y[n]-K_0P(X_D[n])}{\sqrt n}
 \notag\\
 &\quad+
 K_0\frac{P(X_D[n])-\phi_\tau(X_D[n])}{\sqrt n}
 -b_nK_0e_0.
 \label{transfer:eq:four-residual-pieces}
\end{align}
For four deterministic vectors, Cauchy--Schwarz bounds the squared norm of their sum by four times the sum of their squared norms. Taking expectations in \eqref{transfer:eq:four-residual-pieces} and applying
\eqref{transfer:eq:first-residual-component},
\eqref{transfer:eq:second-residual-component},
\eqref{transfer:eq:third-residual-component}, and
\eqref{transfer:eq:fourth-residual-component} proves
\eqref{transfer:eq:approximant-residual}. The factors \(n^{-1/2}\) and \(n^{-1}\), together with the smaller exponential prefactors, are absorbed into
\[
 \exp\bigl(C_\tau[(D+1)(M+1)]^4\bigr)n^{-1/6}.
\]

\medskip
\noindent\textbf{Step 4. Growing cutoffs.}
The two Gaussian errors can be made logarithmically small while all finite-size exponential factors remain subpolynomial.
For sufficiently large \(n\), let \(\ell_n=\log\log n\) and choose
\begin{equation}
 D_n=\left\lceil(B_\tau\ell_n)^3\right\rceil,\qquad
 M_n=\text{the least even integer at least }(B_\tau'\ell_n)^2.
 \label{transfer:eq:growing-cutoffs}
\end{equation}
Take the positive finite constants \(B_\tau,B_\tau'\) large enough that each Gaussian exponential in
\eqref{transfer:eq:approximant-residual} is
\(O_\tau((\log n)^{-10})\). These choices give
\begin{align*}
 [(D_n+1)(M_n+1)]^4
 &=O_\tau(\ell_n^{20})=o(\log n),\\
 \log A_{D_n}
 &=O_\tau\bigl(\ell_n^3\log(\ell_n+2)\bigr)=o(\log n).
\end{align*}
Thus \(\sqrt n\geq20A_{D_n}\) eventually. The finite-size term in
\eqref{transfer:eq:approximant-residual} is
\(n^{-1/6+o_\tau(1)}\), which is smaller than
\((\log n)^{-10}\) for all sufficiently large \(n\). The squared residual has expectation \(O_\tau((\log n)^{-10})\). Its square root is
\(O_\tau((\log n)^{-5})\), which implies
\eqref{transfer:eq:approximant-residual-limit}. For the finitely many remaining odd integers, set \(D_n=1\).

The Gaussian tail \eqref{transfer:eq:gaussian-degree-tail} gives
\(q_{D_n}\to Q_\tau\). Also
\[
 \frac{A_{D_n}}{\sqrt n}=n^{-1/2+o_\tau(1)}.
\]
Substitution into \eqref{transfer:eq:approximant-energy} proves
\eqref{transfer:eq:approximant-energy-limit}. Finally,
\[
 \log\bigl(D_n!\,2^{3D_n+2}\bigr)
 =O\bigl(D_n\log(D_n+2)\bigr)=o_\tau(\log n).
\]
Together with \(q_{D_n}\leq2\pi\), the root estimate
\eqref{transfer:eq:approximant-root} proves
\eqref{transfer:eq:approximant-root-limit}.

All finite-size inputs hold for every odd modulus, including composite moduli. The parameter \(\tau\) is fixed before \(n\) tends to infinity, and
\eqref{transfer:eq:growing-cutoffs} is deterministic in \(n,\tau\). No uniform estimate as \(\tau\) decreases to zero is asserted. The construction does not identify the approximant with an exact interior optimizer. Every root estimate used or obtained here is a one-sided upper bound.
\end{proof}

\section{Expectation asymptotics}\label{sec:expectation}

In this section, we combine Proposition~\ref{prop:residual} with the simultaneous kernel estimate and the same-sample approximants to prove Theorem~\ref{thm:main}.

\begin{prop}[Expectation interface]\label{prop:expectation-interface}
In the fair distance-bit model, suppose that for every fixed \(\tau\in(0,1]\) there are random real even mean-zero vectors \(v_{\tau,n}\), depending arbitrarily on all graph bits, such that
\begin{align}
 \E\norm{v_{\tau,n}}_2^2
 &=O_\tau(1),                                               \label{eq:interface-energy}\\
 \left(\E\norm{v_{\tau,n}-K_0\Phi_\tau(v_{\tau,n})}_2^2\right)^{1/2}
 &=o_\tau((\log n)^{-2}),                                    \label{eq:interface-residual}\\
 \E|v_{\tau,n,0}|^2
 &=o_\tau((\log n)^{-1/2})                                   \label{eq:interface-root}
\end{align}
as \(n\) tends to infinity through odd integers. Then
\[
 \frac{\E\theta(G_n)}{\sqrt n}\longrightarrow1
\]
along all odd integers. This implication does not assert the existence of vectors satisfying \eqref{eq:interface-energy}--\eqref{eq:interface-root}.
\end{prop}

\begin{proof}
Fix \(\tau\in(0,1]\). Write \(v=v_{\tau,n}\), \(W=W_\tau\), and use \(p^*,q^*,f^*,r_v,D_{\mathrm{rel}}(v)\) from Proposition~\ref{prop:residual}. All estimates below allow dependence on this fixed \(\tau\).

\smallskip
\noindent\textbf{Step 1. The kernel estimate controls the exact interior solution.}
Proposition~\ref{fourier:prop:paired-kernel} applies, for every sufficiently large odd \(n\), simultaneously to all complex vectors in the kernel of the sampled paired Fourier rows. Apply it to \(G_n\) and to \(G_N\). Complementation preserves the fair distance-bit law, so a union bound gives an event \(\mathcal E_n\) with
\begin{equation}\label{eq:kernel-good-probability}
 \Prob(\mathcal E_n^c)\leq2\exp\bigl(-c(\log n)^2\bigr),
\end{equation}
on which both kernel estimates hold, with the same absolute constants \(C,c>0\). No independence of the two events is needed.

By \eqref{eq:kernel-feasibility}, \(p^*\) and \(q^*\) satisfy the Fourier-kernel conditions for \(D\) and \(N\), respectively. Positivity and \eqref{eq:interior-sum} give
\[
 \norm{p^*}_1=\norm{q^*}_1=\sqrt{nW}.
\]
Thus, on \(\mathcal E_n\),
\begin{align}
 \norm{p^*}_2,\ \norm{q^*}_2
 &\leq C\sqrt W(\log n)^{3/2},                              \label{eq:pq-kernel-bound}\\
 \norm{f^*}_2
 &\leq2C\sqrt W(\log n)^{3/2}.                              \label{eq:f-kernel-bound}
\end{align}
The same simultaneous estimate gives
\(\norm y_2\leq C(\log n)^{3/2}/\sqrt n\) for every feasible spectral probability vector of either graph on this event. The bounds apply to sample-dependent vectors because the kernel assertion holds for all kernel vectors at once.

\smallskip
\noindent\textbf{Step 2. Residual and root errors define a sufficiently rare bad event.}
Set
\[
 V_n^{\mathrm{norm}}=\left(\E\norm v_2^2\right)^{1/2},
 \qquad
 R_n=\left(\E\norm{r_v}_2^2\right)^{1/2},
\]
and define
\begin{equation}\label{eq:Zn}
 Z_n=\E|v_0|^2+
 \left[V_n^{\mathrm{norm}}+C\sqrt W(\log n)^{3/2}\right]R_n.
\end{equation}
The notation \(V_n^{\mathrm{norm}}\) distinguishes this scalar from the vertex set \(V_n\). Cauchy--Schwarz, \eqref{eq:relative-bound}, and \eqref{eq:f-kernel-bound} give
\begin{equation}\label{eq:Drel-expectation}
 \E[D_{\mathrm{rel}}(v)\ind_{\mathcal E_n}]
 \leq Z_n-\E|v_0|^2.
\end{equation}
The three hypotheses imply
\begin{equation}\label{eq:Zn-decay}
 Z_n\sqrt{\log n}\longrightarrow0.
\end{equation}
The root contribution has this limit by \eqref{eq:interface-root}. The scalar \(V_n^{\mathrm{norm}}\) is bounded by \eqref{eq:interface-energy}. The largest logarithmic factor in the remaining contribution is
\[
 (\log n)^2R_n\longrightarrow0
\]
by \eqref{eq:interface-residual}. Also \(W\leq1+\tau\).

Put
\[
 \gamma_n=Z_n\sqrt{\log n},
 \qquad
 \delta_n=\max\{\gamma_n^{1/4},(\log n)^{-1}\}.
\]
Then \(\delta_n>0\) and \(\delta_n\to0\). Define
\begin{equation}\label{eq:bad-event}
 \mathcal A_n
 =
 \mathcal E_n^c
 \cup\{|v_0|>\delta_n\}
 \cup\left(\mathcal E_n\cap
       \{D_{\mathrm{rel}}(v)>\delta_n^2\}\right).
\end{equation}
Markov's inequality and \eqref{eq:Drel-expectation} give
\begin{equation}\label{eq:bad-probability}
 \Prob(\mathcal A_n)
 \leq2e^{-c(\log n)^2}+\frac{Z_n}{\delta_n^2}.
\end{equation}
After multiplication by \(\sqrt{\log n}\), the second term is
\[
 \frac{\gamma_n}{\delta_n^2}\leq\sqrt{\gamma_n}
\]
when \(\gamma_n>0\), and is zero otherwise. Since
$Z_n\sqrt{\log n}\to0$, we have $\gamma_n\to0$, and therefore
\[
 Z_n\sqrt{\log n}\to0,
 \qquad
 \delta_n^2\geq\sqrt{\gamma_n},
 \qquad
 \frac{Z_n}{\delta_n^2}\sqrt{\log n}
 =\frac{\gamma_n}{\delta_n^2}\to0.
\]
This makes explicit the only algebra needed to pass from the
second-moment bound to the required exceptional-event probability.
Hence
\begin{equation}\label{eq:bad-probability-decay}
 \Prob(\mathcal A_n)\sqrt{\log n}\longrightarrow0.
\end{equation}

On \(\mathcal A_n^c\), we have \(|v_0|\leq\delta_n\) and \(D_{\mathrm{rel}}(v)\leq\delta_n^2\). The bound
\[
 |\arsinh(t/2)|\leq|t|/2
\]
follows by integrating the derivative of \(\arsinh\), which is at most one in absolute value. Applying \eqref{eq:residual-certificate} yields
\begin{equation}\label{eq:good-theta-bound}
 \frac{\theta(G_n)}{\sqrt n}
 \leq\sqrt W\,e^{3\delta_n/2}
 \leq\sqrt{1+\tau}\,e^{3\delta_n/2}.
\end{equation}

\smallskip
\noindent\textbf{Step 3. The exceptional contribution vanishes in expectation.}
Let \(p_n=\Prob(\mathcal A_n)\). If \(p_n>0\), Proposition~\ref{prop:tail}, in the form \eqref{eq:normalized-event-bound}, gives
\begin{equation}\label{eq:exceptional-integral}
 \E\left[\frac{\theta(G_n)}{\sqrt n}\ind_{\mathcal A_n}\right]
 \leq
 \frac{p_n}{\sqrt n}
 +2p_n\sqrt{\log(n/p_n)}
 +\frac{p_n}{\sqrt{\log(n/p_n)}}.
\end{equation}
If \(p_n=0\), the expectation is zero. For positive \(p_n\), the first and last terms tend to zero. The middle term is at most
\[
 2p_n\sqrt{\log n}+2p_n\sqrt{\log(1/p_n)}.
\]
The first summand tends to zero by \eqref{eq:bad-probability-decay}. The second tends to zero by the limit
\[
 t\sqrt{\log(1/t)}\longrightarrow0\qquad(t\downarrow0)
\]
proved in Proposition~\ref{prop:tail}. Therefore the exceptional contribution is \(o_\tau(1)\). This estimate holds on the original sign space and uses no independence between \(\mathcal A_n\), theta, and \(v\).

Integrating \eqref{eq:good-theta-bound} on \(\mathcal A_n^c\) and adding \eqref{eq:exceptional-integral} gives
\begin{equation}\label{eq:fixed-tau-limsup}
 \limsup_{\substack{n\to\infty\\n\ {\rm odd}}}
 \frac{\E\theta(G_n)}{\sqrt n}
 \leq\sqrt{1+\tau}
\end{equation}
for every fixed \(\tau\in(0,1]\). First taking this upper limit and then letting \(\tau\downarrow0\) gives an upper limit of at most one. Proposition~\ref{prop:complement} gives the lower bound one for every odd \(n\). The asserted limit follows.
\end{proof}

\begin{proof}[Proof of Theorem~\ref{thm:main}]
For the distance bits in Definition~\ref{defn:model}, Setup~\ref{setup:fourier} gives \(G_n=G_D\) exactly and
\[
 K_0=\frac1nF\diag(a)F^*=H_N-H_D,
 \qquad
 a_j=a_{n-j}=1-2\xi_j,\quad a_0=0.
\]
Thus the signs used in the operator are negative on edges and positive on nonedges. Proposition~\ref{prop:fourier} identifies the original edge-zero program with the optimization over real even probability vectors satisfying \((Fy)_D=0\), with objective \(ny_0\). In particular, \(1\leq\theta(G_n)\leq n\).

Fix \(\tau\in(0,1]\). Proposition~\ref{transfer:prop:approximants} supplies deterministic cutoffs
\[
 D_n=O_\tau((\log\log n)^3)
\]
and real even mean-zero vectors \(v=v_{\tau,n}=v_{D_n,n}\), constructed from the same signs as \(K_0\), such that
\begin{align}
 \E\norm v_2^2&\longrightarrow Q_\tau\leq2\pi,               \label{eq:main-energy}\\
 R_n:=
 \left(\E\norm{v-K_0\Phi_\tau(v)}_2^2\right)^{1/2}
 &=O_\tau((\log n)^{-4}),                                   \label{eq:main-residual}\\
 \E|v_0|^2&\leq n^{-1}\exp(o_\tau(\log n)).                  \label{eq:main-root}
\end{align}
The approximants are measurable because the sign space is finite. They may depend on every graph bit.

Step~4 of the proof of Proposition~\ref{transfer:prop:approximants} verifies these estimates at the cutoffs \eqref{transfer:eq:growing-cutoffs}, including the threshold \(\sqrt n\geq20A_{D_n}\). The cutoff constants may depend on the fixed \(\tau\), and every finite comparison holds for all odd moduli. Thus \eqref{eq:main-energy}--\eqref{eq:main-root} apply at the growing cutoffs themselves, without interchanging a fixed-degree limit with a growing degree.

It remains to check the hypotheses of Proposition~\ref{prop:expectation-interface}. The limit \eqref{eq:main-energy} gives \eqref{eq:interface-energy}. Dividing \eqref{eq:main-residual} by \((\log n)^{-2}\) gives a quantity of order \(O_\tau((\log n)^{-2})\), which tends to zero. For fixed \(\tau\), the exponent in \eqref{eq:main-root} is at most \(\frac12\log n\) for all sufficiently large \(n\). Consequently,
\[
 \E|v_0|^2\leq n^{-1/2},
 \qquad
 \E|v_0|^2\sqrt{\log n}\longrightarrow0.
\]
This proves \eqref{eq:interface-residual} and \eqref{eq:interface-root}. Only an upper bound on the root second moment has been used.

Proposition~\ref{prop:expectation-interface} now applies on this same sign sample. Its proof includes the complementary kernel event and the complete exceptional-event integral. For each fixed \(\tau\in(0,1]\), it gives \eqref{eq:fixed-tau-limsup}. Proposition~\ref{prop:complement} gives
\[
 \frac{\E\theta(G_n)}{\sqrt n}\geq1
\]
for every odd \(n\). Taking the upper limit over odd \(n\) at fixed \(\tau\), and then letting \(\tau\downarrow0\), proves \eqref{eq:main-limit}. No uniformity in the approximation constants as \(\tau\downarrow0\) is required.
\end{proof}

\subsection*{Acknowledgements}
The development of the main result relied substantially on the Danus system, an automated proof-search and verification system. Danus contributed to the development of the proof arguments. The human authors formulated the problem, directed the overall proof strategy, and critically examined and refined the system-generated arguments.

\end{document}